\documentclass[final,hidelinks,onefignum,onetabnum]{siamart220329}

\usepackage{lipsum}
\usepackage{amsfonts}
\usepackage{graphicx}
\usepackage{epstopdf}
\usepackage{algorithmic}
\ifpdf
\DeclareGraphicsExtensions{.eps,.pdf,.png,.jpg}
\else
\DeclareGraphicsExtensions{.eps}
\fi

\usepackage{booktabs}
\usepackage{hyperref}
\usepackage{stmaryrd}
\usepackage{bm}
\usepackage{caption}
\usepackage{subcaption}
\usepackage{multirow}
\usepackage{enumitem}

\usepackage{amsopn}

\usepackage[left=2.48cm,right=5.96cm,top=3.2cm,bottom=3.3cm,asymmetric]{geometry}

\newsiamthm{test}{Test}
\newsiamthm{assumption}{Assumption}

\def\jump#1{\llbracket #1 \rrbracket }

\newsiamremark{remark}{Remark}
\newsiamremark{hypothesis}{Hypothesis}
\crefname{hypothesis}{Hypothesis}{Hypotheses}
\newsiamthm{claim}{Claim}
\newsiamremark{exmp}{Example}

\headers{}{}

\title{Oscillation-eliminating central DG schemes for hyperbolic conservation laws
	\thanks{
		This work was partially supported by Shenzhen Science and Technology Program (No.~RCJC20221008092757098) and 
		National Natural Science Foundation of China (No.~12171227).}}

\author{Manting Peng \thanks{Department of Mathematics, Southern University of Science and Technology, Shenzhen, Guangdong 518055, China.} 
	\and Kailiang Wu \thanks{Corresponding author. Department of Mathematics and Shenzhen International Center for Mathematics, Southern University of Science and Technology, Shenzhen, Guangdong 518055, China (\email{wukl@sustech.edu.cn}).}
 \and Caiyou Yuan\thanks{Department of Mathematics, Southern University of Science and Technology, Shenzhen, Guangdong 518055, China.} 
 } 

\ifpdf
\hypersetup{
	pdftitle={},
	pdfauthor={}
}
\fi

\usepackage{comment}

\begin{document}
	
	\maketitle
	
	\begin{abstract}This paper proposes and analyzes a class of essentially non-oscillatory central discontinuous Galerkin (CDG) methods for general hyperbolic conservation laws. 
First, we introduce a novel compact, non-oscillatory stabilization mechanism that effectively suppresses spurious oscillations while preserving the high-order accuracy of CDG methods. Unlike existing limiter-based approaches that rely on large stencils or problem-specific parameters for oscillation control, our \emph{dual damping mechanism} is inspired by CDG-based numerical dissipation and leverages overlapping solutions within the CDG framework, significantly enhancing stability while maintaining compactness. Our approach is {free of problem-dependent parameters and complex characteristic decomposition}, making it both efficient and robust. 
Second, we provide---for the first time---a rigorous stability and optimal error analysis for fully discrete Runge--Kutta (RK) CDG schemes, addressing a gap in the theoretical understanding of these methods. Specifically, we establish the approximate skew-symmetry and weak boundedness of the CDG discretization, which form the foundation for proving the linear stability of RK CDG methods via the matrix transferring technique. These results enable us to rigorously analyze the \emph{fully discrete error estimates} for our oscillation-eliminating CDG (OECDG) method, a challenging task due to its nonlinear nature, even for linear advection equations. 
Building on this framework, we reformulate nonlinear oscillation-eliminating CDG schemes as linear RK CDG schemes with a nonlinear source term, extending error estimates beyond the linear case to schemes with nonlinear oscillation control. While existing error analyses for DG or CDG schemes have largely been restricted to linear cases without nonlinear oscillation-control techniques, our analysis represents an important theoretical advancement. 
Extensive numerical experiments validate the theoretical findings and demonstrate the effectiveness of the OECDG method across a wide range of hyperbolic conservation laws.

	\end{abstract}
	
	\begin{keywords}
		Hyperbolic conservation laws, central discontinuous Galerkin method, oscillation control, fully discrete, error analysis, stability
	\end{keywords}
	
	\begin{MSCcodes}
		65M60, 65M12, 35L65
	\end{MSCcodes}
	
    \section{Introduction}

The central Discontinuous Galerkin (CDG) methods, first introduced by Liu et al.~in \cite{liu2007central}, are a class of finite element techniques that combine the strengths of central schemes \cite{nessyahu1990non,kurganov2000new,liu2004central} and discontinuous Galerkin (DG) methods \cite{cockburn1989tvb,cockburn1998runge,cockburn1990runge,cockburn2001runge}. These methods preserve many of the advantageous characteristics of both approaches. For instance, CDG methods, like traditional DG methods, utilize discontinuous piecewise polynomial spaces, enabling great flexibility and supporting both $h$- and $p$-adaptive strategies. As a category of central schemes, CDG methods eliminate the need for Riemann solvers, thus providing a versatile, black-box solution for hyperbolic conservation laws. 
In recent years, CDG methods have garnered considerable interest and have been successfully applied to a wide array of hyperbolic conservation laws  \cite{zhao2017runge,zhao2017runge2,li2011central,li2012arbitrary,yakovlev2013locally}, as well as to other types of partial differential equations \cite{li2010central}.

Designing effective numerical schemes for hyperbolic conservation laws poses a key challenge: handling discontinuities and sharp transitions without inducing nonphysical oscillations. Inadequately designed schemes often produce spurious oscillations near these features, leading to violations of physical constraints, numerical instability, or even simulation crashes. Thus, there is a critical need to develop numerical methods capable of controlling these spurious oscillations effectively.

There are two primary strategies for mitigating oscillations in numerical schemes. 
The first one is to apply suitable limiters,
which modify the numerical solution within troubled cells 
identified by specific indicators. 
Examples of these limiters include but are not limited to
total variation diminishing limiter, 
total variation bounded limiter, 
and weighted essentially nonoscillatory type limiters, 
as discussed in \cite{qiu2005runge,cockburn1989tvb,zhong2013simple,chavent1989local}.
The second strategy involves introducing artificial viscosity to diffuse oscillations, as detailed in  \cite{hiltebrand2014entropy,zingan2013implementation,huang2020adaptive,yu2020study}. While both approaches can be effective, they often rely on problem-specific parameters or require a large stencil for reconstructing the modified solutions, which may limit their general applicability, compactness, and computational efficiency.

Recently, Lu, Liu, and Shu \cite{lu2021oscillation} introduced the oscillation-free DG (OFDG) method, which incorporates innovative damping terms to suppress spurious oscillations. Using predefined damping coefficients, OFDG methods have demonstrated applicability across a wide range of problems, as shown in \cite{liu2022essentially, liu2022oscillation, tao2023oscillation, du2023oscillation}. However, the damping terms are not scale-invariant and can become highly stiff when solutions develop strong shocks, restricting the allowable time step size for explicit time discretization. 
Motivated by the OFDG methods, 
the authors of \cite{peng2023oedg} 
proposed oscillation-eliminating (OE) DG methods, referred to as OEDG, 
which introduce a novel approach 
by evolving a new scale-invariant damping ordinary differential equation (ODE) 
after each Runge--Kutta (RK) stage. 
This damping ODE is linear and can be solved exactly without the need for additional discretization. It 
acts as a modal filter that adjusts the high-order modal coefficients 
to effectively control spurious oscillations. 
This innovative OEDG method not only preserves conservation and optimal convergence rates,
but also maintains stability under standard Courant--Friedrichs--Lewy (CFL) conditions,
even in the presence of strong shocks. 
Moreover, the OE technique can be seamlessly integrated 
into existing DG codes in a non-intrusive manner, 
delivering essentially non-oscillatory, high-resolution numerical results without requiring characteristic decomposition. 
Due to their effectiveness, compactness, and robustness, the OEDG methods have been extended to unstructured meshes \cite{DingCuiWu2024}, adapted for the design of OE Hermite WENO schemes \cite{FanWu2024}, and successfully applied to various hyperbolic models across diverse applications, including two-phase flows \cite{YanAbgrallWu}, magnetohydrodynamics \cite{LiuWu2024}, and general relativistic hydrodynamics \cite{CaoPengWu}.

Building on the success of OE approach and the widespread adoption of CDG schemes, this paper proposes and analyzes a novel, efficient CDG scheme termed the OECDG method. This new approach integrates a specially designed OE procedure with an innovative dual damping mechanism within the CDG framework. The key innovations and contributions of this work are outlined as follows:

\begin{itemize}[leftmargin=*] 

\item We introduce robust CDG methods that incorporate a newly designed OE procedure. Unlike the OE procedure in DG methods \cite{peng2023oedg}, which relies on jump information from two adjacent cells, the proposed OE procedure draws inspiration from the numerical dissipation term of CDG methods and is based on a new dual damping mechanism that utilizes overlapping solutions within the CDG framework. This approach enhances stencil compactness and improves numerical resolution, providing an effective new method for oscillation control.

\item 
We rigorously establish optimal error estimates for the fully discrete OECDG methods, addressing a gap in the error analysis of fully discrete CDG schemes, including the original linear CDG methods without oscillation control. 
Achieving this involves several novel theoretical contributions. We prove the approximate skew-symmetry and weak boundedness of the CDG spatial discretization for the first time, a novel contribution in itself. Building on this foundation, we first establish the linear stability of CDG methods with Runge--Kutta (RK) discretization by employing the matrix transferring technique \cite{xu20192}. Finally, leveraging the linear stability of CDG methods, we rigorously derive the fully discrete optimal error estimates for the OECDG methods. Notably, since the OECDG methods are nonlinear schemes even for linear advection equations, deriving these optimal error estimates poses much more challenges compared to standard CDG methods.

\item 
We present extensive numerical experiments covering a variety of hyperbolic equations, including the convection equation, Burgers equation, traffic flow model, non-convex conservation law, and Euler equations. The numerical results validate the theoretical optimal error estimates and demonstrate the effectiveness of the OECDG methods in eliminating spurious oscillations, confirming their robustness and accuracy across a diverse range of test cases.

\end{itemize}

The rest of this paper is organized as follows: 
\Cref{sec:CDG} 
presents the proposed OECDG methods.  
\Cref{sec:Optimal}
establishes the linear stability of the
CDG methods with Runge--Kutta discretization,
and proves
the optimal error estimates of the fully discrete OECDG methods
for linear advection equation.
Various numerical tests 
are conducted in \Cref{sec:numerical}
to verify the optimal error estimates and
show the effectiveness of OECDG methods. 
Concluding remarks are provided in \Cref{sec:conclusions}.  

    \section{OECDG method}\label{sec:CDG}

This section presents the OECDG method for the scalar conservation law:
\begin{equation}\label{equ:conservation_law}
    u_t + \nabla \cdot \boldsymbol{f}(u) = 0,
\end{equation}
where $u = u(\bm{x}, t)$ and $\boldsymbol{f}(u)$ is the flux function. The extension to hyperbolic systems is straightforward and can be implemented component-wisely. 
We begin by revisiting the semi-discrete CDG method and then introduce the fully discrete OECDG method.

\subsection{Semi-discrete CDG formulation}

Let $\mathcal{T}_h^C$ be a partition of the computational domain $\Omega$ into cells (the primal mesh). For each cell $K \in \mathcal{T}_h^C$, select its center and connect it to the centers of its neighboring cells, generating the dual mesh $\mathcal{T}_h^D$. 
Define the finite element spaces
\begin{equation}\nonumber
\begin{aligned}
    V_h & :=\left\{\phi_1 \in L^2(\Omega):\left.\phi_1 \right|_{K} \in \mathbb{P}^k(K)~\forall K\in \mathcal{T}_h^C\right\},\\
    W_h & :=\left\{\phi_2 \in L^2(\Omega):\left.\phi_2 \right|_{K} \in \mathbb{P}^k(K)~\forall K\in \mathcal{T}_h^D\right\},
\end{aligned}
\end{equation}
where $\mathbb{P}^k(K)$ denotes the space of polynomials of degree at most $k$ on cell $K$. 
The semi-discrete CDG method seeks numerical solutions $(u_h^C, u_h^D) \in V_h \times W_h$ such that, for all test functions $(\phi_1, \phi_2) \in V_h \times W_h$ and all cells $C \in \mathcal{T}_h^C$, $D \in \mathcal{T}_h^D$, the following equations hold:
\begin{equation}
    \label{equcdg}
    \begin{aligned}
        \int_{C} \partial_t u_h^C\phi_1  \mathrm{d} \bm{x} 
        &= \frac{1}{\tau_{\rm max}}\int_{C} \left(u_h^D-u_h^C\right)\phi_1  \mathrm{d} \bm{x}
        +\int_{C}\boldsymbol{f}(u_h^D)\cdot\nabla\phi_1 \mathrm{d} \bm{x}
        -\sum_{e\in \partial C}\int_{e} \boldsymbol{f}(u_h^D) \cdot \bm{n}_e\,\phi_1 \mathrm{d}s,\\
        \int_{D} \partial_t u_h^D\phi_2  \mathrm{d} \bm{x} 
        &= \frac{1}{\tau_{\rm max}}\int_{D} \left(u_h^C-u_h^D\right)\phi_2 \mathrm{d} \bm{x}
        +\int_{D}\boldsymbol{f}(u_h^C) \cdot\nabla\phi_2\mathrm{d} \bm{x}
        -\sum_{e\in \partial D}\int_{e} \boldsymbol{f}(u_h^C)\cdot \bm{n}_e\,\phi_2 \mathrm{d}s,
    \end{aligned}
\end{equation}
where $\tau_{\text{max}}$ denotes the maximum stable time step size, 
and $\bm{n}_e$ is the outward unit normal vector
at $e$ with respect to the cell $C$ or $D$. 
The semi-discrete CDG method can be rewritten as follows: 
find $\boldsymbol{u}_h := (u_h^C,u_h^D) \in V_h\times W_h$ such that
\begin{equation*}
\langle \partial_t\boldsymbol{u}_h, \boldsymbol{\phi}_h \rangle 
= \mathcal{H}(\boldsymbol{u}_h, \boldsymbol{\phi}_h)\quad \forall \bm{\phi}_h \in V_h\times W_h,
\end{equation*}
where the inner product $\langle \cdot, \cdot \rangle$ is defined by 
$
\langle \boldsymbol{u}, \boldsymbol{\phi}\rangle =  
\int_{\Omega} u_1\phi_1+u_2\phi_2 \mathrm{d} \bm{x}
$ with $\bm{u}=(u_1, u_2)$ and $\bm{\phi}=(\phi_1, \phi_2)$. 
The operator $\mathcal{H}(\boldsymbol{u}_h, \boldsymbol{\phi}_h)$ represents the spatial discretization associated with the CDG method, explained in detail as follows. 

\subsubsection*{One-dimensional setting}
Let $\mathcal{T}_h^C = \{[x_{i-\frac{1}{2}},x_{i+\frac{1}{2}}]\}$ 
be a partition of 
the one-dimensional (1D) domain $\Omega=[x_{\min}, x_{\max}]$
with $h_{i} = x_{i+\frac{1}{2}}-x_{i-\frac{1}{2}}$
and $h = \max_i h_i$.
Let $x_i = \frac{1}{2}(x_{i-\frac{1}{2}} + x_{i+\frac{1}{2}})$ denote the cell centers. The dual mesh is then given by $\mathcal{T}_h^D = { [x_i, x_{i+1}] }$. 
The spatial discretization operator $\mathcal{H}(\boldsymbol{u},\boldsymbol{\phi})$ is given by
\begin{equation}\label{eq:1D_H}
\begin{aligned}
\mathcal{H}(\boldsymbol{u},\boldsymbol{\phi}) =
&-\frac{1}{\tau_{\max }} \int_{\Omega}(u_1-u_2)(\phi_1-\phi_2) \mathrm{d}x
+\int_{\Omega} f(u_2)\partial_x\phi_1+ f(u_1)\partial_x\phi_2\mathrm{d}x\\ 
&-\sum_{i} \left(f(u_2)_{i+\frac{1}{2}}(\phi_1)_{i+\frac{1}{2}}^--f(u_2)_{i-\frac{1}{2}}(\phi_1)_{i-\frac{1}{2}}^+ 
+f(u_1)_{i+1}(\phi_2)_{i+1}^--f(u_1)_{i}(\phi_2)_{i}^+ \right),
\end{aligned}
\end{equation}
where 
$(\phi_1)_{i+\frac{1}{2}}^{\pm}:= {\rm lim}_{\epsilon\rightarrow 0^+}\phi_1(x_{i+\frac{1}{2}}\pm \epsilon)$, and 
$(\phi_2)_{i}^{\pm}$ is defined similarly. 
Here, $\tau_{\rm max} = C_{\rm CFL}\frac{h}{\max{|f'(u)|}}$ 
with $C_{\rm CFL}$ being the Courant--Friedrichs--Lewy (CFL) number.

\subsubsection*{Two-dimensional setting}
Consider a two-dimensional (2D) rectangular domain $\Omega = [x_{\min}, x_{\max}] \times [y_{\min}, y_{\max}]$, discretized using Cartesian meshes. The primal mesh $\mathcal{T}_h^C$ consists of cells $[x_{i-\frac{1}{2}}, x_{i+\frac{1}{2}}] \times [y_{j-\frac{1}{2}}, y_{j+\frac{1}{2}}]$, and the dual mesh $\mathcal{T}_h^D$ consists of cells $[x_i, x_{i+1}] \times [y_j, y_{j+1}]$, where $x_i = \frac{1}{2}(x_{i-\frac{1}{2}} + x_{i+\frac{1}{2}})$ and $y_j = \frac{1}{2}(y_{j-\frac{1}{2}} + y_{j+\frac{1}{2}})$. 
The spatial discretization operator $\mathcal{H}(\boldsymbol{u}, \boldsymbol{\phi})$ is defined by
\begin{equation}\label{eq:2D_H}
\begin{aligned}
\mathcal{H}(\boldsymbol{u},\boldsymbol{\phi}) =
&-\frac{1}{\tau_{\max }} \int_{\Omega}(u_1-u_2)(\phi_1-\phi_2) \mathrm{d} \bm{x}
+\int_{\Omega} \bm{f}(u_2) \cdot \nabla \phi_1+ \bm{f}(u_1)\cdot \nabla \phi_2\mathrm{d}\bm{x}\\ 
&-\sum_{i,j} \bigg(\int_{x_{i-\frac{1}{2}}}^{x_{i+\frac{1}{2}}} 
f_2(u_2(x,y_{j+\frac{1}{2}}))(\phi_1)_{j+\frac{1}{2}}^{-}
-f_2(u_2(x,y_{j-\frac{1}{2}}))(\phi_1)_{j-\frac{1}{2}}^{+} \mathrm{d}x+\\
&\hspace{1.2cm}\int_{y_{j-\frac{1}{2}}}^{y_{j+\frac{1}{2}}} 
f_1(u_2(x_{i+\frac{1}{2}},y))(\phi_1)_{i+\frac{1}{2}}^{-}
-f_1(u_2(x_{i-\frac{1}{2}},y))(\phi_1)_{i-\frac{1}{2}}^{+} \mathrm{d}y+\\
&\hspace{1cm}\int_{x_i}^{x_{i+1}} f_2(u_1(x,y_{j+1}))(\phi_2)_{j+1}^{-}
-f_2(u_1(x,y_{j}))(\phi_2)_{j}^{+} \mathrm{d}x+\\
&\hspace{1cm}\int_{y_j}^{y_{j+1}} f_1(u_1(x_{i+1},y))(\phi_2)_{i+1}^{-}
-f_1(u_1(x_{i},y))(\phi_2)_{i}^{+} \mathrm{d}y\bigg),
\end{aligned}
\end{equation}
where $\boldsymbol{f}(u) = (f_1(u), f_2(u))$ is the flux function, and $\tau_{\text{max}}$ is defined as
$ 
\tau_{\text{max}} = \frac{C_{\text{CFL}}}{{\max |f_1'(u)|}/{h_x} + {\max |f_2'(u)|}/{h_y}},
$ 
with $h_x = \max_i (x_{i+\frac{1}{2}} - x_{i-\frac{1}{2}})$ and $h_y = \max_j (y_{j+\frac{1}{2}} - y_{j-\frac{1}{2}})$. 
The traces at cell interfaces are defined as
$(\phi_1)_{i+\frac{1}{2}}^{\pm} = \lim_{\epsilon \to 0^+} \phi_1(x_{i+\frac{1}{2}} \pm \epsilon, y)$,  
$(\phi_1)_{j+\frac{1}{2}}^{\pm} = \lim_{\epsilon \to 0^+} \phi_1(x, y_{j+\frac{1}{2}} \pm \epsilon)$, 
and similarly for $(\phi_2)_{i}^{\pm}$ and $(\phi_2)_{j}^{\pm}$.

\subsection{Fully-discrete 
OECDG methods}
The fully discrete OECDG method computes numerical solutions at discrete time levels $\{ t_n \}$. Let $\tau = t_{n+1} - t_n$ denote the time step size. The method incorporates a novel OE procedure after each Runge--Kutta (RK) stage. 
Coupled with an $r$th-order, $s$-stage RK method, the fully discrete OECDG method is expressed as
\begin{subequations}
\label{eq:RKCDG}
\begin{align}
    \bm{u}_\sigma^{n,0} &= \bm{u}_\sigma^{n},\\
    \langle \bm{u}_h^{n,l+1} , \boldsymbol{\varphi}_h\rangle  
    &=\sum_{0\leq m\leq l}(c_{lm}\langle \boldsymbol{u}_{\sigma}^{n,m} , \boldsymbol{\varphi}_h\rangle+\tau d_{lm} \mathcal{H}(\boldsymbol{u}_{\sigma}^{n,m} , \boldsymbol{\varphi}_h)),\label{OECDG}\\
    \boldsymbol{u}_{\sigma}^{n,l+1} &= \mathcal{F}_{\tau}\boldsymbol{u}_h^{n,l+1},\quad l = 0,1,...,s-1,\label{OEstep1}\\
    \bm{u}_\sigma^{n+1} &= \bm{u}_\sigma^{n,s},
\end{align}
\end{subequations}
where $\boldsymbol{u}_h^{n,l}$ is the numerical solution at the $l$th stage, and $\boldsymbol{u}_\sigma^{n,l}$ is the solution after applying the OE step \eqref{OEstep1}. The coefficients $c_{lm}$ and $d_{lm}$ are determined by the chosen RK method.

For any $\boldsymbol{u}_h = (u_h^C, u_h^D) \in V_h \times W_h$, the operator $\mathcal{F}_\tau \boldsymbol{u}_h$ represents the solution $\boldsymbol{u}_\sigma(\bm{x}, \hat{t})$ at $\hat{t} = \tau$ of the following damping ODEs with initial condition $\boldsymbol{u}_\sigma(\bm{x}, 0) = \boldsymbol{u}_h$:
\begin{equation}\label{OEstep2D}
    \left\{
    \begin{aligned}
        \frac{\mathrm{d}}{\mathrm{d}\hat{t}}\int_{C} u_{\sigma}^C\varphi_1\mathrm{d} \bm{x}
        +\sum_{m=0}^k {\delta_C^m(\bm{u}_h)}\int_{C}(u_{\sigma}^C-P^{m-1} u_{\sigma}^C)\varphi_1\mathrm{d} \bm{x} = 0,\\
        \frac{\mathrm{d}}{\mathrm{d}\hat{t}}\int_{D} u_{\sigma}^D\varphi_2\mathrm{d} \bm{x} 
        +\sum_{m=0}^k {\delta_D^m(\bm{u}_h)}\int_{D}(u_{\sigma}^D-P^{m-1} u_{\sigma}^D)\varphi_2\mathrm{d} \bm{x} = 0,
    \end{aligned}
    \right.
\end{equation}
where $P^m$ denotes the standard $L^2$ projection into $\mathbb{P}^m(K)$ (with $K = C$ or $D$) for $m \geq 0$, and $P^{-1}$ is defined as $P^0$. 
The coefficients $\delta_K^m(\boldsymbol{u}_h)$ are given by
\begin{equation*}
\delta_{K}^m(\bm{u}_h) 
=
\begin{cases}
0, & \sigma_K(\bm{u}_h) = 0,\\
\displaystyle
\sum_{e\in \partial K} \beta_{e,K}\frac{(2m+1)h_{e,K}^{m-1}}{(2k-1)m!\sigma_K(\bm{u}_h)|e|}
\int_e\left(\sum_{|\bm{\alpha}|=m}\big(\partial^{\bm{\alpha}}(u_h^C-u_h^D)\big)^2\right)^{\frac12}\mathrm{d}s,&
\sigma_K(\bm{u}_h) \ne 0,
\end{cases}
\end{equation*}
where $\sigma_K(\bm{u}_h)$ is defined as 
\begin{equation*}
\sigma_{K}(\bm{u}_h) 
=
\left\{
\begin{array}{ll}
\left\|u_h^C-{\rm avg}(u_h^C)\right\|_{L^{\infty}(\Omega)}, & K \in \mathcal{T}_h^C,\\
\left\|u_h^D-{\rm avg}(u_h^D)\right\|_{L^{\infty}(\Omega)}, & K \in \mathcal{T}_h^D,
\end{array}
\right.
\end{equation*}
and ${\rm avg} (f) = \frac{1}{|\Omega|} \int_\Omega f \, d\bm{x}$ is the average over the entire domain $\Omega$, $h_{e,K} = \sup_{\bm{x} \in K} \operatorname{dist}(\bm{x}, e)$.
The coefficient $\beta_{e,K}$ estimates the local maximum wave speed in the direction $\bm{n}_e$. In computations, $\beta_{e,K}$ is chosen as $|\boldsymbol{f}' \cdot \bm{n}_e|$, where $\boldsymbol{f}'$ is evaluated at the cell average of $u_h^C$ on $K \in \mathcal{T}_h^C$, or at the cell average of $u_h^D$ on $K \in \mathcal{T}_h^D$.\begin{remark}
The proposed OE damping operator is novel and distinct from the original OE operator used in the non-central DG schemes \cite{peng2023oedg,lu2021oscillation}. Specifically, our new damping approach is more compact and relies on the difference between the primal and dual CDG solutions, instead of utilizing the jump information at the cell interface. 
    Specifically, for 1D scalar hyperbolic conservation laws, the coefficients $\delta_K^m(\boldsymbol{u}_h)$ for cells $I_i \in \mathcal{T}_h^C$ and $I_{i+\frac{1}{2}} \in \mathcal{T}_h^D$ are 
    \begin{equation}\nonumber
    \begin{aligned}
        \delta_{I_i}^m(\bm{u}_h) &= \frac{(2m+1)h^{m-1}}{(2k-1)m!}
        \frac{\beta_{i+\frac{1}{2}}|\partial_x^m(u_h^C-u_h^D)_{i+\frac{1}{2}}|+
        \beta_{i-\frac{1}{2}}|\partial_x^m(u_h^C-u_h^D)_{i-\frac{1}{2}}|}{\|u_h^C-{\rm avg}(u_h^C)\|_{L^{\infty}(\Omega)}},\\
        \delta_{I_{i+\frac{1}{2}}}^m(\bm{u}_h) &= \frac{(2m+1)h^{m-1}}{(2k-1)m!}
        \frac{\beta_{i+1}|\partial_x^m(u_h^C-u_h^D)_{i+1}|+\beta_i|\partial_x^m(u_h^C-u_h^D)_{i}|}{\|u_h^D-{\rm avg}(u_h^D)\|_{L^{\infty}(\Omega)}}.
    \end{aligned}
    \end{equation}
\end{remark}

\begin{remark}
For hyperbolic systems with $N$ equations, all components share the same coefficient $\delta_K^m$ in the OE procedure \eqref{OEstep1}. It is defined as
    \begin{equation}
        \delta_K^m = \max_{1\leq q\leq N} \delta_{K}^m \left(\bm{u}_h^{(i)}\right)\quad \mbox{with} \quad 
        \bm{u}^{(i)}_h := \left(u_h^{C,(i)}, u_h^{D,(i)}\right),
    \end{equation}
where $u_h^{C,(q)}$ and $u_h^{D,(q)}$ are the $q$-th components of $\boldsymbol{u}_h^C$ and $\boldsymbol{u}_h^D$, respectively. 
The coefficient $\beta_{e,K}$ is taken as the spectral radius of the Jacobian matrix $\sum_{i=1}^d n_e^{(i)} \frac{\partial \boldsymbol{f}_i}{\partial \boldsymbol{u}}$, where $n_e^{(i)}$ is the $i$-th component of the unit normal vector $\boldsymbol{n}_e$. 
\end{remark}

\section{Theoretical analysis}\label{sec:Optimal}

In this section, we first establish the approximate skew-symmetry and weak boundedness of the semi-discrete CDG method, and then analyze the linear stability of the Runge--Kutta CDG schemes, referred to as RKCDG, using the matrix transferring technique \cite{xu20192}. Based on these results, we prove the optimal error estimate for the fully-discrete OECDG method. 
Unless otherwise stated, 
the flux $f(u)$ 
is assumed to be linear 
with respect to $u$. 
Specifically, in the 1D case,  
$f(u)=\beta_1 u$,
and in the 2D case, 
$\bm{f}(u)= (\beta_1 u,\beta_2 u)$, 
where $\beta_1$ and $\beta_2$ are constants. 
For simplicity, we focus on 
periodic boundary conditions, 1D quasi-uniform meshes, and 2D uniform Cartesian meshes.

\subsection{Properties of CDG discretization}
We show that 
the CDG spatial discretization $\mathcal{H}(\boldsymbol{h}, \boldsymbol{g})$ exhibits the 
approximate skew-symmetry 
and weak boundedness,
which are essential for the linear stability analysis
in the next subsection. 

\begin{lemma}\label{lem:skew-symm}
The CDG discretization operator 
$\mathcal{H}(\boldsymbol{h},\boldsymbol{g})$ is approximately skew-symmetric:
\begin{equation*}
    \mathcal{H}(\boldsymbol{h},\boldsymbol{g})+\mathcal{H}(\boldsymbol{g},\boldsymbol{h}) 
    = -\frac{2}{\tau_{\rm max}}\int_{\Omega}(h_1-h_2)(g_1-g_2) 
    \mathrm{d}\boldsymbol{x}.
\end{equation*}
\end{lemma}

\begin{proof}

We decompose each cell integral in \eqref{eq:1D_H}  and \eqref{eq:2D_H} into integrals over the sub-cells split by the dual cells. 
Noting the continuity of $\boldsymbol{h}$ 
and $\boldsymbol{g}$ on each dual region, we employ integration by parts 
and cancel the non-dissipative terms. Specifically, in the 1D case, we obtain 
\begin{equation}\nonumber
\begin{aligned}
&\mathcal{H}(\boldsymbol{h},\boldsymbol{g})+\mathcal{H}(\boldsymbol{g},\boldsymbol{h}) 
+\frac{2}{\tau_{\rm max}}\int_{\Omega}(h_1-h_2)(g_1-g_2)\mathrm{d}x
\\=& \beta_1 \sum_{i} \sum_{s_1 \in \left\{0,\frac{1}{2}\right\}}
\bigg(\int_{x_{i-\frac{1}{2}+s_1}}^{x_{i+s_1}}
\partial_x \left(h_2 g_1 + h_1 g_2\right)\mathrm{d}x
-(h_2g_1+g_2h_1)_{i+s_1}^-+(h_2g_1+g_2h_1)_{i-\frac{1}{2}+s_1}^+ \bigg)\\
=&\,0.
\end{aligned}
\end{equation}
Similarly, in the 2D case, we have 
\begin{equation}\nonumber
\begin{aligned}
&\mathcal{H}(\boldsymbol{h},\boldsymbol{g})+\mathcal{H}(\boldsymbol{g},\boldsymbol{h}) +\frac{2}{\tau_{\max}}\int_{\Omega} (h_1-h_2)(g_1-g_2)\mathrm{d} \bm{x} \\
=&\sum_{i,j} \sum_{s_1,s_2 \in \left\{0,\frac{1}{2}\right\}} 
\bigg(
\int_{y_{j-\frac{1}{2}+s_2}}^{y_{j+s_2}}
\int_{x_{i-\frac{1}{2}+s_1}}^{x_{i+s_1}}
\bm{\beta} \cdot \nabla(h_2g_1+h_1g_2) \,\mathrm{d}x \mathrm{d}y
-\beta_1
\int_{y_{j-\frac{1}{2}+s_2}}^{y_{j+s_2}}
(h_2g_1+h_1g_2)_{i+s_1}^- \\& - (h_2g_1+h_1g_2)_{i-\frac{1}{2}+s_1}^+ \mathrm{d} y
-\beta_2
\int_{x_{i-\frac{1}{2}+s_1}}^{x_{i+s_1}}
(h_2g_1+h_1g_2)_{j+s_2}^- - (h_2g_1+h_1g_2)_{j-\frac{1}{2}+s_2}^+ \mathrm{d} x\bigg)\\
=&\,0.
\end{aligned}
\end{equation}
The proof is completed.
\end{proof}

\begin{remark}
As a direct corollary of \Cref{lem:skew-symm}, setting $\boldsymbol{g} = \boldsymbol{h} = \boldsymbol{u}_h$ gives the $L^2$-boundedness for the semi-discrete CDG method:
\begin{equation*}
\partial_t \left(\frac{1}{2} \|\boldsymbol{u}_h\|^2\right)
=
\mathcal{H}(\boldsymbol{u}_h, \boldsymbol{u}_h) 
= -\frac{1}{\tau_{\rm max}} \int_{\Omega} \left(u_h^C - u_h^D\right)^2 \, \mathrm{d} \bm{x} \leq 0,
\end{equation*}
which has been previously established in \cite{liu2008stability}.
\end{remark}

\begin{lemma}\label{Hcontinuous}
The CDG discretization operator is weakly bounded
in the sense that
\begin{equation}\label{eq:H_bounded}
\left|\mathcal{H}(\boldsymbol{h},\boldsymbol{g})\right|\leq Ch^{-1}\left\|\boldsymbol{h}\right\|\left\| \boldsymbol{g}\right\| 
\quad \forall \bm{h}, \bm{g} \in V_h \times W_h, 
\end{equation}
where the constant $C > 0$ is independent of $h$ and of $\boldsymbol{h}$ and $\boldsymbol{g}$.
\end{lemma}

\begin{proof}
For the 1D case, $\mathcal{H}(\boldsymbol{h}, \boldsymbol{g})$ can be reformulated as
\begin{align*}
\mathcal{H}(\boldsymbol{h}, \boldsymbol{g}) = 
& -\frac{1}{\tau_{\max}} \int_{\Omega} (h_1 - h_2)(g_1 - g_2) \, \mathrm{d}x 
+ \beta_1 \int_{\Omega} h_2 \partial_x g_1 + h_1 \partial_x g_2 \, \mathrm{d}x \\
& + \beta_1 \sum_i \left((h_2)_{i+\frac{1}{2}} \jump{g_1}_{i+\frac{1}{2}} + (h_1)_{i} \jump{g_2}_{i} \right),
\end{align*}
where the jump $\jump{f}_{i+\frac{1}{2}} := f_{i+\frac{1}{2}}^+ - f_{i+\frac{1}{2}}^-$. Applying Cauchy--Schwarz yields
\begin{align*}
|\mathcal{H}(\boldsymbol{h}, \boldsymbol{g})| \leq 
& \frac{1}{\tau_{\max}} \|h_1 - h_2\| \|g_1 - g_2\| + |\beta_1| \left(\|h_2\| \|\partial_x g_1\| + \|h_1\| \|\partial_x g_2\|\right) \\
& + |\beta_1| \left(\sqrt{\sum_i (h_2)_{i+\frac{1}{2}}^2} \sqrt{\sum_i \jump{g_1}_{i+\frac{1}{2}}^2} + \sqrt{\sum_i (h_1)_{i}^2} \sqrt{\sum_i \jump{g_2}_{i}^2}\right).
\end{align*}
Note that $1/\tau_{\max} = |\beta_1| C_{CFL}^{-1} h^{-1}$, $\|h_1-h_2\| \le \sqrt{2}\|\bm{h}\|$
and $\|g_1-g_2\| \le \sqrt{2}\|\bm{g}\|$, one has 
$\frac{1}{\tau_{\max}} \|h_1-h_2\| \|g_1-g_2\| \le 2|\beta_1| C_{CFL}^{-1}h^{-1} \|\bm{h}\|\|\bm{g}\|$.
By the inverse inequalities on 
\begin{equation}\nonumber
\mathbb{V}_h:=\big\{\phi \in L^2(\Omega): 
\phi\big|_{[x_{i-\frac{1}{2}}, x_i]} \in \mathbb{P}^k([x_{i-\frac{1}{2}}, x_i]),~~
\phi\big|_{[x_{i}, x_{i+\frac{1}{2}}]} \in \mathbb{P}^k([x_{i}, x_{i+\frac{1}{2}}])\,\,
\forall i
\big\},
\end{equation}
we have for any $p \in \mathbb{V}_h$, there exist constants $\mu_s>0$ ($s=1,2$) 
independent of $p$ and $h$ such that
\begin{equation*}
    \|\partial_x p\| \le \mu_1 h^{-1} \|p\|,\quad
    \|p\|_{\Gamma_h} \le \mu_2 h^{-1/2} \|p\|,
\end{equation*}
where
\begin{equation*}
    \|p\|_{\Gamma_h}^2 := 
    \sum_{i} 
    \left(
        (p^{+}_{i-\frac{1}{2}})^2+
        (p^{-}_{i})^2+
        (p^{+}_{i})^2+
        (p^{-}_{i+\frac{1}{2}})^2
    \right).
\end{equation*}
Since $h_1, h_2, g_1, g_2 \in \mathbb{V}_h$, by applying the inverse inequalities,
for $k=1,2$, one has
$\|\partial_x g_k\| \le \mu_1 h^{-1} \|g_k\|$,
$\sqrt{\sum_{i} (h_k)_{i+\frac{1}{2}}^2} \le \mu_2 h^{-1/2} \|h_k\| $,
and 
$\sqrt{\sum_{i}\jump{g_k}_{i+\frac{1}{2}}^2} \le \mu_2 h^{-1/2}\sqrt{2} \|g_k\|$.
Therefore,
we conclude that
$
\left|\mathcal{H}(\bm{h},\bm{g})\right|
\le Ch^{-1}\left\|\bm{h}\right\|\left\| \bm{g}\right\|
$.

The analysis can be extended to the 2D case.
The CDG discretization operator 
$\mathcal{H}(\bm{h},\bm{g})$ can be rewritten as
\begin{align*}
\mathcal{H}(\bm{h},\bm{g})= 
&-\frac{1}{\tau_{\max }} \int_{\Omega}(h_1-h_2)(g_1-g_2) \mathrm{d} \bm{x}+
\int_{\Omega} 
\beta_1(h_2 \partial_x g_1+ h_1 \partial_x g_2)+
\beta_2(h_2 \partial_y g_1+ h_1 \partial_y g_2)
\mathrm{d}\bm{x}\\ 
&+\beta_1\int_{y_{\min}}^{y_{\max}} 
\sum_i\left(
(h_2)_{i-\frac{1}{2}}\jump{g_1}_{i-\frac{1}{2}}  
+ (h_1)_{i}\jump{g_2}_{i} 
\right)
\mathrm{d}y\\
&+\beta_2\int_{x_{\min}}^{x_{\max}} 
\sum_{j}\left(
(h_2)_{j-\frac{1}{2}}\jump{g_1}_{j-\frac{1}{2}} + (h_1)_{j}\jump{g_2}_{j} 
\right)
\mathrm{d}x.
\end{align*}
Note that
\begin{align*}
&\int_{y_{\min}}^{y_{\max}}
\left[
\int_{x_{\min}}^{x_{\max}}
h_2 \partial_x g_1+ h_1 \partial_x g_2 \mathrm{d}x
+
\sum_i\left(
(h_2)_{i-\frac{1}{2}}\jump{g_1}_{i-\frac{1}{2}}  
+ (h_1)_{i}\jump{g_2}_{i} 
\right) 
\right]
\mathrm{d} y\\
\le&\,
Ch^{-1}
\int_{y_{\min}}^{y_{\max}}
\left[
\left(\int_{x_{\min}}^{x_{\max}}h_2^2 \mathrm{d}x\right)^{\frac{1}{2}}
\left(\int_{x_{\min}}^{x_{\max}}g_1^2 \mathrm{d}x\right)^{\frac{1}{2}}
\right]
\mathrm{d} y
\le 
Ch^{-1} \|h_2\| \|g_1\|,
\end{align*}
where the first inequality is derived similarly to the 1D case, and the second follows from the Cauchy--Schwarz inequality. These lead to  $\left|\mathcal{H}(\bm{h},\bm{g})\right| \le Ch^{-1}\|\bm{h}\| \|\bm{g}\|$.
\end{proof}

\subsection{Linear stability of RKCDG method}
We now prove the linear stability of the RKCDG method,
namely \eqref{eq:RKCDG} with $\mathcal{F}_\tau$ as the identity operator,
by following the standard techniques in \cite{xu20192, sun2019strong}.

\begin{definition}\label{stage_diff}
    The $i$-th temporal difference operator $\mathbb{D}_i$, mapping $V_h\times W_h$ to $V_h\times W_h$, is defined recursively as
    \begin{equation}\label{timedif_D}
        \begin{aligned}
            \mathbb{D}_0\boldsymbol{u}_h &= \boldsymbol{u}_h,\\
            \langle\mathbb{D}_i \bm{u}_h,\boldsymbol{\varphi}_h 
            \rangle &= \tau \mathcal{H}(\mathbb{D}_{i-1} \boldsymbol{u}_h, \boldsymbol{\varphi}_h),\quad \forall \boldsymbol{\varphi}_h\in V_h\times W_h, \, i \geq 1.
        \end{aligned}
    \end{equation}
\end{definition}

Using $\mathbb{D}_i$, we can rewrite the RKCDG method, expressing $\bm{u}_h^{n+m}$ as a linear combination of the terms $\mathbb{D}_i \bm{u}_h^n$. Specifically, for any integer $m \geq 0$, there exist constants $\alpha_0, \alpha_1, \dots, \alpha_{ms}$ with $\alpha_0>0$ such that
\begin{equation}\label{eq:evolution_identity}
\alpha_0 \bm{u}_h^{n+m} = 
\sum_{0 \le i \le ms} \alpha_i \mathbb{D}_i \bm{u}_h^n.
\end{equation}
Typically, $m$ is set to $1$, except when combining multiple time steps for stability analysis.

Taking the $L^2$ norm of both sides of \eqref{eq:evolution_identity}, we derive the energy equation for stability analysis:
\begin{equation}\label{energy}
    \alpha_0^2 \left(\|\bm{u}_h^{n+m}\|^2 - \|\bm{u}_h^{n}\|^2\right)
    = \sum_{0 \leq i,j \leq ms, \, i+j>0} \alpha_i \alpha_j \left\langle \mathbb{D}_{i}\bm{u}_h^n, \mathbb{D}_{j} \boldsymbol{u}_h^n \right\rangle.
\end{equation}
By applying \eqref{timedif_D}, it follows that
$
\left\langle \mathbb{D}_{i}\bm{u}_h^n, \mathbb{D}_{j} \boldsymbol{u}_h^n \right\rangle
=
\tau \mathcal{H}(\mathbb{D}_{i-1}\bm{u}_h^n, \mathbb{D}_{j} \boldsymbol{u}_h^n),
$ 
and the right-hand side of \eqref{energy} can be expressed as
\begin{equation*}
\sum_{0\le i,j \le ms} a_{ij} \left\langle \mathbb{D}_{i}\bm{u}_h^n, \mathbb{D}_{j} \boldsymbol{u}_h^n \right\rangle
+
\sum_{0\le i,j \le ms} b_{ij} \tau \mathcal{H}(\mathbb{D}_{i}\bm{u}_h^n, \mathbb{D}_{j} \boldsymbol{u}_h^n).
\end{equation*}

The coefficients of $\langle \mathbb{D}_{i}\boldsymbol{u}_h^n,\mathbb{D}_{j}\boldsymbol{u}_h^n \rangle$ 
and $\tau\mathcal{H}(\mathbb{D}_{i}\boldsymbol{u}_h^n,\mathbb{D}_{j}\boldsymbol{u}_h^n)$ are encoded in matrices $\mathbb{A} = \{a_{ij}\}$ and $\mathbb{B} = \{b_{ij}\}$, respectively. Initially, $a_{00}=0$, $a_{ij}=\alpha_i \alpha_j$ if $i+j>0$, and $b_{ij}=0$. For $i > j+1$,
$$\langle \mathbb{D}_{i+1}\boldsymbol{u}_h,\mathbb{D}_{j}\boldsymbol{u}_h\rangle 
+\langle \mathbb{D}_{i}\boldsymbol{u}_h,\mathbb{D}_{j+1}\boldsymbol{u}_h\rangle = 
\tau(\mathcal{H}(\mathbb{D}_{i}\boldsymbol{u}_h,\mathbb{D}_{j}\boldsymbol{u}_h)
+\mathcal{H}(\mathbb{D}_{j}\boldsymbol{u}_h,\mathbb{D}_{i}\boldsymbol{u}_h)).$$ 
The matrices $\mathbb{A}$ and $\mathbb{B}$ 
are transferred using the following inductive relations: 
$$a_{ij}^{(\ell+1)} = \left\{\begin{aligned}
	&0, &\quad &0\leq j\leq \ell,\\
	&a_{ij}^{(\ell)}-2a_{i+1,j-1}^{(\ell)}, &\quad &(i,j) = (\ell+1,\ell+1),\\
	&a_{ij}^{(\ell)}-a_{i+1,j-1}^{(\ell)},&\quad &\ell+2\leq i\leq ms-1 \ {\rm and }\ j = \ell+1,\\
	&a_{ij}^{(\ell)}, &\quad &{\rm otherwise},
\end{aligned}\right.$$
and
$$b_{ij}^{(\ell+1)} = \left\{\begin{aligned}
    &2a_{i+1,j}^{(\ell)},&\quad &\ell\leq i\leq ms-1 \ {\rm and }\ j = \ell,\\
    &b_{ij}^{(\ell)}, &\quad &{\rm otherwise}. 
\end{aligned}\right.$$
The matrices after $\ell$ transfers are denoted as 
$\mathbb{A}^{(\ell)} = \{a_{ij}^{(\ell)}\}$ and
$\mathbb{B}^{(\ell)} = \{b_{ij}^{(\ell)}\}$.

The transfer halts when $a_{\zeta\zeta}^{(\zeta)} \neq 0$, with $\zeta$ called the termination index. Let $\mathbb{B}^{(\zeta)}_{\kappa}$ denote the $(\kappa+1)$-th order leading principal submatrix of $\mathbb{B}^{(\zeta)}$. Define the contribution index $\rho := \max \{\kappa: 0 \leq \kappa \leq \zeta \text{ and } \mathbb{B}^{(\zeta)}_{\kappa-1} \text{ is positive definite}\}$. Similar to the discussion in \cite{xu20192}, we can show that there exist polynomials $\mathbb{Q}_1$ and $\mathbb{Q}_2$ with nonnegative coefficients such that
\begin{equation*}
    \alpha_0^2 \left(\|\bm{u}_h^{n+m}\|^2 - \|\bm{u}_h^{n}\|^2\right)
    \le \mathbb{Y}_1 + \mathbb{Y}_2
\end{equation*}
with
$$\begin{aligned}
    \mathbb{Y}_1 &= \big(a_{\zeta\zeta}^{(\zeta)}+\lambda\mathbb{Q}_1(\lambda)+\lambda\mathbb{Q}_2(\lambda)\big)\|\mathbb{D}_{\zeta}\boldsymbol{u}_h^n\|^2+\lambda\mathbb{Q}_2(\lambda)\|\mathbb{D}_{\rho}\boldsymbol{u}_h^n\|^2,\\
    \mathbb{Y}_2 & = -\frac{\epsilon\tau}{2\tau_{\max}}\sum_{0\leq i\leq \rho-1} \int_{\Omega} \Big((\mathbb{D}_i\boldsymbol{u}_h^n)^C-(\mathbb{D}_i\boldsymbol{u}_h^n)^D\Big)^2 \mathrm{d}\bm{x},
\end{aligned}$$
where $\epsilon$ is the smallest eigenvalue of $\mathbb{B}^{(\zeta)}_{\rho-1}$,
$\gamma:=\tau/h$, and $(\mathbb{D}_i\boldsymbol{u}_h^n)^C$ and $(\mathbb{D}_i\boldsymbol{u}_h^n)^D$
denote the first and second components of $\mathbb{D}_i\boldsymbol{u}_h^n$. 
 We now state the stability theorem for the RKCDG method.

\begin{theorem}\label{stability}
    For the $r$th-order $s$-stage RKCDG method, the following stability properties hold:
    \begin{enumerate}
        \item If $\rho = \zeta$ and $a_{\zeta\zeta}^{(\zeta)} < 0$, the RKCDG method is monotonically stable. Specifically, there exists a constant $C>0$ such that $\| \boldsymbol{u}_h^{n+m}\| \leq \| \boldsymbol{u}_h^{n}\|$ for $\tau \le Ch$.
    
        \item If $\rho = \zeta$ and $a_{\zeta\zeta}^{(\zeta)} > 0$, the RKCDG method is weak$(2\zeta)$ stable. This means that for sufficiently small $\lambda$, there exists a constant $C>0$ such that $\|\boldsymbol{u}_h^{n+m}\|^2 \leq (1+C \tau) \|\boldsymbol{u}_h^{n}\|^2$ when $\tau \le h^{1+\frac{1}{2\zeta-1}}$.
        
        \item If $\rho < \zeta$, the RKCDG method is weak$(2\rho+1)$ stable. This means that for sufficiently small $\lambda$, there exists a constant $C>0$ such that $\| \boldsymbol{u}_h^{n+m}\|^2 \leq (1+C\tau) \| \boldsymbol{u}_h^{n}\|^2$ when $\tau \leq h^{1+\frac{1}{2\rho}}$.
    \end{enumerate} 
\end{theorem}

\begin{proof}
If $a_{\zeta\zeta}^{(\zeta)} < 0$, then $\|\boldsymbol{u}_h^{n+m}\|^2 - \|\boldsymbol{u}_h^{n}\|^2 \leq 0$ holds for $\lambda = 0$. Since the polynomials involved are smooth and have a finite number of roots, there exists a constant $C$ such that $\|\boldsymbol{u}_h^{n+m}\|^2 - \|\boldsymbol{u}_h^{n}\|^2 \leq 0$ for $\lambda < C$. This proves the first result.

Note that $\|\mathbb{D}_i \bm{u}^n_h \| \le C \lambda \|\mathbb{D}_{i-1} \bm{u}^n_h\|$ for $i>0$, which follows from \eqref{eq:H_bounded} and \eqref{timedif_D}.
When $\rho = \zeta$, we have
\begin{equation*}
\mathbb{Y}_1 \leq \lambda^{2\zeta} 
(|a_{\zeta\zeta}^{(\zeta)}| + \lambda\mathbb{Q}_1(\lambda) + 2\lambda\mathbb{Q}_2(\lambda)) \|\mathbb{D}_{0}\boldsymbol{u}_h^n\|^2 \leq C\lambda^{2\zeta}\|\boldsymbol{u}_h^n\|^2,
\end{equation*}
for sufficiently small $\lambda$. Since $\mathbb{Y}_2 \leq 0$, we have $\|\boldsymbol{u}_h^{n+m}\|^2 \leq (1+C\lambda^{2\zeta})\|\boldsymbol{u}_h^{n}\|^2$. Since $\tau \leq h^{1+\frac{1}{2\zeta-1}}$ is equivalent to $\lambda^{2\zeta} < \tau$, it follows that $\|\boldsymbol{u}_h^{n+m}\|^2 \leq (1+C\tau)\|\boldsymbol{u}_h^{n}\|^2$. This proves the second result.

If $\rho < \zeta$, then for sufficiently small $\lambda$:
\begin{equation*}
\mathbb{Y}_1 \leq \big[\lambda^{2\zeta-2\rho-1} (|a_{\zeta\zeta}^{(\zeta)}| + \lambda\mathbb{Q}_1(\lambda) + \lambda\mathbb{Q}_2(\lambda)) + \mathbb{Q}_2(\lambda)\big]\lambda\|\mathbb{D}_{\rho}\boldsymbol{u}_h^n\|^2 \leq C\lambda^{2\rho+1}\|\boldsymbol{u}_h^n\|^2.
\end{equation*}
Since $\tau \leq h^{1+\frac{1}{2\rho}}$ is equivalent to $\lambda^{2\rho+1} < \tau$, the result follows: $\|\boldsymbol{u}_h^{n+m}\|^2 \leq (1+C\tau)\|\boldsymbol{u}_h^{n}\|^2$. This completes the proof of the third result.
\end{proof}

\begin{remark}[Stability for RKCDG method with source terms]
    For the RKCDG method with source terms, we have
    \begin{equation*}
    \langle \boldsymbol{u}_h^{n,l+1} , \boldsymbol{\varphi}_h\rangle  
    =\sum_{0\leq m\leq l} \big(c_{lm}\langle \boldsymbol{u}_h^{n,m} , \boldsymbol{\varphi}_h\rangle+\tau d_{lm} (\mathcal{H}(\boldsymbol{u}_h^{n,m}, \boldsymbol{\varphi}_h)+
    \langle \boldsymbol{g}^{n,m}, \boldsymbol{\varphi}_h\rangle)\big),~ l=0,\cdots,s-1.
    \end{equation*}
    We have the stability result: 
    there exists a constant $\kappa\in[1,2]$ such that, for $\tau \leq Ch^{\kappa}$, 
    \begin{equation}\label{eq:stability_RKCDG}
    \|\boldsymbol{u}_h^{n+m}\|^2\leq (1+C\tau)\|\boldsymbol{u}_h^n\|^2+C\tau\sum_{0\leq l< ms}\|\boldsymbol{g}^{n,l}\|^2.
    \end{equation}
\end{remark}

\subsection{Optimal error estimates of fully-discrete OECDG method}
\begin{theorem}\label{thm:optimal_error}
Consider the $\mathbb{P}^k$-based OECDG method for the linear advection equation with an $r$th-order explicit RK method on 1D quasi-uniform meshes or 2D uniform Cartesian meshes. Assume $k > \frac{d}{2}$ and $r > \frac{d}{2} + 1$ with $d$ being the spatial dimension, and the exact solution $u(x,t)$ is sufficiently smooth. If 
$\left\|\bm{u}_\sigma^0-
\begin{bmatrix}
    u(\cdot,0)\\ u(\cdot,0)
\end{bmatrix} 
\right\|\le Ch^{k+1}$,
then the OECDG method admits the optimal error estimate 
\begin{equation}
\max_{0\leq n\leq \lceil \frac{T}{\tau}\rceil} 
\left\|\boldsymbol{u}_{\sigma}^n-
\begin{bmatrix}
    u(x,n\tau)\\
    u(x,n\tau)
\end{bmatrix}
\right\|
\leq C(h^{k+1}+\tau^r)
\end{equation}
under the same type of time step constraint as the corresponding RKCDG method.
\end{theorem}

The outline of the proof of Theorem \ref{thm:optimal_error} is presented as follows.  
\begin{description}
    \item[Step 1:] We make an a priori assumption:
    \begin{equation}\label{eq:assumption}
    \|\bm{\zeta}^{n,l}\|_{L^{\infty}(\Omega)} \le h,
    \end{equation}
    where $\bm{\zeta}^{n,l}:=\bm{u}_h^{n,l}-P\bm{U}^{n,l}$. 
    The definition of the reference function $\bm{U}^{n,l}$
    and projection $P$ will be given later.
    We then prove Theorem \ref{thm:optimal_error} based on this assumption through the following steps:
    \begin{enumerate}
        \item[(a)]  Reformulate the error equation for the OECDG method as an RKCDG scheme with source terms.
        \item[(b)] Show that these source terms are of high order.
        \item[(c)] Apply the stability result of RKCDG method to derive the global-in-time error estimate.
    \end{enumerate}

    \item[Step 2:] we prove the a priori assumption \eqref{eq:assumption}
    by mathematical induction.
\end{description}

\begin{remark}
For simplicity, we assume $h<1$ and $\tau<1$ are sufficient small.
Since $u(\cdot,0) \equiv C$ is a trivial case, as $\bm{u}_{\sigma}^{n,l} \equiv [C,C]^\top$ and the error estimate in Theorem \ref{thm:optimal_error} holds automatically. Thus, we focus on the non-constant case $u(\cdot,0) \not\equiv C$, and assume without loss of generality that 
$\int_{\Omega} (u^C_\sigma)^0 \mathrm{d}\bm{x} 
=\int_{\Omega} (u^D_\sigma)^0 \mathrm{d}\bm{x} 
=\int_{\Omega} u(\cdot,0) \mathrm{d}\bm{x} 
$,
which is typically satisfied by 
projection-based initialization.
\end{remark}

\subsection*{Preliminaries}
To prove Theorem \ref{thm:optimal_error}, we introduce some preliminary concepts, including the reference function $\bm{U}^{n,l}$ and a specific spatial projection operator $P$, both of which are essential for the error analysis.

The reference function 
$\boldsymbol{U}^{n,l}(x) := 
\begin{bmatrix}
        u^{n,l}(x)\\ u^{n,l}(x)
\end{bmatrix}$,
where $u^{n,l}(x)$ is recursively defined as
\begin{align}
u^{n,0} &= u(t_n,x),\nonumber\\
u^{n,l+1} &=
\sum_{0\leq m\leq l}\left(c_{lm} u^{n,m}
-\tau d_{lm} \bm{\beta}\cdot \nabla u^{n,m}\right)
+\tau \rho^{n,l},\,l=0,\cdots, s-1, \label{eq:reference_u}
\end{align}
with $\rho^{n,l} = 0$ for $l < s-1$, and $\rho^{n,s-1}$ chosen such that $u^{n,s} = u(x,t_{n+1})$. Since the RK method is of order $r$, we have
\begin{equation}\label{eq:rho}
    \|\rho^{n,s-1}\|\le C\tau^r.
\end{equation}
Let $\bm{\rho}^{n,l} = \begin{bmatrix}\rho^{n,l}\\\rho^{n,l}\end{bmatrix}$. Since $u^{n,l}$ is smooth, we find that $\mathcal{H}(\bm{U}^{n,l}, \bm{\varphi}_h) = -\langle \bm{\beta} \cdot \nabla \bm{U}^{n,l}, \bm{\varphi}_h \rangle$ for any $\bm{\varphi}_h \in V_h \times W_h$, so that
\begin{equation}\label{stagesol}
\langle \boldsymbol{U}^{n,l+1}, \boldsymbol{\varphi}_h\rangle  =
\sum_{0\leq m\leq l}(c_{lm}\langle \boldsymbol{U}^{n,m} , \boldsymbol{\varphi}_h\rangle
+\tau d_{lm} \mathcal{H}(\boldsymbol{U}^{n,m},\boldsymbol{\varphi}_h))
+\tau \langle\bm{\rho}^{n,l}, \bm{\varphi}_h\rangle.
\end{equation}

The spatial projection $P$ is defined as
$P\boldsymbol{U}^{n,l} := \begin{bmatrix}
    P_h^* u^{n,l}\\ 
    Q_h^* u^{n,l}
\end{bmatrix} \in V_h\times W_h$, 
where the projections $P_h^*$ and $Q_h^*$ are as in \cite{liu2018optimal}. Define the projection error by $\bm{\eta}^{n,l} := \boldsymbol{U}^{n,l} - P\boldsymbol{U}^{n,l}$. The projection operator $P$ is carefully constructed to satisfy a superconvergence property (see \eqref{eq:superconvergence_H}):
\begin{equation}\label{eq:superconvergence}
|\mathcal{H}(\boldsymbol{\eta}^{n,l}, \boldsymbol{\varphi}_h)| \leq Ch^{k+1}\|u^{n,l}\|_{H^{k+1}}\|\boldsymbol{\varphi}_h\|,
\end{equation}
which does not hold for the standard $L^2$ projection.
Additionally, the projection $P$ admits
the following approximation property \cite{liu2018optimal}:
\begin{equation}\label{eq:projection_approximation}
    \| \bm{\eta}^{n,l}\|_{H^m} \le Ch^{k+1-m}
    \| \bm{\eta}^{n,l}\|_{H^{k+1}}
    \quad \forall 0\le m\le k+1.
\end{equation}

\subsection*{Proof of Theorem \ref{thm:optimal_error}}
This subsection provides the proof of Theorem \ref{thm:optimal_error}. We start by rewriting the RK step \eqref{OECDG} of the OECDG method as follows:
\begin{equation}\label{1D1}
    \langle \boldsymbol{u}_\sigma^{n,l+1}, \boldsymbol{\varphi}_h\rangle  
    =\sum_{0\leq m\leq l}
    (c_{lm}\langle \boldsymbol{u}_{\sigma}^{n,m}, \boldsymbol{\varphi}_h\rangle
    +\tau d_{lm} \mathcal{H}(\boldsymbol{u}_{\sigma}^{n,m}, \boldsymbol{\varphi}_h))
    +\langle \boldsymbol{u}_\sigma^{n,l+1}- \boldsymbol{u}_h^{n,l+1}, \boldsymbol{\varphi}_h\rangle.
\end{equation} 
Subtracting \eqref{stagesol} from \eqref{1D1} and splitting the error, we get
$
\boldsymbol{u}_{\sigma}^{n,l}-\boldsymbol{U}^{n,l}   
=
\left(\boldsymbol{u}_{\sigma}^{n,l} - P\boldsymbol{U}^{n,l}\right)
-
\left(\boldsymbol{U}^{n,l}-P\boldsymbol{U}^{n,l}\right)
=: \bm{\zeta}_{\sigma}^{n,l} - \bm{\eta}^{n,l},
$
which yields
\begin{equation}\label{split}
    \langle \bm{\zeta}_{\sigma}^{n,l+1}, \boldsymbol{\varphi}_h\rangle  
    =\sum_{0\leq m\leq l}(c_{lm}\left\langle \boldsymbol{\zeta}_{\sigma}^{n,m}, \boldsymbol{\varphi}_h \right\rangle
    +\tau d_{lm} \mathcal{H}(\boldsymbol{\zeta}_{\sigma}^{n,m} , \boldsymbol{\varphi}_h))+\tau\mathcal{G}^{n,l}(\boldsymbol{\varphi}_h),
\end{equation}
where 
\begin{equation*}
\begin{aligned}
\mathcal{G}^{n,l}(\boldsymbol{\varphi}_h) :=&\frac{1}{\tau}\langle \boldsymbol{u}_\sigma^{n,l+1} - \boldsymbol{u}_h^{n,l+1}, \boldsymbol{\varphi}_h\rangle
-\langle \bm{\rho}^{n,l}, \boldsymbol{\varphi}_h \rangle + \frac{1}{\tau}\langle\boldsymbol{\eta}^{n,l+1}-\sum_{0\leq m\leq l}c_{lm} \boldsymbol{\eta}^{n,m},\boldsymbol{\varphi}_h\rangle\\
&-\sum_{0\leq m\leq l}d_{lm} \mathcal{H}(\boldsymbol{\eta}^{n,m},\boldsymbol{\varphi}_h).
\end{aligned}   
\end{equation*}
To apply the stability estimate result, we rewrite \eqref{split} as
\begin{equation}\label{eq:error_rk_source}
    \langle \bm{\zeta}_{\sigma}^{n,l+1}, \boldsymbol{\varphi}_h\rangle  
    =\sum_{0\leq m\leq l}(c_{lm}\left\langle \boldsymbol{\zeta}_{\sigma}^{n,m}, \boldsymbol{\varphi}_h \right\rangle
    +\tau d_{lm}(\mathcal{H}(\boldsymbol{\zeta}_{\sigma}^{n,m}, \boldsymbol{\varphi}_h)
    +\langle \bm{G}^{n,m}, \boldsymbol{\varphi}_h\rangle)
    ),
\end{equation}
where $\bm{G}^{n,m}$ is defined inductively by
\begin{equation*}
    d_{ll} \langle \bm{G}^{n,l}, \boldsymbol{\varphi}_h\rangle
    =
    \mathcal{G}^{n,l}(\boldsymbol{\varphi}_h)
    -
    \sum_{0\le m \le l-1}
    d_{lm} \langle \bm{G}^{n,m}, \boldsymbol{\varphi}_h\rangle,
    \quad 0\le l \le s-1.
\end{equation*}
This completes Step 1(a) in the outline. 
For the proof of Step 1(b), 
by definition, we have
\begin{equation*}
\|\bm{G}^{n,l}\| \le C \sum_{0\le m \le l} \|\mathcal{G}^{n,m}\|,
\quad 
\|\mathcal{G}^{n,m}\| := \mathrm{sup}_{\bm{\varphi} \in V_h \times W_h, \bm{\varphi} \ne \bm{0}} 
\frac{|\mathcal{G}^{n,m}(\boldsymbol{\varphi})|}
{\|\boldsymbol{\varphi}\|}.
\end{equation*}
The estimate of $\|\mathcal{G}^{n,l}\|$
is a critical step in the error analysis of OECDG method,
which will be carefully explored later. 
We will prove that, under the assumption \eqref{eq:assumption}, the following inequality holds:
\begin{equation}\label{eq:estimate_source}
    \|\mathcal{G}^{n,l}\| \le C(\|\bm{\zeta}^{n,l}_\sigma\|+h^{k+1}+\tau^r).
\end{equation}
For the proof of Step 1(c), 
we apply the stability estimate \eqref{eq:stability_RKCDG} for the RKCDG method to \eqref{eq:error_rk_source}, and using \eqref{eq:estimate_source}, we obtain
$$\|\boldsymbol{\zeta}_{\sigma}^{n+1}\| 
\leq 
(1+C\tau)\|\boldsymbol{\zeta}_{\sigma}^{n}\|+C\tau(h^{k+1}+\tau^{r}).$$
Applying the discrete Gr\"ownwall's inequality yields
\begin{equation}
\|\boldsymbol{\zeta}_{\sigma}^n\|\leq C (\|\boldsymbol{\zeta}_{\sigma}^0\|+h^{k+1}+\tau^r).
\end{equation}
Since $\boldsymbol{\zeta}_{\sigma}^0 = \bm{u}_\sigma^0 - \bm{U}^0 + \bm{\eta}^0$, and noting that $\|\bm{u}_\sigma^0 - \bm{U}^0\| \le Ch^{k+1}$ and the approximation properties of $P$, we find
\begin{equation}\label{eq:initial_zeta}
    \|\boldsymbol{\zeta}_{\sigma}^0\|\le \|\bm{u}_\sigma^0 - \bm{U}^0\|+\|\bm{\eta}^0\| \le Ch^{k+1}.
\end{equation}
Therefore,
$$\|\boldsymbol{u}_{\sigma}^n-\boldsymbol{U}^n\|\leq \|\boldsymbol{\zeta}_{\sigma}^n\|+\|\boldsymbol{\eta}^n\|\leq C(h^{k+1}+\tau^r).$$
This completes the proof under the assumption \eqref{eq:assumption}.
The proof of assumption \eqref{eq:assumption} will be provided at the end of this subsection.

\subsection*{\em Proof of estimate \eqref{eq:estimate_source}}
We now supplement the  proof of the estimate \eqref{eq:estimate_source} 
under the assumption \eqref{eq:assumption}.
For clarity in the analysis, we decompose $\mathcal{G}^{n,l}(\boldsymbol{\varphi}_h)$ into the following two parts:
$$\mathcal{G}_1^{n,l+1}(\boldsymbol{\varphi}_h):= 
-\langle \bm{\rho}^{n,l}, \bm{\varphi}_h \rangle+
\frac{1}{\tau}\langle\boldsymbol{\eta}^{n,l+1}-\sum_{0\leq m\leq l}c_{lm} \boldsymbol{\eta}^{n,m}, \boldsymbol{\varphi}_h\rangle
-\sum_{0\leq m\leq l}d_{lm}\mathcal{H}(\boldsymbol{\eta}^{n,m},\boldsymbol{\varphi}_h)
$$
and
$$\mathcal{G}_2^{n,l+1}(\boldsymbol{\varphi}_h):= \frac{1}{\tau}\langle \boldsymbol{u}_\sigma^{n,l+1}- \boldsymbol{u}_h^{n,l+1}, \boldsymbol{\varphi}_h\rangle.$$
The following lemma provides the estimate for $\mathcal{G}_1^{n,l+1}$.

\begin{lemma}\label{G1}
    $|\mathcal{G}_1^{n,l+1}(\boldsymbol{\varphi}_h)|
    \leq C(h^{k+1}+\tau^r)\left\|\boldsymbol{\varphi}_h\right\| \quad \forall \boldsymbol{\varphi}_h\in V_h\times W_h.$
\end{lemma}
\begin{proof}
Using \eqref{eq:rho}, we have $\langle \bm{\rho}^{n,l}, \bm{\varphi}_h \rangle \le C\tau^r \|\boldsymbol{\varphi}_h\|$.
Based on the definition in \eqref{eq:reference_u}, it follows that
$$
\bm{W}^{n,l}:=
\frac{1}{\tau}\big(\boldsymbol{U}^{n,l+1}-\sum_{0\leq m\leq l}c_{lm}\boldsymbol{U}^{n,m}\big)
=
-\sum_{0\leq m\leq l}
d_{lm} \bm{\beta}\cdot \nabla \bm{U}^{n,m}
+\bm{\rho}^{n,l},
$$
where $\bm{W}^{n,l}$ does not depend on $\tau$.
By the approximation properties of $P$, as given in \eqref{eq:projection_approximation}, we obtain
$
\|\tau^{-1}(\bm{\eta}^{n,l+1}-\sum_{0\leq m\leq l}c_{lm}\boldsymbol{\eta}^{n,m})\| 
=\|\bm{W}^{n,l} - P\bm{W}^{n,l}\|\leq Ch^{k+1}
$.

We can extend Lemma 3.3 in \cite{liu2018optimal} as follows: for any $\bm{\varphi} \in V_h \times W_h$, 
$\mathcal{H}(\bm{\eta}^n,\boldsymbol{\varphi}) = 0$
if $u^n \in \mathbb{P}^{k+1}(\Omega)$. 
By decomposing $u^{n,m}$ into $Tu^{n,m}+Ru^{n,m}$ through a Taylor expansion, where $Tu^{n,m}$ represents the sum of polynomials 
of degree at most $k+1$, and $Ru^{n,m}$ is the remainder term, we have
\begin{align}
|\mathcal{H}(\boldsymbol{\eta}^{n,m},\bm{\varphi}_h)|
\le &
|\mathcal{H}(T\bm{U}^{n,m}-PT\bm{U}^{n,m},\bm{\varphi}_h)|+
|\mathcal{H}(R\bm{U}^{n,m}-PR\bm{U}^{n,m},\bm{\varphi}_h)|\nonumber\\
=&
|\mathcal{H}(R\bm{U}^{n,m}-PR\bm{U}^{n,m},\bm{\varphi}_h)| \label{eq:superconvergence_H}\\
\le &
Ch^{-1}\|R\bm{U}^{n,m}-PR\bm{U}^{n,m}\|\|\bm{\varphi}_h\|
\le Ch^{k+1}\|\bm{\varphi}_h\|.\nonumber
\end{align}
Thus, we conclude that
$\big|\sum_{0\leq m\leq l}d_{lm}\mathcal{H}(\bm{\eta}^{n,m},\bm{\varphi}_h)\big|
\le Ch^{k+1}\|\bm{\varphi}_h\big\|.$
This completes the proof.
\end{proof}

The second term, $\mathcal{G}_2^{n,l+1}(\boldsymbol{\varphi}_h)$, is more challenging to estimate, and we approach it through an inductive process. First, by integrating $\hat{t}$ from 0 to $\tau$ in the damping ODE \eqref{OEstep2D}, we obtain 
$$
\left\{
\begin{aligned}
    \int_{C} ((u_{\sigma}^C)^{n,l}-&(u_{h}^C)^{n,l})\varphi_1\mathrm{d}\boldsymbol{x}= \\
    &-\sum_{m=0}^k \delta_C^m(\bm{u}_h^{n,l})\int_0^\tau \int_{C}\big((u_{\sigma}^C)^{n,l}(\zeta)-P^{m-1}(u_{\sigma}^C)^{n,l}(\zeta)\big)\varphi_1\mathrm{d}\bm{x}\mathrm{d}\zeta,\\
    \int_{D} ((u_{\sigma}^D)^{n,l}-&(u_{h}^D)^{n,l})\varphi_2\mathrm{d}\boldsymbol{x}= \\
    &-\sum_{m=0}^k \delta_D^m(\bm{u}_h^{n,l})\int_0^\tau \int_{D}\big((u_{\sigma}^D)^{n,l}(\zeta)-P^{m-1}(u_{\sigma}^D)^{n,l}(\zeta)\big)\varphi_2\mathrm{d}\bm{x}\mathrm{d}\zeta.
\end{aligned}
\right.
$$
Taking $\varphi_1 = (u_{\sigma}^C)^{n,l}-(u_{h}^C)^{n,l}$,
$\varphi_2 =(u_{\sigma}^D)^{n,l}-(u_{h}^D)^{n,l}$,
and applying the Cauchy--Schwarz inequality, we have
\begin{equation*}
    \left\{\begin{aligned}
    \left\|(u_{\sigma}^C)^{n,l}-(u_{h}^C)^{n,l}\right\|_{L^2(C)} 
    &\leq \sum_{m=0}^k \delta_C^m(\bm{u}_h^{n,l})\int_0^{\tau}\left\|(u_{\sigma}^C)^{n,l}(\zeta)-P^{m-1}(u_{\sigma}^C)^{n,l}(\zeta)\right\|_{L^2(C)}\mathrm{d}\zeta,
    \\
    \left\|(u_{\sigma}^D)^{n,l}-(u_{h}^D)^{n,l}\right\|_{L^2(D)} 
    &\leq \sum_{m=0}^k \delta_D^m(\bm{u}_h^{n,l})\int_0^{\tau}\left\|(u_{\sigma}^D)^{n,l}(\zeta)-P^{m-1}(u_{\sigma}^D)^{n,l}(\zeta)\right\|_{L^2(D)}\mathrm{d}\zeta.
\end{aligned}\right.
\end{equation*}

It is important to note that when the damping ODE \eqref{OEstep2D} is rewritten as an ODE system for the coefficients of the Legendre basis, the magnitudes of these coefficients do not increase; see equation (2.12) in \cite{peng2023oedg}. This observation implies that, for $K = C$ or $D$ and $\zeta\in (0,\tau]$, we have 
 $$\begin{aligned}\left\|(u_{\sigma}^{K})^{n,l}(\zeta)-P^{m-1}(u_{\sigma}^{K})^{n,l}(\zeta)\right\|_{L^2(K)}&\leq \left\|(u_{\sigma}^{K})^{n,l}(0)-P^{m-1}(u_{\sigma}^{K})^{n,l}(0)\right\|_{L^2(K)} \\&= \left\|(u_{h}^{K})^{n,l}-P^{m-1}(u_{h}^{K})^{n,l}\right\|_{L^2(K)}.\end{aligned}$$ 
Thus, 
\begin{equation}\label{eq:estimate}
\left\{
\begin{aligned}
    \left\|(u_{\sigma}^C)^{n,l}-(u_{h}^C)^{n,l}\right\|_{L^2(C)}&\leq 
    \tau\sum_{m=0}^k \delta_C^m(\bm{u}_h^{n,l})\left\|(u_{h}^C)^{n,l}-P^{m-1}(u_{h}^C)^{n,l}\right\|_{L^2(C)},
    \\
    \left\|(u_{\sigma}^D)^{n,l}-(u_{h}^D)^{n,l}\right\|_{L^2(D)}&\leq 
    \tau\sum_{m=0}^k \delta_D^m(\bm{u}_h^{n,l})\left\|(u_{h}^D)^{n,l}-P^{m-1}(u_{h}^D)^{n,l}\right\|_{L^2(D)},
\end{aligned}\right.
\end{equation}
which provides a basis for the following estimates of 
$\left\|(u_{h}^K)^{n,l}-P^{m-1}(u_{h}^K)^{n,l}\right\|_{L^2(K)}$
and $\delta_K^m(\bm{u}_h^{n,l})$
for $K=C$ or $D$.

\begin{lemma}\label{prop2}
    If $\left\|\boldsymbol{\zeta}^{n,l}\right\|_{L^{\infty}(\Omega)} \leq h$, then for $0\le m \le k$, we have
    $$
    \|(u_h^K)^{n,l}-P^{m-1}(u_h^K)^{n,l}\|_{L^2(K)}\leq Ch^{1+\frac{d}{2}}.$$
\end{lemma}

\begin{proof}
    This lemma follows by applying Lemma 4.14 from \cite{peng2023oedg} to the components of $\boldsymbol{u}_h^{n,l}$.
\end{proof}

\begin{lemma}\label{lem_avg}
    If $\|\boldsymbol{\zeta}^{n,l}\|_{L^{\infty}(\Omega)} \leq h$, then there exists a constant $C_{**,l}>0$, independent of $n$, such that
    \begin{equation*}
        \|(u_h^K)^{n,l}-{\rm avg}((u_h^K)^{n,l})\|_{L^{\infty}(\Omega)}
        \geq \frac{1}{2}\|(u^K_h)^0-{\rm avg}((u^K_h)^0)\|_{L^{\infty}(\Omega)}>0
    \end{equation*}
    provided $h \le C_{**,l}$.
\end{lemma}

\begin{proof}
    Since we have assumed that $\text{avg}((u_h^C)^{0}) = \text{avg}((u_h^D)^{0})$, it can be shown that for any $n\ge0$ and $0\le l \le s$, the average satisfies ${\rm avg}((u_h^K)^{n,l}) = {\rm avg}((u_h^K)^0)$. This result follows by taking $\bm{\varphi}_h = [1,0]^\top$ or $[0,1]^\top$ in \eqref{eq:RKCDG}. Applying the proof technique of Lemma 4.12 in \cite{peng2023oedg} yields the desired conclusion.
\end{proof}

Since our new damping coefficient $\delta_{K}$ differs from that used in the OEDG method \cite{peng2023oedg}, we must establish the following lemma to estimate $|\delta_{K}^m(\bm{u}_h^{n,l})|$.

\begin{lemma}\label{lem_delta}
    If $\|\boldsymbol{\zeta}^{n,l}\|_{L^{\infty}(\Omega)} \leq h$, then there exists a constant $C_{**,l} > 0$, independent of $n$, such that
    \begin{align*}
        |\delta_{K}^m(\bm{u}_h^{n,l})| &\leq C\left(h^{-1-\frac{d}{2}}\|\bm{\zeta}^{n,l}\|_{L^2(K)} + h^{k-\frac{d}{2}}\|u^{n,l}\|_{H^{k+1}(K)}\right),
    \end{align*}
    provided $h \le C_{**,l}$.
\end{lemma}

\begin{proof} 
Without loss of generality, we prove this lemma only for the case $K \in \mathcal{T}_h^C$. Note that
\begin{equation}\label{eq:estimate_0}
\begin{aligned}
&|\delta_{K}^m(\bm{u}_h^{n,l})|
=\bigg|\sum_{e\in \partial K} \beta_{e,K}\frac{(2m+1)h_{e,K}^{m-1}}{(2k-1)m!\sigma_K(\bm{u}_h)|e|}
\int_e\sqrt{\sum_{|\bm{\alpha}|=m}\big(\partial^{\bm{\alpha}}((u_h^C)^{n,l}-(u_h^D)^{n,l})\big)^2} \mathrm{d}s\bigg|\\
\le &\,C\sum_{e\in \partial K} h^{m-1}  |e|^{-\frac{1}{2}}
\bigg\|\sum_{|\bm{\alpha}|=m}\partial^{\bm{\alpha}}((u_h^C)^{n,l}-(u_h^D)^{n,l})\bigg\|_{L^2(e)}\\
\le &\,C 
\sum_{|\bm{\alpha}|=m} h^{m-\frac{d+1}{2}}
\Big(
\big\|\partial^{\bm{\alpha}}((u_h^C)^{n,l}-u^{n,l})\big\|_{L^2(\partial K)}+
\big\|\partial^{\bm{\alpha}}((u_h^D)^{n,l}-u^{n,l})\big\|_{L^2(\partial K)}
\Big),
\end{aligned}
\end{equation}
where we have used Lemma \ref{lem_avg}, the fact $|\beta_{e,K}|\le C$,
and the Cauchy--Schwarz inequality
 in the first inequality, 
and triangular inequality in the second inequality. 
For the first term, we observe that
\begin{equation*}
\begin{aligned}
\big\|\partial^{\bm{\alpha}}((u_h^C)^{n,l}-u^{n,l})\big\|_{L^2(\partial K)}
& =\big\|\partial^{\bm{\alpha}}((\zeta^C)^{n,l}-(\eta^C)^{n,l})\big\|_{L^2(\partial K)}\\
&\le\, 
\big\|\partial^{\bm{\alpha}}(\zeta^C)^{n,l} \big\|_{L^2(\partial K)}
+
\big\|\partial^{\bm{\alpha}}(\eta^C)^{n,l} \big\|_{L^2(\partial K)},
\end{aligned}
\end{equation*}
where $\zeta^C$ and $\eta^C$ are the first components of $\bm{\zeta}$ and $\bm{\eta}$, respectively.
Employing the inverse inequality gives 
\begin{equation*}
\big\|\partial^{\bm{\alpha}}(\zeta^C)^{n,l} \big\|_{L^2(\partial K)}
\le C h^{-m-\frac{1}{2}}
\big\|(\zeta^C)^{n,l} \big\|_{L^2(K)}.
\end{equation*}
Applying the multiplicative trace inequality and the approximation property of $P^*_h$, we obtain
\begin{equation*}
\begin{aligned}
\big\|\partial^{\bm{\alpha}}(\eta^C)^{n,l} \big\|_{L^2(\partial K)}^2
\le 
&C \big\|\partial^{\bm{\alpha}}(\eta^C)^{n,l} \big\|_{L^2(K)}
\big(
\big\|\partial^{\bm{\alpha}}(\eta^C)^{n,l}\big\|_{H^1(K)}
+ h^{-1}
\big\|\partial^{\bm{\alpha}}(\eta^C)^{n,l}\big\|_{L^2(K)}
\big)\\
\le&
C h^{2k+1-2m} 
\big\|u^{n,l}\big\|_{H^{k+1}(K)}^2.
\end{aligned}
\end{equation*}
Therefore,
\begin{equation}\label{eq:estimate_1}
\big\|\partial^{\bm{\alpha}}((u_h^C)^{n,l}-u^{n,l})\big\|_{L^2(\partial K)}
\le 
C h^{-m-\frac{1}{2}}
\left(
\big\|(\zeta^C)^{n,l} \big\|_{L^2(K)}+
h^{k+1} 
\big\|u^{n,l}\big\|_{H^{k+1}(K)}^2\right).
\end{equation}
Similarly,  it can be shown that
\begin{equation}\label{eq:estimate_2}
\big\|\partial^{\bm{\alpha}}((u_h^D)^{n,l}-u^{n,l})\big\|_{L^2(\partial K)}
\le 
C h^{-m-\frac{1}{2}}
\left(
\big\|(\zeta^D)^{n,l} \big\|_{L^2(K)}+
h^{k+1} 
\big\|u^{n,l}\big\|_{H^{k+1}(K)}^2\right).
\end{equation}
It is worth mentioning that although $(u_h^D)^{n,l}$ is a piecewise polynomial on cell $K \in \mathcal{T}_h^C$, we can decompose $K=[x_{i-\frac{1}{2}},x_{i+\frac{1}{2}}]\times [y_{j-\frac{1}{2}},y_{j+\frac{1}{2}}]$ into four parts, i.e.,
\begin{equation*}
K = \bigcup_{s_1,s_2 \in \{0,\frac{1}{2}\}}
[x_{i-\frac{1}{2}+s_1},x_{i+s_1}]\times
[y_{j-\frac{1}{2}+s_2},y_{j+s_2}],
\end{equation*}
and prove \eqref{eq:estimate_2}.
Finally, substituting \eqref{eq:estimate_1} and \eqref{eq:estimate_2} into \eqref{eq:estimate_0}, and noting that for $K=C$ or $D$, $\big\|(\zeta^K)^{n,l} \big\|_{L^2(K)} \le 
\big\|\bm{\zeta}^{n,l} \big\|_{L^2(K)}$,
the proof is completed.
\end{proof}

By combining Lemma \ref{prop2}, Lemma \ref{lem_delta}, 
and equations \eqref{eq:estimate}, 
we arrive at the following lemma.

\begin{lemma}\label{lem:u_sigma_u_h}
If $\|\boldsymbol{\zeta}^{n,l}\|_{L^{\infty}(\Omega)}\leq h$,
there exists a constant $C_{**,l}>0$, independent of $n$,
such that
\begin{align}\label{G3_1}
    \|\bm{u}_{\sigma}^{n,l}-\bm{u}_{h}^{n,l}\|
    \le C\tau\big(\|\boldsymbol{\zeta}^{n,l}\|+h^{k+1}\big)
\end{align}
provided $h\le C_{**,l}$.
\end{lemma}
	
\begin{lemma}
    \label{zeta_zetasigma}
    $\left\|\boldsymbol{\zeta}^{n,l+1}\right\|\leq C\sum_{m=0}^{l}\left\|\boldsymbol{\zeta}_{\sigma}^{n,m}\right\|+C\tau(h^{k+1}+\tau^r),\quad \forall 0\leq l \leq s-1.$
\end{lemma}
\begin{proof}
    The proof is similar to that of Lemma 4.11 in 
    \cite{peng2023oedg}, and is omitted here.
\end{proof}

Since $\boldsymbol{\zeta}_{\sigma}^{n,l+1} = \boldsymbol{\zeta}^{n,l+1}+\boldsymbol{u}_{\sigma}^{n,l+1}-\boldsymbol{u}_{h}^{n,l+1}$, by the triangle inequality, we have
\begin{equation*}
\|\boldsymbol{\zeta}_{\sigma}^{n,l+1}\| \leq \|\boldsymbol{\zeta}^{n,l+1}\|+\|\boldsymbol{u}_{\sigma}^{n,l+1}-\boldsymbol{u}_{h}^{n,l+1}\|,
\quad \text{for}\ 0 \leq l \leq s-1.
\end{equation*}
Then, by \Cref{lem:u_sigma_u_h} and \Cref{zeta_zetasigma}, an inductive argument (by setting $l=0,1,\cdots,s-1$) allows us to derive that
\begin{equation}\label{zeta_sigma_in}
\begin{aligned}
    \|\boldsymbol{\zeta}_{\sigma}^{n,l+1}\|\leq C\|\boldsymbol{\zeta}_{\sigma}^{n}\|+C\tau(h^{k+1}+\tau^r),\quad \text{for}\ 0 \leq l \leq s-1,
\end{aligned}
\end{equation} 
and 
\begin{equation}\label{zeta_in}
    \begin{aligned}
        \|\boldsymbol{\zeta}^{n,l+1}\|\leq C\|\boldsymbol{\zeta}_{\sigma}^{n}\|+C\tau(h^{k+1}+\tau^r),\quad \text{for}\ 0 \leq l \leq s-1,
\end{aligned}\end{equation}
under the condition
$\|\boldsymbol{\zeta}^{n,l+1}\|_{L^{\infty}(\Omega)}\leq h$ for $0\le l \le s-1$.
We then obtain the following estimate for $\|\mathcal{G}_2^{n,l+1}\|$:
\begin{equation}\label{eq:G2}
    \|\mathcal{G}_2^{n,l+1}\|  \le \frac{\|\boldsymbol{u}_{\sigma}^{n,l+1}-\boldsymbol{u}_{h}^{n,l+1}\|}{\tau}\leq C(\|\boldsymbol{\zeta}_{\sigma}^n\|+h^{k+1}+\tau^r).
\end{equation}
The constant $C$ here is independent of $n$, $h$, $\tau$, and $l$. 
By combining Lemma \ref{G1} with the estimate \eqref{eq:G2},
we establish \eqref{eq:estimate_source}
under the assumption \eqref{eq:assumption}.
The proof of estimate \eqref{eq:estimate_source} is completed.

\subsection*{\em Proof of Assumption \eqref{eq:assumption}}
We now provide the proof of assumption \eqref{eq:assumption} for 
$0 \le n \le \lceil \frac{T}{\tau} \rceil$ 
and $1 \le l \le s$, as required by Step 2 in the outline of the proof of Theorem \ref{thm:optimal_error}. 
The validity of assumption \eqref{eq:assumption} is established by the following lemma.

\begin{lemma}\label{lem:assumption}
Suppose $k>\frac{d}{2}$ and $r>\frac{d}{2}+1$. For integers $n \ge 0$ and $0\le l \le s-1$, 
if there exists a constant $C_{*}$
such that
\begin{align}\label{eq:assumption_2}
\|\bm{\zeta}_\sigma^{n,m}\| \le C(h^{k+1}+\tau^r) \quad
\forall 0\le m \le l,
\end{align}
for $h\le C_{*}$,
then
$
\|\bm{\zeta}^{n,l+1}\|_{L^{\infty}(\Omega)} \le h
$
and
$
\|\bm{\zeta}_\sigma^{n,l+1}\| \le C(h^{k+1}+\tau^r)
$
for $h\le {C}_{*}^l$.
Here ${C}_{*}^l$ is a fixed constant independent of $n$.
\end{lemma}

\begin{proof}
It is worth noting that Lemma \ref{zeta_zetasigma} does not depend on assumption \eqref{eq:assumption}.
By employing Lemma \ref{zeta_zetasigma} and assumption \eqref{eq:assumption_2}, we obtain
\begin{equation}\label{eq:zeta_L2}
\|\bm{\zeta}^{n,l+1}\| \le C(h^{k+1}+\tau^r).
\end{equation}
Since the components of $\bm{\zeta}^{n,l+1}$ are piecewise polynomials, we can estimate the infinity norm using the inverse inequality:
\begin{align*}
\|\bm{\zeta}^{n,l+1}\|_{L^{\infty}(\Omega)}
&\le
\sup_{K}
\|\bm{\zeta}^{n,l+1}\|_{L^{\infty}(K)}
\le
\sup_{K}
Ch^{-\frac{d}{2}}
\|\bm{\zeta}^{n,l+1}\|_{L^{2}(K)}\\
&\le
Ch^{-\frac{d}{2}}
\|\bm{\zeta}^{n,l+1}\|.
\end{align*}
Combining this with \eqref{eq:zeta_L2} and the CFL-type condition $\tau \le C_{\rm CFL}h^{\kappa}$, we have
\begin{align*}
\|\bm{\zeta}^{n,l+1}\|_{L^{\infty}(\Omega)}
&\le 
Ch^{-\frac{d}{2}}
(h^{k+1}+(C_{\rm CFL}h^{\kappa})^r)
\le
C_{l+1}h^{\min\{k+1, r\}-\frac{d}{2}}\\
&\le h \quad \text{if } 
h \le C_{l+1}^{-1/(\min\{k,r-1\}-d/2)},
\end{align*}
where we have used the fact that $\kappa\ge 1$.
This validates assumption \eqref{eq:assumption}.
Therefore, invoking Lemma \ref{lem:u_sigma_u_h}, we have
\begin{equation*}
    \|\bm{\zeta}_{\sigma}^{n,l+1}\| 
    \le 
    \|\bm{\zeta}^{n,l+1}\|+
    \|\bm{u}_{\sigma}^{n,l+1}-\bm{u}_{h}^{n,l+1}\|
    \le 
    C(h^{k+1}+\tau^r)
\end{equation*}
if $h \le C_{*}^l:=\min\{C_{*}, C_{l+1}^{-1/(\min\{k,r-1\}-d/2)}, C_{**,l}\}$.
The proof of Lemma \ref{lem:assumption} is thus complete.
\end{proof}

With the help of Lemma \ref{lem:assumption}, we can easily prove assumption \eqref{eq:assumption} using an inductive approach. Since \eqref{eq:assumption_2} holds for $n=l=0$ (see \eqref{eq:initial_zeta}), we have
$
\|\bm{\zeta}^{0,1}\|_{L^{\infty}(\Omega)} \le h
$
and
$
\|\bm{\zeta}_\sigma^{0,1}\| \le C(h^{k+1}+\tau^r).
$
Therefore, \eqref{eq:assumption_2} is valid for $n=0$ and $l=1$. Consequently, we obtain 
$
\|\bm{\zeta}^{0,2}\|_{L^{\infty}(\Omega)} \le h
$
and
$
\|\bm{\zeta}_\sigma^{0,2}\| \le C(h^{k+1}+\tau^r).
$
By continuing this process inductively, we can show that
$$
\|\bm{\zeta}^{0,l}\|_{L^{\infty}(\Omega)} \le h
\quad
\forall 1\le l \le s.
$$
Note that $\bm{\zeta}_{\sigma}^{1,0}=\bm{\zeta}_{\sigma}^{0,s}$ by definition, and we have obtained $\|\bm{\zeta}_{\sigma}^{0,s}\| \le C(h^{k+1}+\tau^r)$ by setting $n=0$ and $l=s-1$. Thus, \eqref{eq:assumption_2} is also true for $n=1$ and $l=0$. Continuing this procedure inductively, we can similarly show that
$$
\|\bm{\zeta}^{1,l}\|_{L^{\infty}(\Omega)} \le h
\quad
\forall 1\le l \le s.
$$
Therefore, assumption \eqref{eq:assumption} is proven by induction.

    \section{Numerical examples}\label{sec:numerical}

This section presents benchmark tests in 1D and 2D to validate the optimal error estimates and the effectiveness of the OECDG method. For accuracy tests, we use the $(k+1)$th-order explicit RK time discretization for the $\mathbb{P}^k$-based OECDG method, while a third-order explicit strong-stability-preserving (SSP) RK method is used for other tests. In 1D, the time step $\tau$ is set as $\frac{C_{\rm CFL}h}{\beta}$, where $\beta$ denotes the maximum wave speed. In 2D, we use $\tau = \frac{C_{\rm CFL}}{{\beta_x}/{h_x} + {\beta_y}/{h_y}}$, with $\beta_x$ and $\beta_y$ as the wave speeds in the $x$ and $y$ directions, respectively. Unless stated otherwise, $C_{\rm CFL}$ is set to 0.52, 0.3, and 0.2 for discontinuous cases, while for accuracy tests, values of 0.45, 0.33, and 0.31 are used for $\mathbb{P}^1$, $\mathbb{P}^2$, and $\mathbb{P}^3$ elements, respectively. 

\subsection{1D linear advection equation}
We solve $u_t + u_x = 0$ on $\Omega = [0,1]$ with periodic boundary conditions.

\begin{exmp}
{\rm 
This smooth test evaluates convergence rates with the initial condition $u_0(x) = \sin(2\pi x)$. Table \ref{linear_accuracy} shows the errors and convergence rates for the $\mathbb{P}^k$-based OECDG method at $t = 1.1$, confirming $(k+1)$-order accuracy across three norms, which aligns with our theoretical error estimates.} 

\begin{table}[!tbh]
\centering
\caption{Errors and convergence rates of $\mathbb{P}^k$-based OECDG for 1D advection equation.}
\begin{center}
\begin{tabular}{c|c|c|c|c|c|c|c}
\bottomrule[1.0pt]
$k$ & $N_x$ & $l^1$ errors  & rate & $l^2$ errors & rate & $l^\infty$ errors & rate \\
\hline
\multirow{5}*{$1$}
& 80  & 1.82e-3 & -- & 8.97e-4 & -- & 9.33e-4 & -- \\ 
& 160 & 4.01e-4 & 2.18 & 1.96e-4 & 2.19 & 2.02e-4 & 2.21 \\ 
& 320 & 9.66e-5 & 2.06 & 4.71e-5 & 2.06 & 4.76e-5 & 2.08 \\ 
& 640 & 2.38e-5 & 2.02 & 1.17e-5 & 2.01 & 1.18e-5 & 2.01 \\ 
& 1280 & 5.92e-6 & 2.01 & 2.91e-6 & 2.00 & 2.93e-6 & 2.01 \\ 
\hline
\multirow{5}*{$2$}
& 80  & 1.16e-5 & -- & 5.41e-6 & -- & 4.93e-6 & -- \\ 
& 160 & 1.15e-6 & 3.33 & 5.57e-7 & 3.28 & 5.46e-7 & 3.18 \\ 
& 320 & 1.27e-7 & 3.17 & 6.20e-8 & 3.17 & 6.02e-8 & 3.18 \\ 
& 640 & 1.52e-8 & 3.07 & 7.39e-9 & 3.07 & 7.12e-9 & 3.08 \\ 
& 1280 & 1.87e-9 & 3.02 & 9.11e-10 & 3.02 & 8.85e-10 & 3.01 \\ 
\hline
\multirow{4}*{$3$}    
& 80  & 5.26e-8 & -- & 2.66e-8 & -- & 3.19e-8 & -- \\ 
& 160 & 2.40e-9 & 4.45 & 1.31e-9 & 4.35 & 1.57e-9 & 4.35 \\ 
& 320 & 1.36e-10 & 4.14 & 7.52e-11 & 4.12 & 8.53e-11 & 4.20 \\ 
& 640 & 8.41e-12 & 4.01 & 4.59e-12 & 4.04 & 4.94e-12 & 4.11 \\ 
& 1280 & 5.27e-13 & 4.00 & 2.85e-13 & 4.01 & 2.99e-13 & 4.05 \\ 
\toprule[1.0pt]
\end{tabular}
\label{linear_accuracy}
\end{center}
\end{table}
\end{exmp}

\begin{exmp}\label{linear_discontinuity}
{\rm 
This discontinuous test case is initialized with 
$
u_0(x) =  
    \sin(2\pi x)$ for $x \in [0.3,0.8]$ and 
    $\cos(2\pi x) - 0.5$ for $x \notin [0.3,0.8]$, 
using 256 uniform cells to demonstrate our method's capability to resolve discontinuities without spurious oscillations. Figure \ref{Fig: 3.2a} presents the oscillation-free results, confirming that the OECDG method accurately captures discontinuities and that higher-order elements enhance resolution. Moreover, the new, more compact, dual damping mechanism provides slightly better resolution than the original one from the non-central OEDG method \cite{peng2023oedg}.
}

\begin{figure}[!tbh]
\centering 
\begin{subfigure}{0.325\linewidth}
\includegraphics[width=1\linewidth]{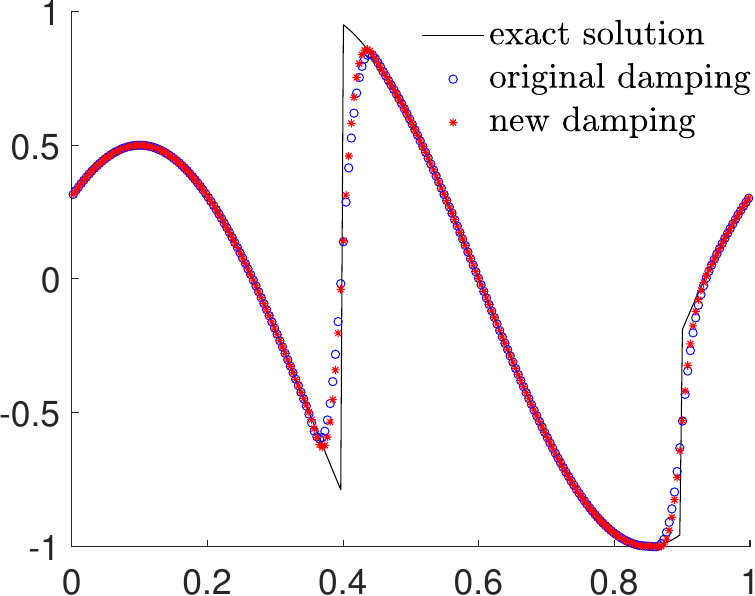}
\end{subfigure}
\begin{subfigure}{0.325\linewidth}
\includegraphics[width=1\linewidth]{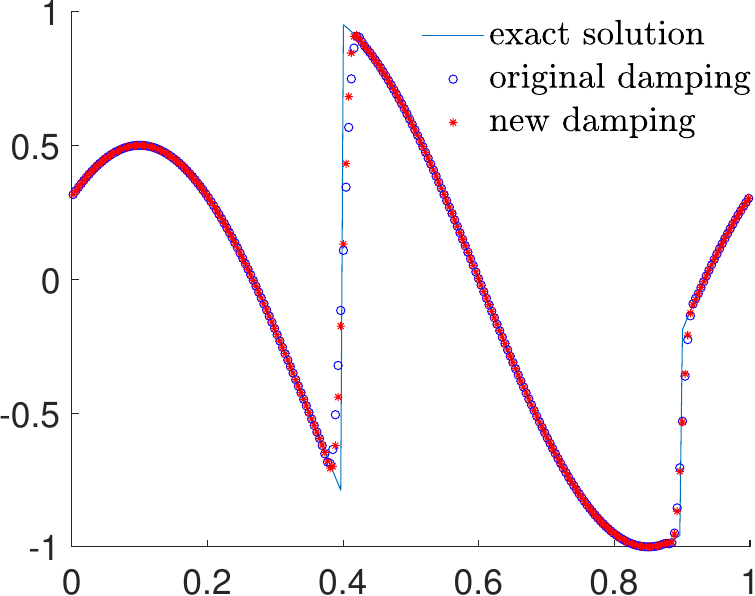}
\end{subfigure}
\begin{subfigure}{0.325\linewidth}
\includegraphics[width=1\linewidth]{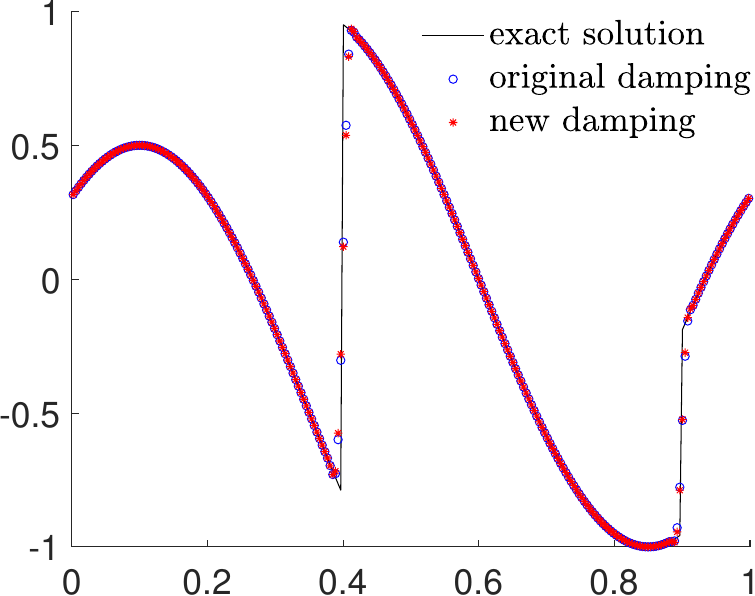}
\end{subfigure}
\caption{Results of the $\mathbb{P}^1$, $\mathbb{P}^2$, and $\mathbb{P}^3$ OECDG methods (from left to right) at $t=1.1$.}
\label{Fig: 3.2a}
\end{figure}
\end{exmp}

\subsection{1D inviscid Burgers equation}
Next, we solve the Burgers equation $u_t + (\frac{u^2}{2})_x = 0$ with $u(x,0)=\sin(x)+0.5$ on the periodic domain $\Omega = [0,2\pi]$.

\begin{exmp}
{\rm 
This test runs the simulation up to $t = 0.6$, while the solution remains smooth. Table \ref{burgers} presents the errors and convergence rates, confirming that the $\mathbb{P}^k$-based OECDG method achieves $(k+1)$th-order accuracy even with nonlinear flux.
}
\begin{table}[!tbh]
\centering
\caption{Errors and convergence rates of $\mathbb{P}^k$-based OECDG for 1D Burgers equation.}
\begin{center}
\begin{tabular}{c|c|c|c|c|c|c|c} 
\bottomrule[1.0pt]
$k$ & $N_x$ & $l^1$ errors & rate & $l^2$ errors & rate & $l^{\infty}$ errors & rate \\
\hline
\multirow{5}*{$1$}
& 128 & 7.38e-4 & -- & 5.90e-4 & -- & 1.83e-3 & -- \\
& 256 & 1.66e-4 & 2.15 & 1.30e-4 & 2.18 & 3.91e-4 & 2.23 \\
& 512 & 4.05e-5 & 2.03 & 3.12e-5 & 2.06 & 9.17e-5 & 2.09 \\
& 1024 & 9.99e-6 & 2.02 & 7.55e-6 & 2.04 & 2.16e-5 & 2.08 \\
& 2048 & 2.48e-6 & 2.01 & 1.86e-6 & 2.02 & 5.25e-6 & 2.04 \\
\hline
\multirow{5}*{$2$}
& 128 & 7.55e-6 & -- & 7.61e-6 & -- & 2.29e-5 & -- \\
& 256 & 9.11e-7 & 3.05 & 8.97e-7 & 3.09 & 2.58e-6 & 3.15 \\
& 512 & 1.14e-7 & 3.00 & 1.12e-7 & 3.00 & 3.52e-7 & 2.87 \\
& 1024 & 1.43e-8 & 2.99 & 1.41e-8 & 2.99 & 4.62e-8 & 2.93 \\
& 2048 & 1.79e-9 & 3.00 & 1.77e-9 & 2.99 & 5.94e-9 & 2.96 \\
\hline
\multirow{5}*{$3$}    
& 128 & 1.14e-7 & -- & 1.89e-7 & -- & 7.91e-7 & -- \\
& 256 & 5.16e-9 & 4.46 & 8.40e-9 & 4.49 & 3.76e-8 & 4.40 \\
& 512 & 2.69e-10 & 4.26 & 4.25e-10 & 4.30 & 1.95e-9 & 4.27 \\
& 1024 & 1.54e-11 & 4.12 & 2.38e-11 & 4.16 & 1.09e-10 & 4.16 \\
& 2048 & 1.04e-12 & 3.89 & 1.41e-12 & 4.07 & 6.65e-12 & 4.04 \\
\toprule[1.0pt]
\end{tabular}
\label{burgers}
\end{center}
\end{table}
\end{exmp} 

\begin{exmp}\label{burgers_dis}
{\rm 
To evaluate the OECDG method’s capability in handling discontinuities, we run the simulation up to $t = 2.2$ on a mesh of 256 uniform cells. The results are presented in Figure \ref{Fig: 3.1}. The OECDG method with the original damping mechanism from \cite{peng2023oedg} shows a minor overshoot, particularly in the $\mathbb{P}^3$ solution, while the new dual damping mechanism prevents spurious overshoots and oscillations.
}
\begin{figure}[!tbh]
\centering 
\begin{subfigure}{0.325\linewidth}
\includegraphics[width=1\linewidth]{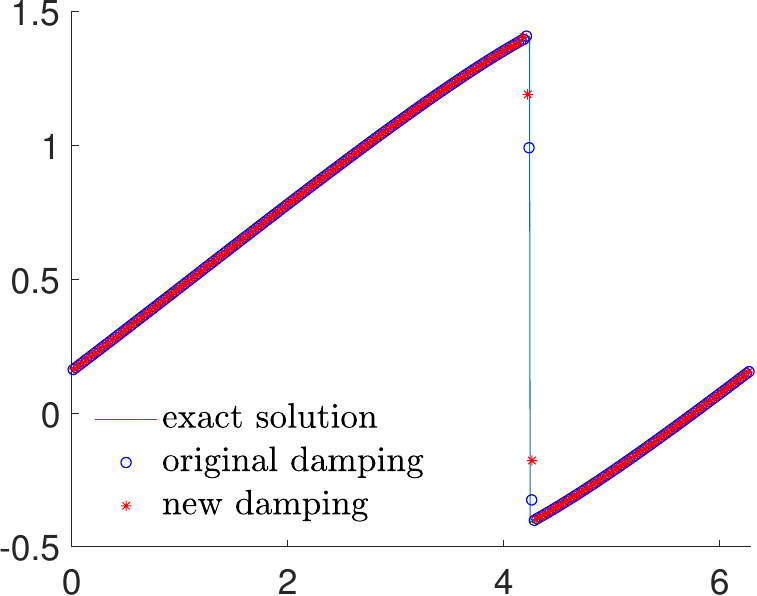}
\end{subfigure}
\begin{subfigure}{0.325\linewidth}
\includegraphics[width=1\linewidth]{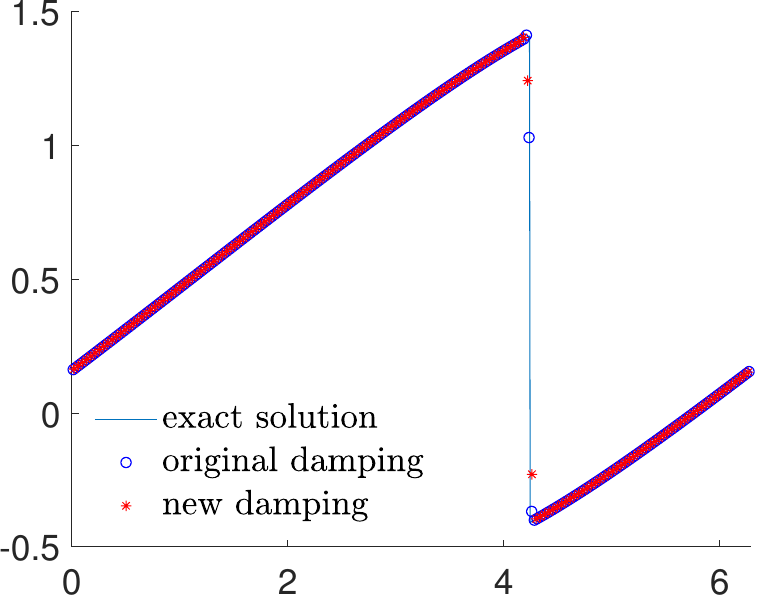}
\end{subfigure}
\begin{subfigure}{0.325\linewidth}
\includegraphics[width=1\linewidth]{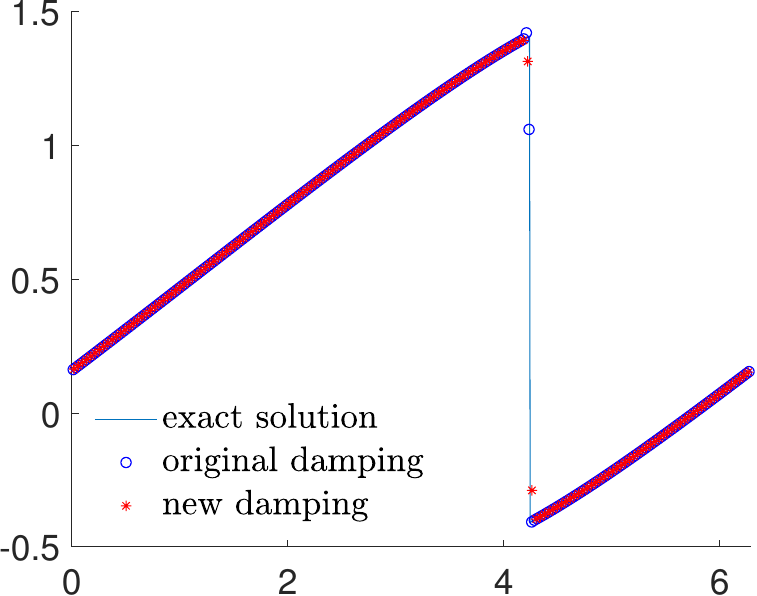}
\end{subfigure}
\caption{OECDG solutions for 1D Burgers equation. Form left to right:  $\mathbb{P}^1$, $\mathbb{P}^2$, and $\mathbb{P}^3$.}
\label{Fig: 3.1}
\end{figure}
\end{exmp}

\subsection{Lighthill--Whitham--Richards traffic flow model \cite{lu2008explicit}}
In this traffic flow model, $f(u)$ is given by piecewise quadratic polynomials as follows:
\begin{itemize}
    \item For $u \in [0, 50]$, $f(u) = -0.4u^2 + 100u$,
    \item For $u \in [50, 100]$, $f(u) = -0.1u^2 + 15u + 3500$,
    \item For $u \in [100, 350]$, $f(u) = -0.024u^2 - 5.2u + 4760$.
\end{itemize}

\begin{exmp}\label{trafficflow}
{\rm 
This test simulates traffic flow on a 20 km one-way road affected by an incident, traffic lights, and entry control measures. Initially, vehicles are concentrated between 10 km and 15 km with a density of 350 vehicles/km due to an accident. To the left, traffic flows at 50 vehicles/km, while to the right, density is modeled as $70(20 - x)$. For the first 10 minutes, vehicle entry is restricted, causing congestion. After that, entry resumes with a higher density of 75 vehicles/km for 20 minutes, and finally returns to 50 vehicles/km. Traffic lights at the exit alternate between green (0 vehicles/km) for two minutes and red (350 vehicles/km) for one minute. 
For the simulation, the road is divided into 800 uniform sections. Figure \ref{fig: traffic} shows the third-order OECDG solution with $C_{\rm CFL} = 0.2$, compared to a reference solution from the local Lax--Friedrichs scheme using 80,000 cells. The OECDG method effectively resolves high-frequency waves without spurious oscillations.
}
\begin{figure}[!tbh]
    \centering 
    \begin{subfigure}{0.325\linewidth}
        \includegraphics[width=1\linewidth]{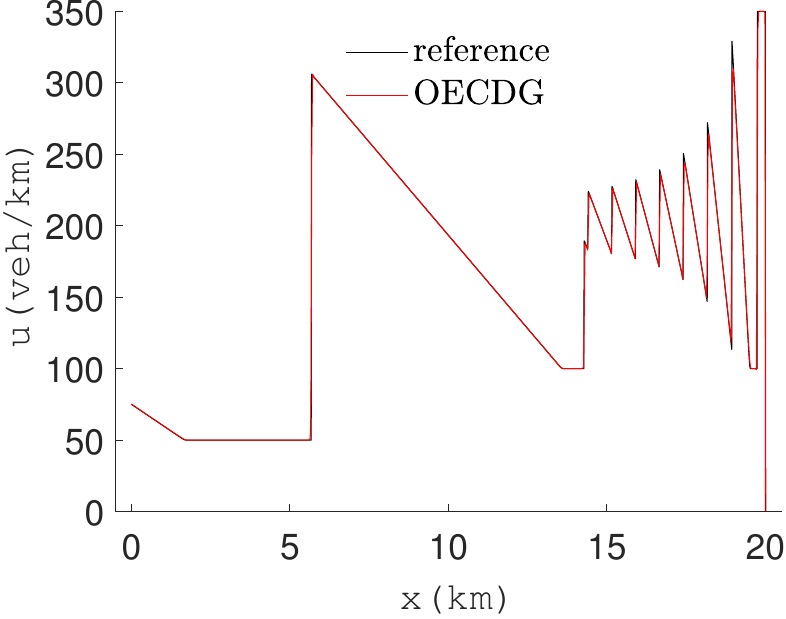}
    \end{subfigure}
    \begin{subfigure}{0.325\linewidth}
        \includegraphics[width=1\linewidth]{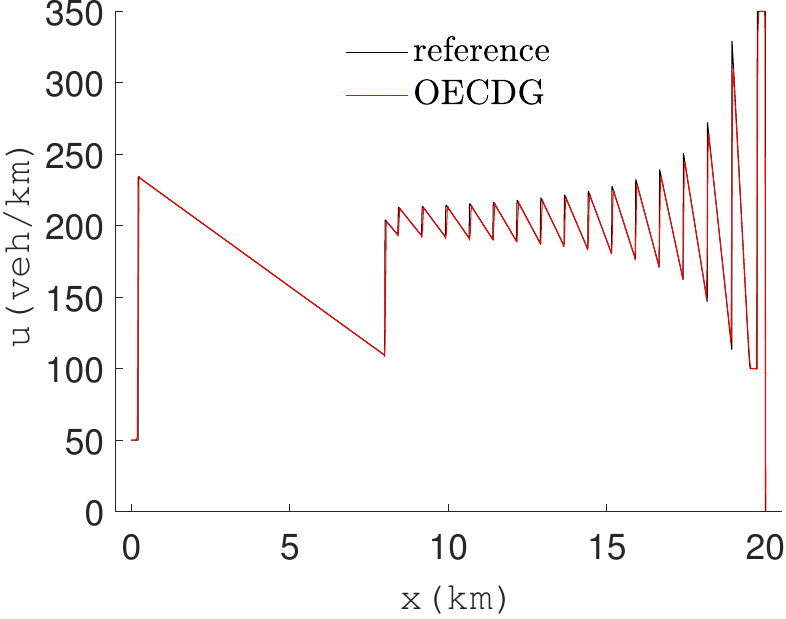}
    \end{subfigure}
    \begin{subfigure}{0.325\linewidth}
        \includegraphics[width=1\linewidth]{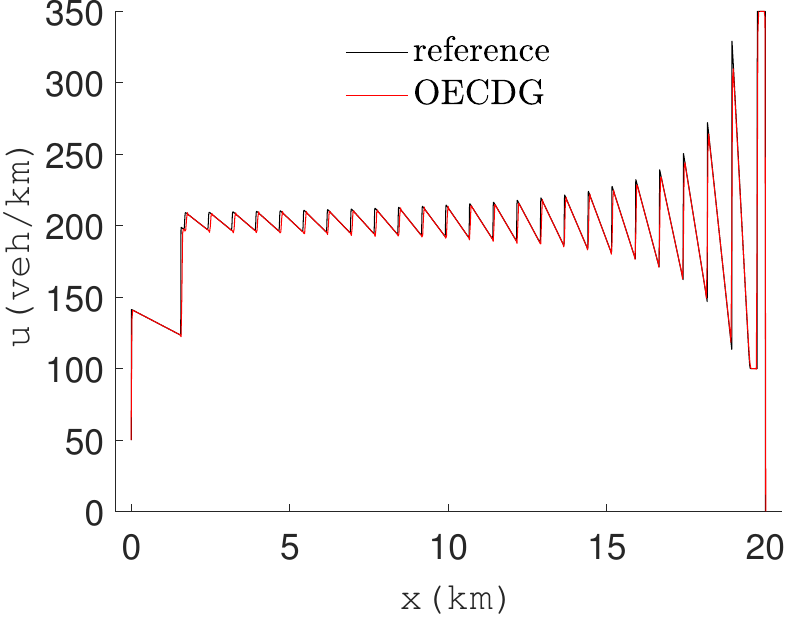}
    \end{subfigure}
    \caption{Density profiles of traffic flow at $t = 0.5h$, $1.0h$, and $1.5h$ (left to right).}
    \label{fig: traffic}	
\end{figure}
\end{exmp}

	\subsection{1D compressible Euler equations}
We consider several benchmark examples of the Euler equations ${\bm u}_t + ( {\bm f}({\bm u}) )_x={\bf 0}$, where ${\bm u}=(\rho, \rho v, E)^\top$ and ${\bm f}({\bm u}) = (\rho v, \rho v^2 + p, (E+p)v)^\top$, with $\rho$ as density, $v$ as velocity, $p$ as pressure, and $E=\frac{1}{2} \rho v^2 + \frac{p}{\gamma - 1}$ as total energy. The adiabatic index $\gamma$ is 1.4 unless otherwise specified.

\begin{exmp}[Shu--Osher problem]
\label{shu-osher}
{\rm 
The initial data includes a sinusoidal wave $(\rho_0, v_0, p_0) = (1+0.2\sin(5x), 0, 1)$ for $x > -4$ and a shock at $x = -4$ with $(\rho_0, v_0, p_0) = (3.857143, 2.629369, 10.33333)$ for $x < -4$, generating complex wave structures. This problem is solved on $\Omega = [-5, 5]$ using 200 uniform cells. 
We compare the OECDG results at $t = 1.8$ to a reference solution from the local Lax--Friedrichs scheme with 500,000 cells. As shown in Figure \ref{fig: Shuosher}, higher-order OECDG methods effectively capture intricate wave features, and the new dual damping operator offers slightly improved resolution of high-frequency waves.

\begin{figure}[!tbh]
    \centering 
    \begin{subfigure}{0.325\linewidth}
        \includegraphics[width=1\linewidth]{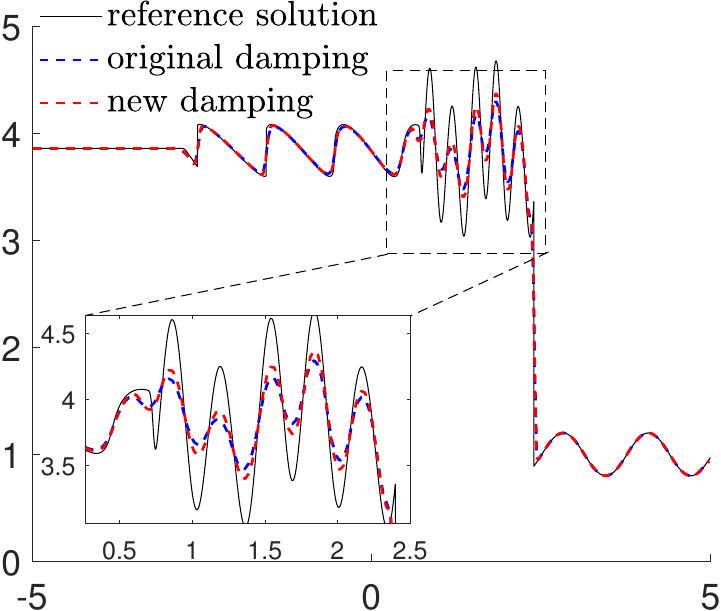}
    \end{subfigure}
    \begin{subfigure}{0.325\linewidth}
        \includegraphics[width=1\linewidth]{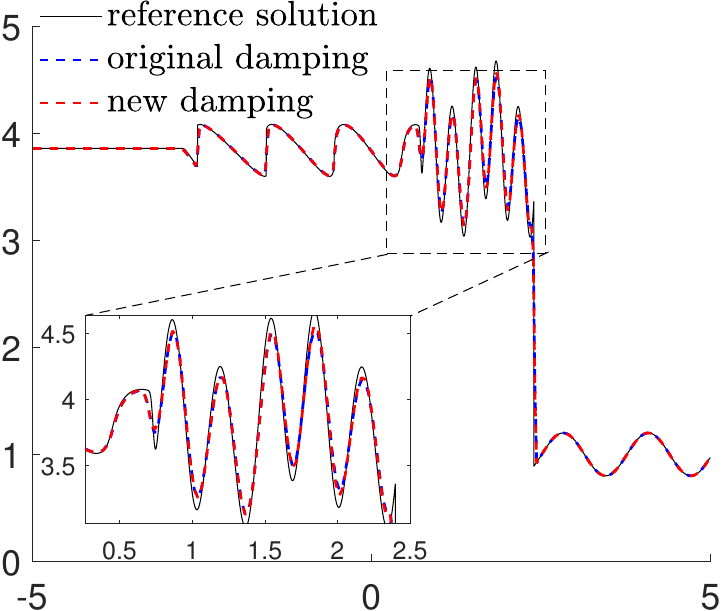}
    \end{subfigure}
    \begin{subfigure}{0.325\linewidth}
        \includegraphics[width=1\linewidth]{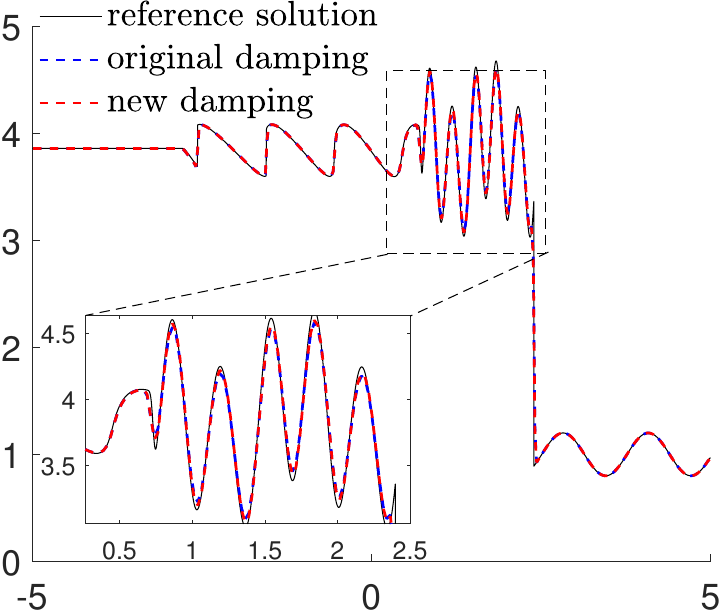}
    \end{subfigure}
    \caption{Density at $t = 1.8$ computed by $\mathbb{P}^1$, $\mathbb{P}^2$, and $\mathbb{P}^3$ OECDG (from left to right).}
    \label{fig: Shuosher}	
\end{figure}
}
\end{exmp}

\begin{exmp}[Sod problem]
\label{Sod}
{\rm
We examine the Sod problem on the domain $[-5, 5]$ with outflow boundary conditions. The initial conditions are $(\rho_0, v_0, p_0) = (0.125, 0, 0.1)$ for $x > 0$ and $(1, 0, 1)$ for $x < 0$. Figure \ref{fig: SOD} shows the density at $t = 1.3$ computed using the $\mathbb{P}^3$-based OECDG method with 256 cells, demonstrating desirable oscillation-free resolution of discontinuities.
}
\end{exmp}

\begin{exmp}[two blast waves]
\label{twoblastwave}
{\rm 
We simulate two interacting blast waves on $\Omega = [0, 1]$ with reflective boundaries. The initial conditions are $(\rho_0, v_0, p_0) = (1, 0, 10^3)$ for $0 < x < 0.1$, $(1, 0, 10^{-2})$ for $0.1 < x < 0.9$, and $(1, 0, 10^2)$ for $x > 0.9$. Figure \ref{fig: tb} shows the density at $t = 0.038$ from the $\mathbb{P}^2$-based OECDG method with 960 cells, compared to a reference solution from the local Lax--Friedrichs scheme with 300,000 cells. The OECDG method, without characteristic decomposition, achieves superior resolution compared to the OFDG method in \cite{liu2022essentially} and avoids nonphysical oscillations.
}
\end{exmp}

	\begin{exmp}[\rm Sedov blast problem \cite{zhang2010positivity}]
		\rm 
		\label{sedov}
The OECDG method’s ability to mitigate spurious numerical oscillations enables the solution of low density, low-pressure problems without requiring positivity-preserving limiters. One example is the Sedov blast problem, where initial conditions set density, velocity, and total energy to 1, 0, and $10^{-12}$, respectively, except in the central cell. The central cell has the same density and velocity but a much higher total energy, scaled by $h$ as $E = \frac{3200000}{h}$. The domain spans $[-2, 2]$ with 201 cells and outflow boundary conditions. 
Figure \ref{fig: sedov} shows the density at $t = 0.001$ computed with the fourth-order OECDG method, compared to a reference solution from the local Lax--Friedrichs scheme using 8005 cells.

		\begin{figure}[!tbh]
			
			\begin{minipage}[t]{0.9\linewidth}
				\centering 
			\begin{subfigure}{0.45\linewidth}{\includegraphics[width = 1\linewidth]{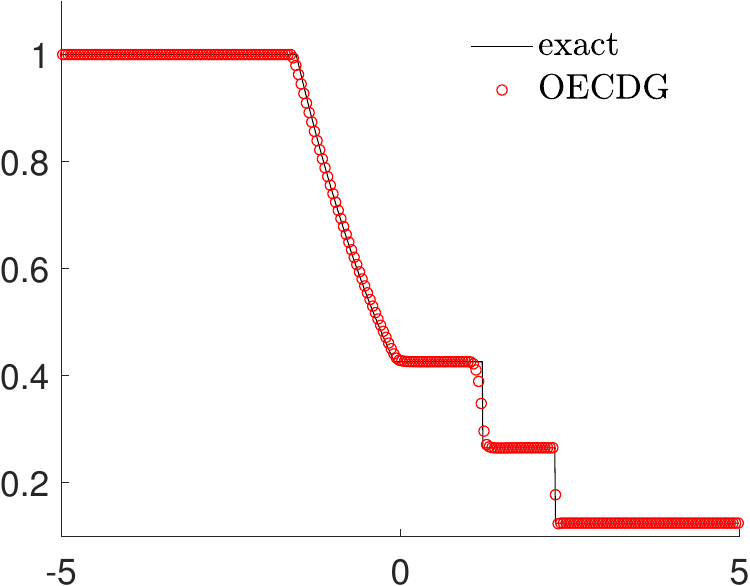}
					\subcaption{$\mathbb P^3$, Sod problem at $t = 1.3$.}
					\label{fig: SOD}	}
			\end{subfigure}
			\begin{subfigure}{0.45\linewidth}{\includegraphics[width = 1\linewidth]{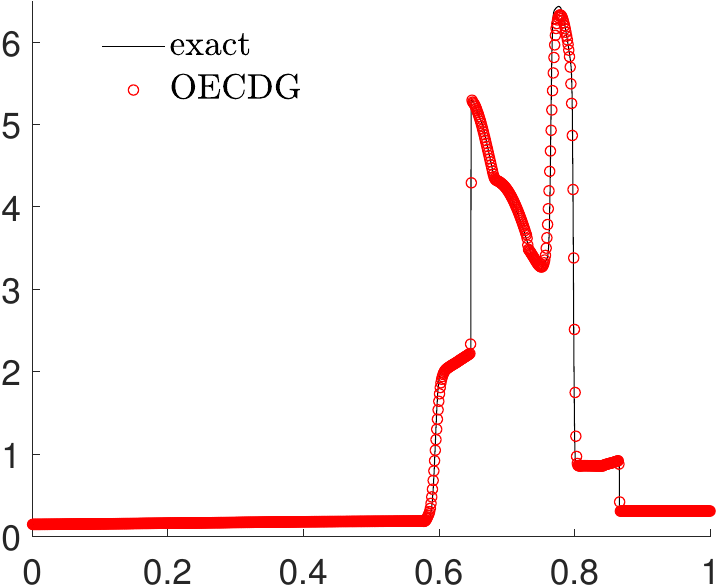}
			\subcaption{$\mathbb P^2$, two blast wave at $t = 0.038$.}
						\label{fig: tb}	}
			\end{subfigure}
			\end{minipage}
			\\
		
			\begin{minipage}[t]{0.9\linewidth}
			\centering
			\begin{subfigure}{0.45\linewidth}{\includegraphics[width = 1\linewidth]{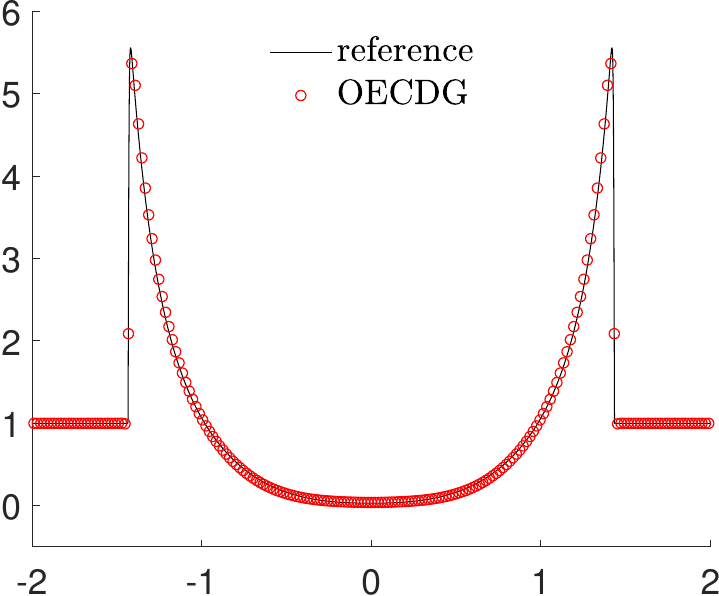}
				\subcaption{$\mathbb P^3$,  Sedov problem at $t = 10^{-3}$.}
				\label{fig: sedov}	}
			\end{subfigure}
			\begin{subfigure}{0.45\linewidth}{\includegraphics[width = 1\linewidth]{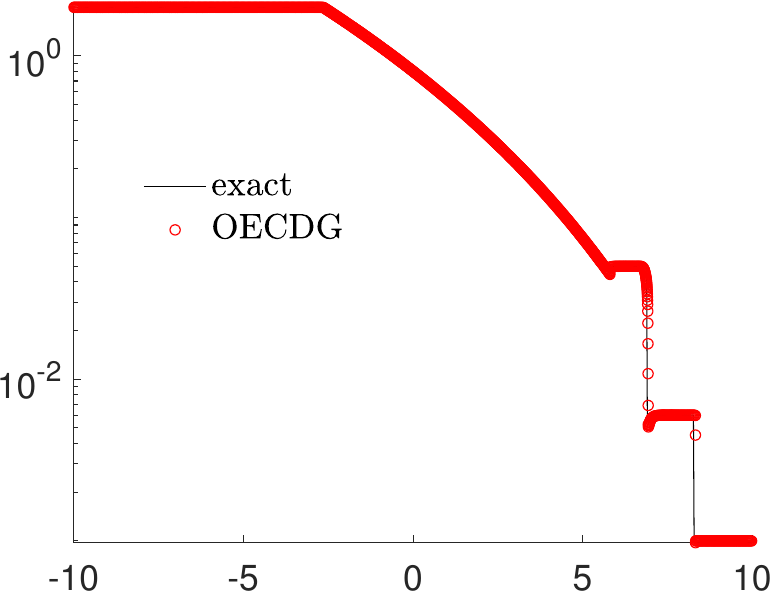}
					\subcaption{$\mathbb P^3$,  Leblanc problem at $t =10^{-4}$.}
					\label{fig: leblanc}	}
			\end{subfigure}
		\end{minipage}
			\caption{Densities of OECDG solutions for \Cref{Sod} to \ref{leblanc}.}
		\end{figure}
	\end{exmp}

\begin{exmp}[Leblanc problem]
\label{leblanc}
\rm
The Leblanc shock tube problem is set on $\Omega = [-10, 10]$ with outflow boundary conditions and initial values of $(\rho_0, v_0, p_0) = (2, 0, 10^9)$ for $x < 0$ and $(0.001, 0, 1)$ for $x > 0$. This setup involves an extreme pressure ratio of $10^9$ and a strong shock in a near-vacuum region, posing challenges for numerical methods. Figure \ref{fig: leblanc} shows the density computed by the $\mathbb{P}^3$-based OECDG method on a natural logarithmic scale at $t = 0.0001$ with 6400 cells, demonstrating robust performance without positivity-preserving limiters.
\end{exmp}

\subsection{2D linear advection equation}
In this subsection, we solve $u_t + u_x + u_y = 0$ on $\Omega = [0, 1] \times [0, 1]$ with periodic boundary conditions.

\begin{exmp}\label{2Dlinear_smooth} \rm To validate optimal error estimates, we simulate with a smooth initial condition $u_0(x, y) = \sin^2(\pi(x + y))$ up to $t = 1.1$. As shown in \Cref{2Dlinear_accuracy}, the $\mathbb{P}^k$-based OECDG methods achieve convergence rates near $k+1$ with mesh refinement, confirming our theoretical analysis. 

				\begin{table}[!tbh]
			\centering
			\caption{Errors and convergence rates of $\mathbb{P}^k$-based OECDG for 2D advection equation.}
			\begin{center}
				\begin{tabular}{c|c|c|c|c|c|c|c} 
					\bottomrule[1.0pt]
					$k$ &	$N_x$ &$l^1$ errors  & rate & $l^2$ errors  & rate & $l^\infty$ errors  & rate   \\
					\hline
					
					\multirow{5}*{$1$}
                 &256$\times$256&1.10e-03&--&1.25e-03&--&2.02e-03&--\\ 
				&512$\times$512&1.61e-04&2.76&1.90e-04&2.72&3.61e-04&2.48\\ 
				&1024$\times$1024&2.79e-05&2.53&3.28e-05&2.54&6.52e-05&2.47\\ 
				&2048$\times$2048&6.09e-06&2.20&6.85e-06&2.26&1.27e-05&2.36\\ 
				&4096$\times$4096&1.45e-06&2.07&1.62e-06&2.09&2.72e-06&2.23\\
				\hline
				\multirow{4}*{$2$}
				&128$\times$128&1.92e-04&--&2.12e-04&--&3.06e-04&--\\ 
				&256$\times$256&1.24e-05&3.95&1.37e-05&3.95&2.04e-05&3.91\\ 
				&512$\times$512&8.25e-07&3.91&9.04e-07&3.92&1.45e-06&3.82\\ 
				&1024$\times$1024&5.80e-08&3.83&6.33e-08&3.84&1.12e-07&3.69\\ 
				&2048$\times$2048&4.47e-09&3.70&4.89e-09&3.69&9.74e-09&3.53\\
				\hline
				\multirow{3}*{$3$}    
				&128$\times$128&4.22e-06&--&4.97e-06&--&9.91e-06&--\\ 
				&256$\times$256&1.06e-07&5.31&1.18e-07&5.40&2.08e-07&5.58\\ 
				&512$\times$512&3.14e-09&5.08&3.51e-09&5.06&8.19e-09&4.67\\ 
				&1024$\times$1024&1.01e-10&4.96&1.14e-10&4.94&3.55e-10&4.53\\
					\toprule[1.0pt]
				\end{tabular}
				\label{2Dlinear_accuracy}
			\end{center}
		\end{table}

	\end{exmp}
\begin{exmp}\label{linear2D_dis}
\rm
In this example, the initial value forms a pentagram shape, defined as $u_0(x, y) = 1$ for $\sqrt{x^2 + y^2} \leq \frac{3 + 3^{\sin(5\theta)}}{8}$ and $0$ elsewhere, where $\theta = \arccos\left(\frac{x}{r}\right)$ for $y \geq 0$ and $2\pi - \arccos\left(\frac{x}{r}\right)$ for $y < 0$. 
Figure \ref{Fig: linearpentagram} shows OECDG solutions with $320 \times 320$ cells, capturing the pentagram’s propagation without spurious oscillations. As expected, higher-degree elements provide better resolution.

\begin{figure}[!thb]
    \centering 
    \begin{subfigure}{0.325\linewidth}
        \includegraphics[width=1\linewidth]{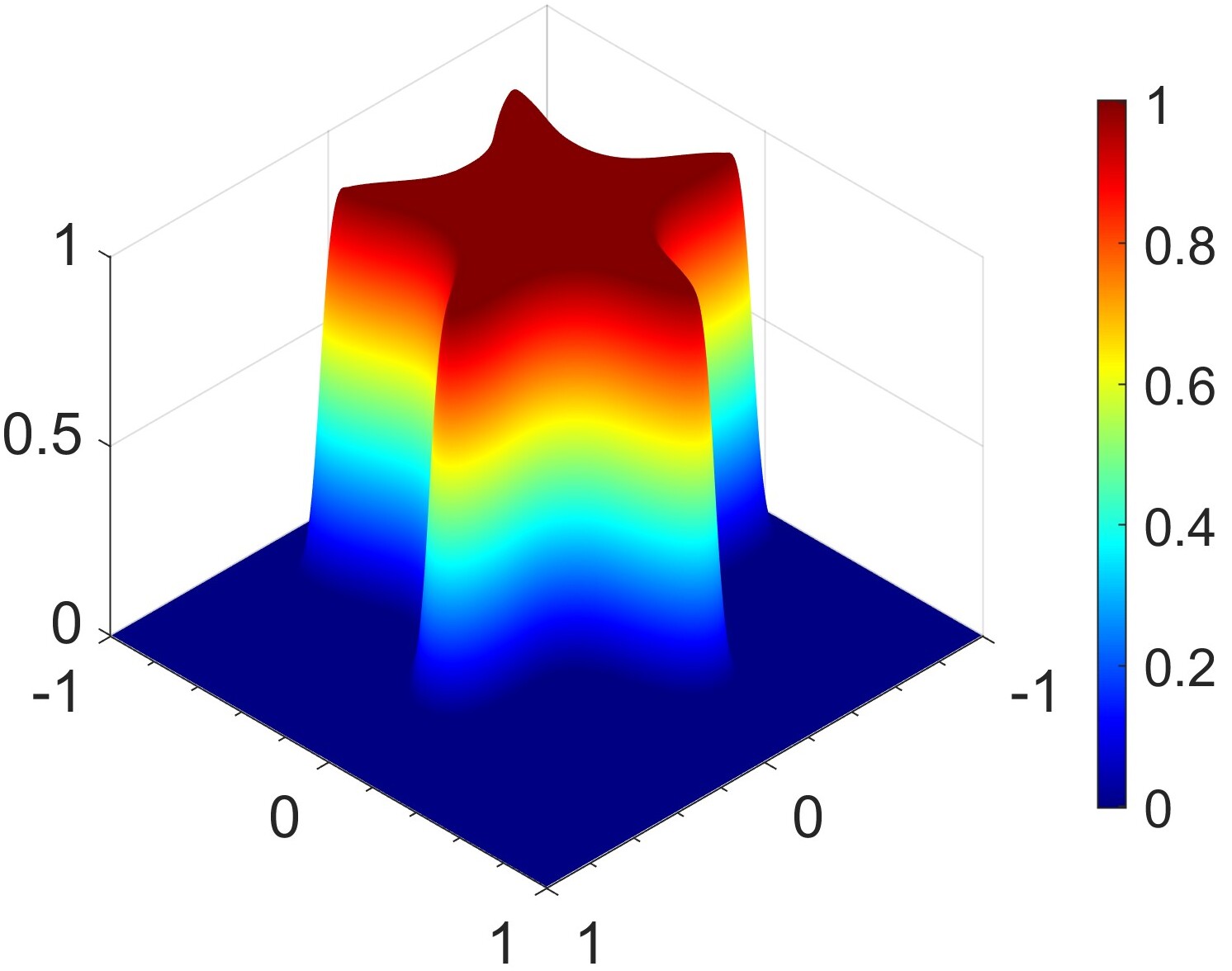}
    \end{subfigure}
    \begin{subfigure}{0.325\linewidth}
        \includegraphics[width=1\linewidth]{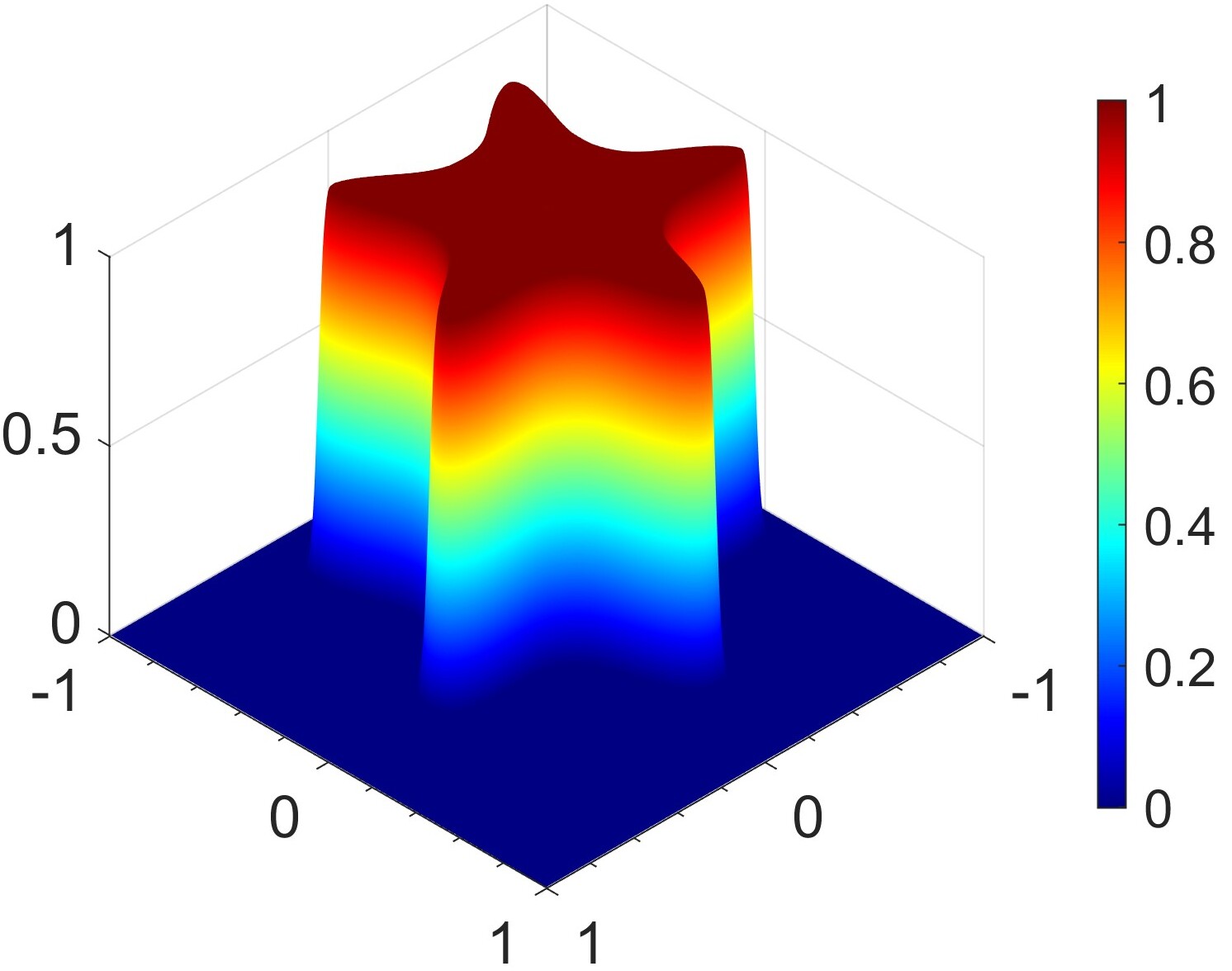}
    \end{subfigure}
    \begin{subfigure}{0.325\linewidth}
        \includegraphics[width=1\linewidth]{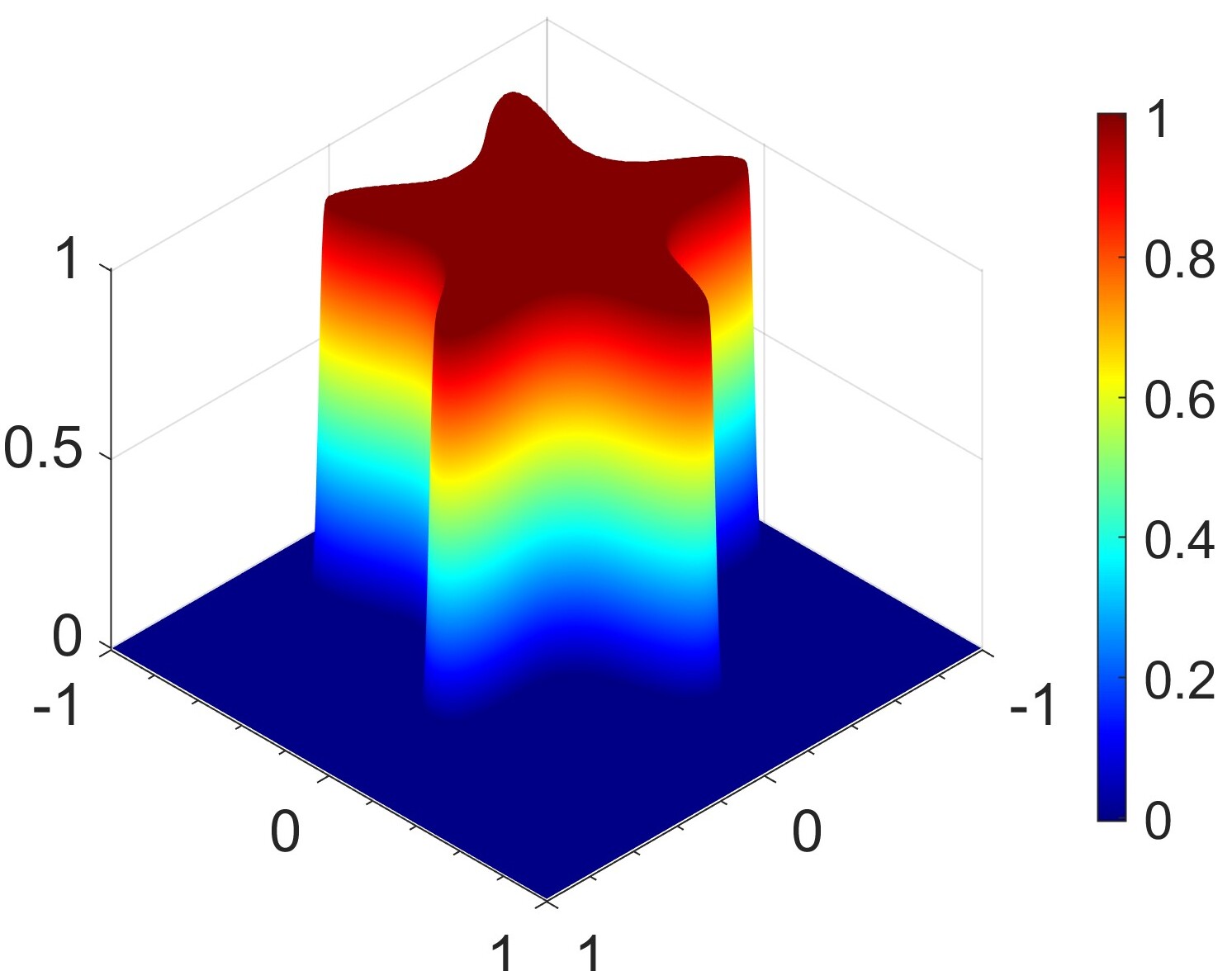}
    \end{subfigure}
    \\
    \begin{subfigure}{0.325\linewidth}
        \includegraphics[width=1\linewidth]{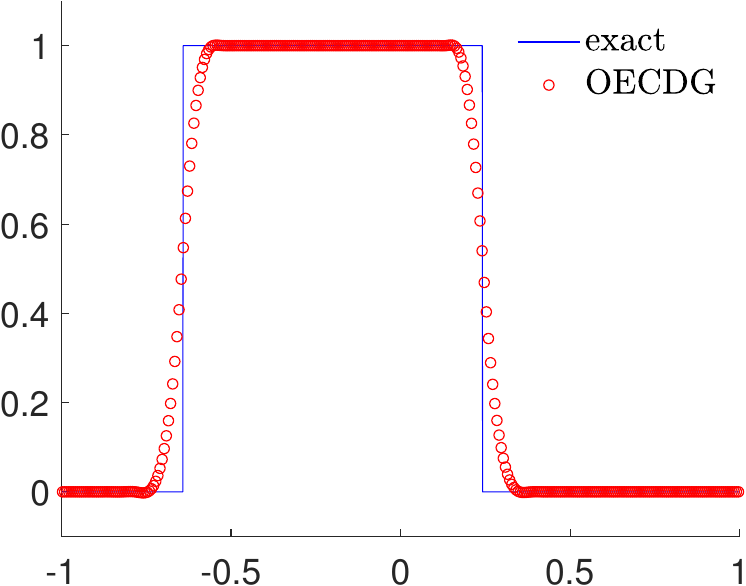}
    \end{subfigure}
    \begin{subfigure}{0.325\linewidth}
        \includegraphics[width=1\linewidth]{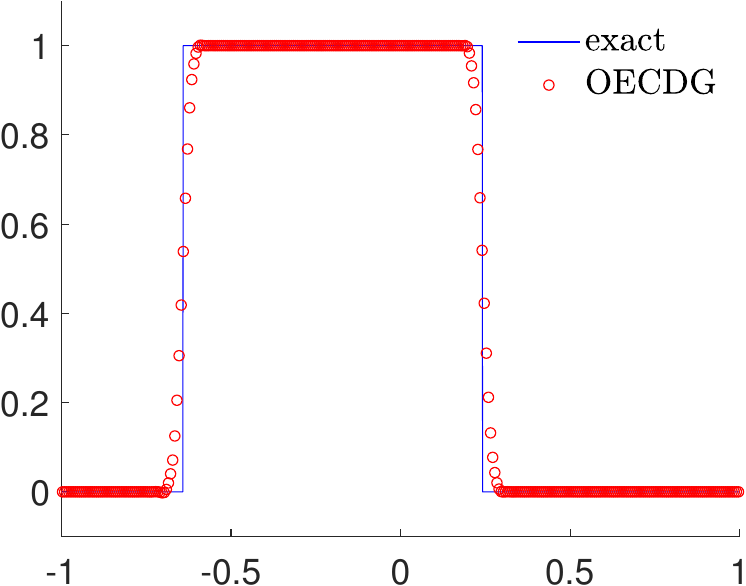}
    \end{subfigure}
    \begin{subfigure}{0.325\linewidth}
        \includegraphics[width=1\linewidth]{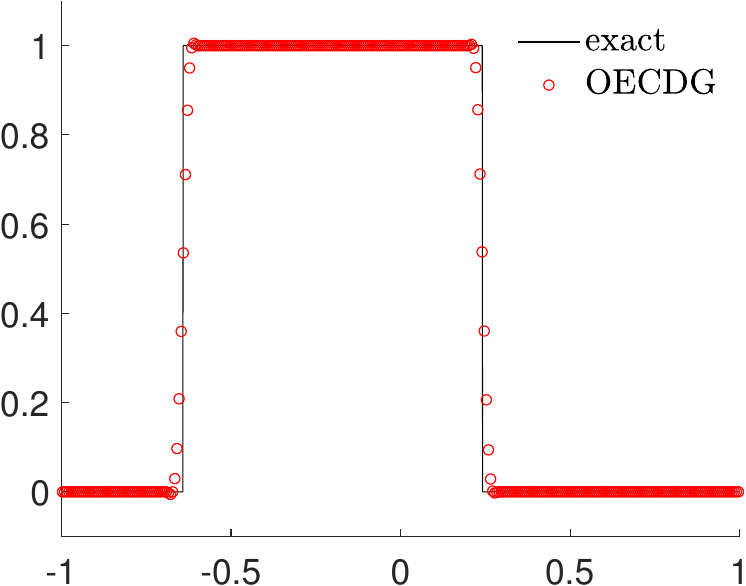}
    \end{subfigure}
    \caption{Top: OECDG solutions at $t = 1.8$. Bottom: Solution cross-sections along $y = -0.25$. Left to right: $\mathbb{P}^1$, $\mathbb{P}^2$, and $\mathbb{P}^3$ elements.}
    \label{Fig: linearpentagram}
\end{figure}
\end{exmp}

\subsection{A nonconvex nonlinear conservation law} We solve $u_t + (\sin(u))_x + (\cos(u))_y = 0$ on $[-2, 2] \times [-2.5, 1.5]$ with $320 \times 320$ cells up to $t = 1$. Inflow conditions are applied on the top and left boundaries, with outflow on the bottom and right. 
\begin{exmp}[KPP problem \cite{kurganov2007adaptive}]
\rm 
The initial data $u(x, y, 0)$ is set to $3.5$ within the circle $x^2 + y^2 < 1$ and $0.25$ elsewhere.  
This problem involves complex 2D wave patterns and a non-convex flux function, challenging high-order schemes. Figure \ref{Fig: KPP} shows the third-order OECDG solution, which captures intricate wave structures, eliminates spurious oscillations, and outperforms the OEDG result.

\begin{figure}[!tbh] 
			\begin{subfigure}{0.46\linewidth}{    \includegraphics[width=1\textwidth]{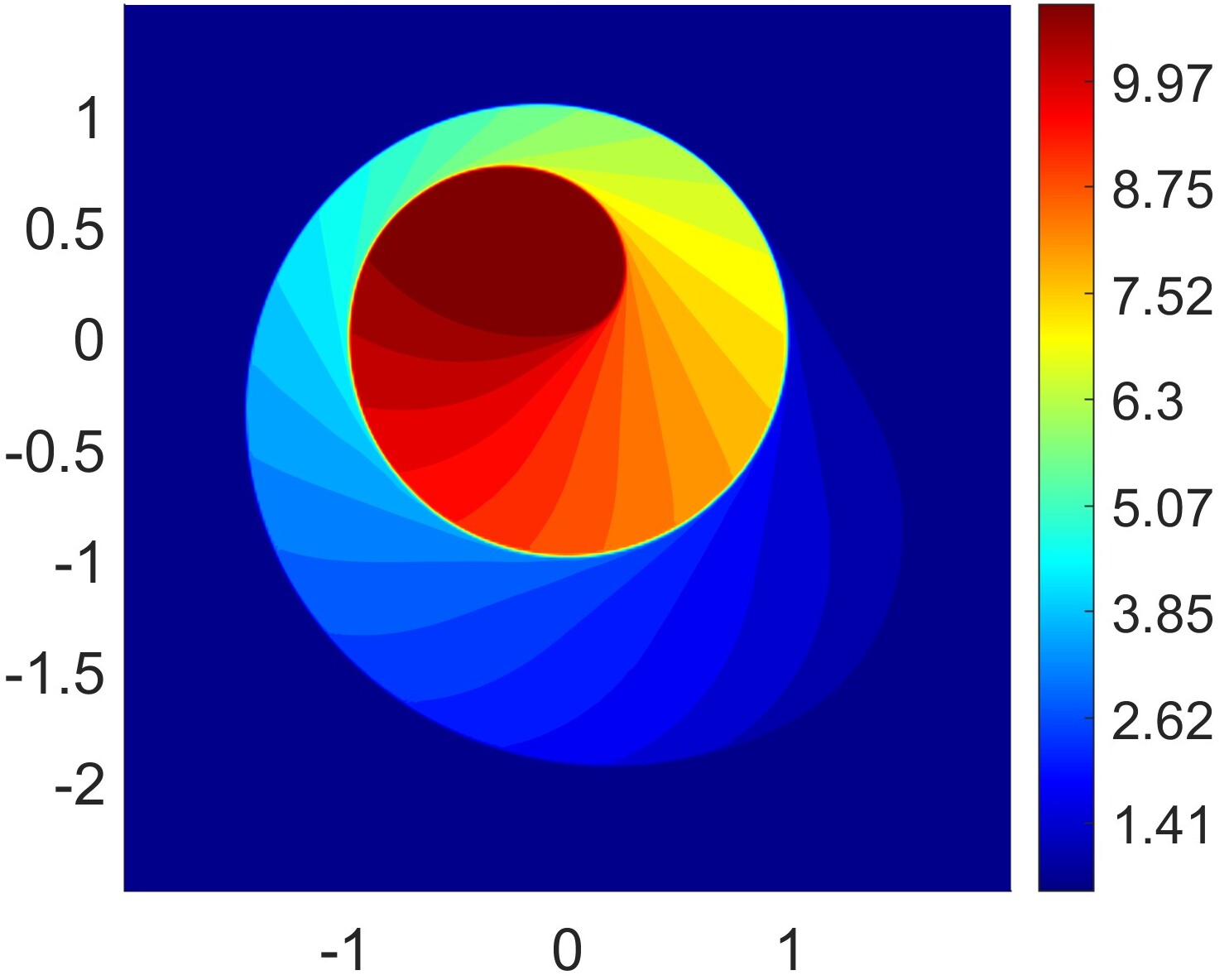}	
					\subcaption{OECDG solution for KPP problem.}
					\label{Fig: KPP}	}
			\end{subfigure}
			\begin{subfigure}{0.47\linewidth}{  \includegraphics[width=1\textwidth]{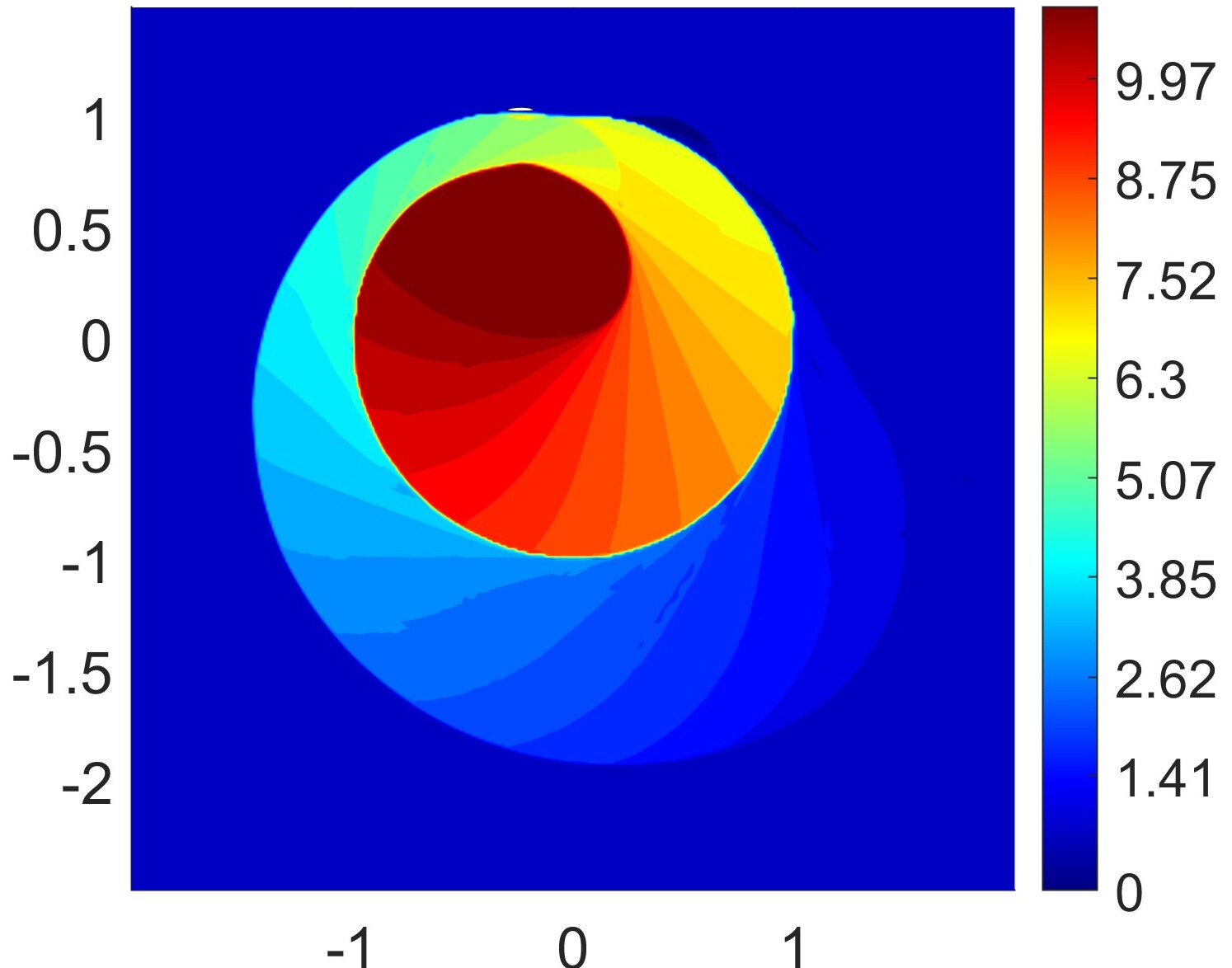}\subcaption{OEDG solution for KPP problem.}
		\label{Fig: OEDGKPP}	}
			\end{subfigure}

\end{figure}
\end{exmp}

\subsection{2D Euler equations}
In this subsection, we simulate various benchmark problems for the 2D Euler equations ${\bm u}_t + \nabla \cdot \boldsymbol{f}({\bm u}) = {\bf 0}$, where ${\bm u} = (\rho, \rho {\bm v}, E)^\top$ and 
${\bm f}({\bm u}) = (\rho {\bm v}, \rho {\bm v} \otimes {\bm v} + p {\bf I}, (E + p) {\bm v})^\top$. 
Here, $\rho$ represents density, ${\bm v}$ is the velocity field, $p$ denotes pressure, and the total energy is $E = \frac{1}{2} \rho |{\bm v}|^2 + \frac{p}{\gamma - 1}$. The adiabatic index $\gamma$ is set to $1.4$ unless specified otherwise.

\begin{exmp}
\rm 
This smooth problem is simulated on $[0, 2] \times [0, 2]$ using uniform rectangular meshes. The exact solution remains $(\rho, u, v, p) = (1 + 0.2 \sin(\pi(x - y)), 1, 1, 1)$. 
Table \ref{2DEuler_steady} lists the errors and convergence rates of the OECDG solutions at $t = 10$, confirming that the methods achieve optimal convergence rates for the 2D Euler equations.

\begin{table}[!tbh]
    \centering
    \caption{Errors and convergence rates for the $\mathbb{P}^k$-based OECDG method.}
    \begin{center}
        \begin{tabular}{c|c|c|c|c|c|c|c}
            \bottomrule[1.0pt]
            $k$ &	$N_x$ &$e_1$  & rate & $e_2$  & rate & $e_3$  & rate \\
            \hline
            \multirow{4}*{$1$}
            &80$\times$80&9.65e-04&-&1.31e-03&-&5.76e-03&-\\ 
            &160$\times$160&1.51e-04&2.68&2.01e-04&2.71&1.21e-03&2.25\\ 
            &320$\times$320&2.09e-05&2.85&2.77e-05&2.86&1.62e-04&2.91\\ 
            &640$\times$640&2.93e-06&2.84&3.80e-06&2.87&2.03e-05&3.00\\ 
            &1280$\times$1280&4.49e-07&2.71&5.70e-07&2.74&2.58e-06&2.98\\ 

            \hline
            \multirow{4}*{$2$}
          &80$\times$80&2.00e-05&-&2.71e-05&-&9.56e-05&-\\ 
          &160$\times$160&1.26e-06&3.99&1.70e-06&3.99&6.265E-06&3.93\\ 
          &320$\times$320&7.98e-08&3.98&1.07e-07&3.99&4.17e-07&3.91\\ 
          &640$\times$640&5.19e-09&3.94&6.89e-09&3.96&2.82e-08&3.89\\ 
          &1280$\times$1280&3.67e-10&3.82&4.80e-10&3.84&2.01e-09&3.81\\ 
            \hline
            \multirow{2}*{$3$}    
            &80$\times$80&1.68e-07&-&2.42e-07&-&1.46e-06&-\\ 
            &160$\times$160&7.08e-09&4.57&9.66e-09&4.65&5.69e-08&4.69\\ 
            &320$\times$320&2.44e-10&4.86&3.28e-10&4.88&1.71e-09&5.06\\ 
            &640$\times$640&8.24e-12&4.88&1.10e-11&4.90&5.71e-11&4.90\\ 
            \toprule[1.0pt]
        \end{tabular}
        \label{2DEuler_steady}
    \end{center}
\end{table}
\end{exmp}

\begin{exmp}[supersonic flow \cite{zhu2017numerical}]
The initial conditions are $(\rho_0, u_0, v_0, p_0) = \left(1, \cos\left(\frac{\pi}{12}\right), \sin\left(\frac{\pi}{12}\right), \frac{1}{9\gamma}\right)$, modeling a supersonic flow interacting with plates positioned at $(x, y) \in (2, 3) \times \{\pm 2\}$ with an attack angle of $15^{\circ}$. Inflow conditions are applied at the lower and left edges, with outflow conditions on the upper and right edges. 
Figure \ref{Fig: supersonic}(a) shows the density contour computed using the fourth-order OECDG method on the domain $[0, 10] \times [-5, 5]$ at $t = 100$ with $200 \times 200$ rectangular cells. Figure \ref{Fig: supersonic}(b) demonstrates that the average residue reaches machine epsilon, indicating a well-resolved steady state. The OECDG solution shows no spurious oscillations.

\begin{figure}[!tbh]
    \centering
    \begin{subfigure}[h]{.52\linewidth}
        \centering 
        \includegraphics[width=0.9\textwidth]{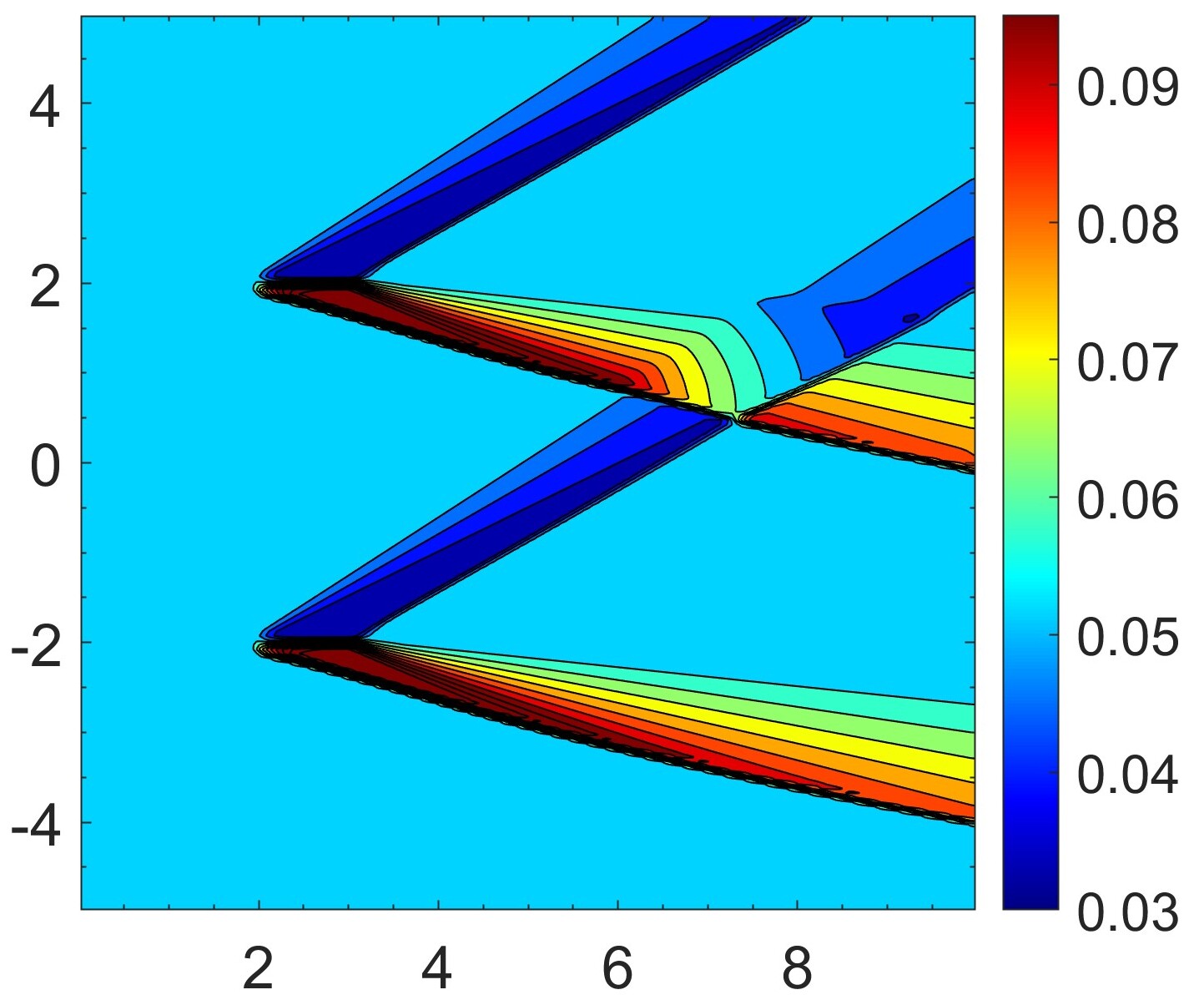}
        \subcaption{Density contour at $t=100$}
        \label{2D_supersonicAA}
    \end{subfigure}
    \hfill 
    \begin{subfigure}[h]{.46\linewidth}
        \centering 
        \includegraphics[width=1.05\textwidth]{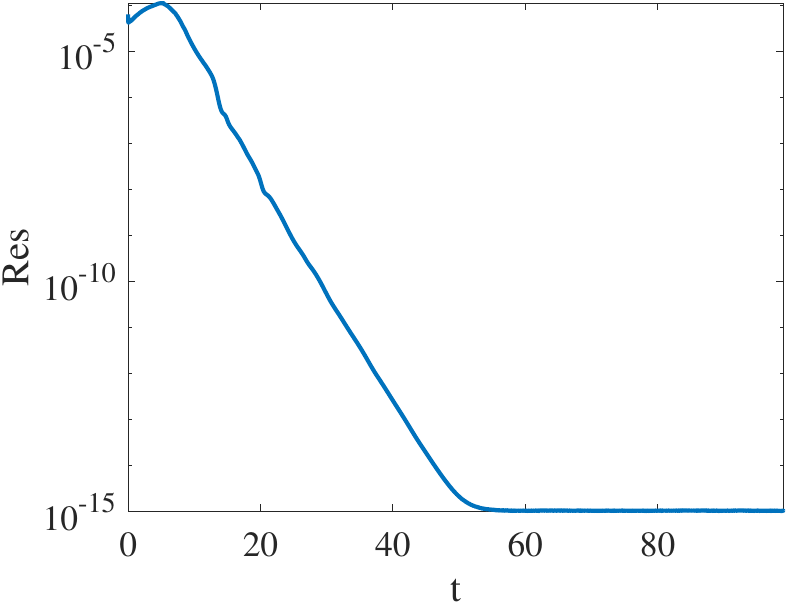}
        \subcaption{Average residue}
        \label{2D_supersonicBB}
    \end{subfigure}
    \caption{Supersonic flow problem simulated with the $\mathbb{P}^3$-based OECDG scheme.}
    \label{Fig: supersonic}
\end{figure}
\end{exmp}

\begin{exmp}[shock-vortex interaction]
\rm
This test simulates the interaction between a vortex and a stationary shock located to the left of $x = 0.5$ in the domain $[0, 2] \times [0, 1]$. The shock, aligned perpendicular to the $x$-axis, has conditions $(\rho, u, v, p) = (1, 1.1\sqrt{\gamma}, 0, 1)$. 
The vortex introduces perturbations defined by
\[
(\delta u, \delta v) = \frac{\epsilon}{r_c} e^{\alpha(1 - \tau^2)} (y - y_c, x_c - x), \quad \delta T = -\frac{(\gamma - 1) \epsilon^2}{4\alpha \gamma} e^{2\alpha(1 - \tau^2)}, \quad \delta S = 0,
\]
where $T = p / \rho$ and $S = \ln(p / \rho^{\gamma})$. The vortex is centered at $(x_c, y_c) = (0.25, 0.5)$ with strength $\epsilon = 0.3$, critical radius $r_c = 0.05$, and decay rate $\alpha = 0.204$. The domain is bounded by walls on the upper and lower sides, with inflow and outflow conditions on the left and right, respectively. Figure \ref{Fig: vortex} shows the density contours from the fourth-order OECDG method at three time points up to $t = 0.8$, using $400 \times 200$ uniform cells. The shock-vortex interaction dynamics match the results in \cite{peng2023oedg} and are captured with high resolution and free of nonphysical oscillations.
\end{exmp}

		\begin{figure}[!tbh]
			\centering
			\begin{subfigure}[h]{.325\linewidth}
				\centering 
				\includegraphics[width=0.99\textwidth]{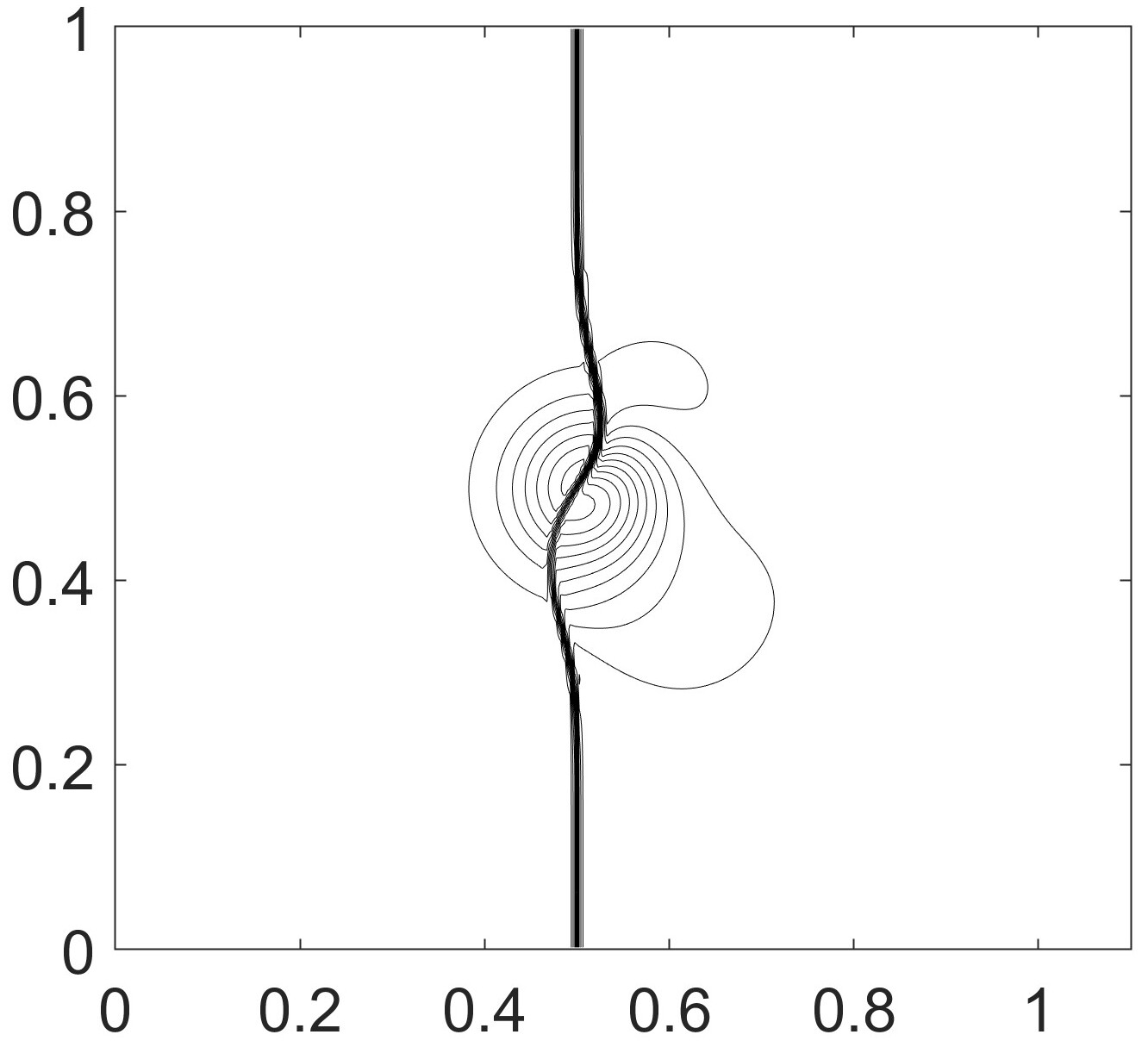}
				\subcaption{$t=0.203$}
			\end{subfigure}
			\hfill 			
			\begin{subfigure}[h]{.325\linewidth}
				\centering 
				\includegraphics[width=0.99\textwidth]{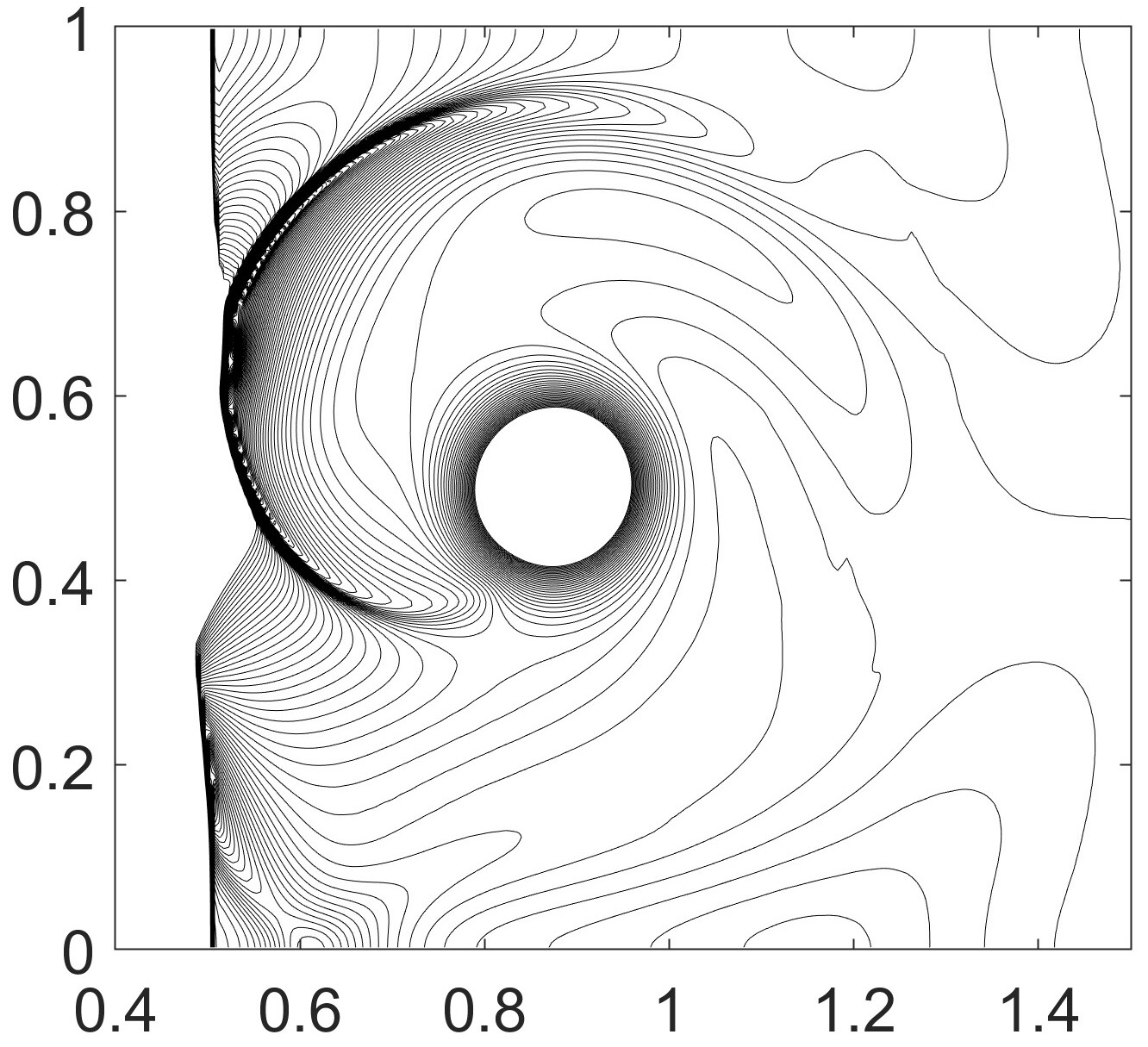}
				\subcaption{$t=0.529$}
			\end{subfigure}
			\hfill 
			\begin{subfigure}[h]{.325\linewidth}
				\centering 
				\includegraphics[width=0.99\textwidth]{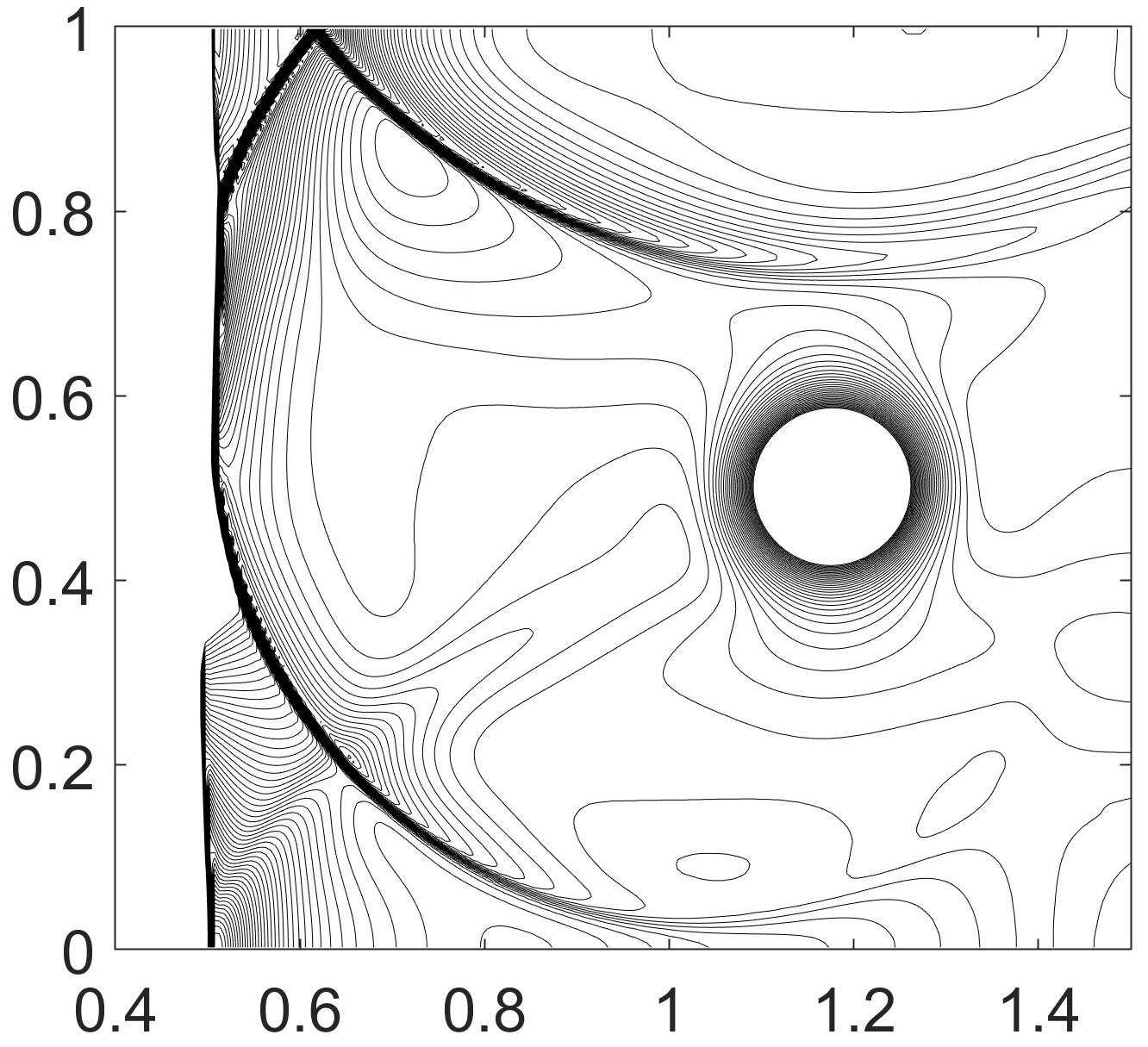}
				\subcaption{$t=0.8$}
			\end{subfigure}
			\caption{Density contours for shock-vortex interaction. (a): 30 contour lines from 0.68 to 1.3. (b)--(c): 90 contour lines from 1.19 to 1.36.}
			\label{Fig: vortex}
		\end{figure}

\begin{exmp}[high-speed jets \cite{ha2005numerical}]
\rm
This example simulates two high-speed jet problems, with the heat capacity ratio set to $\gamma = \frac{5}{3}$. Outflow boundary conditions are applied on all sides except for the jet inlet. Both tests use the $\mathbb{P}^2$-based OECDG method with a positivity-preserving limiter \cite{li2016maximum} and $C_{\rm CFL} = 0.2$. Simulations are performed on the upper half of the domain, with full-domain solutions obtained by mirroring across the midline. 
\begin{itemize}
    \item {\bf Mach 80 Jet}: The initial state is $(\rho, u, v, p) = (0.5, 0, 0, 0.4127)$ over $(0, 2) \times (-0.5, 0.5)$, with a jet inflow of $(\rho, u, v, p) = (5, 30, 0, 0.4127)$ at ${0} \times (-0.05, 0.05)$. The simulation proceeds to $t = 0.07$ with $480 \times 60$ uniform cells. Logarithmic density, pressure, and temperature are plotted in Figure \ref{fig: jet2000}.
    \item {\bf Mach 2000 Jet}: The domain $(0, 1) \times (-0.25, 0.25)$ is initialized with $(\rho, u, v, p) = (0.5, 0, 0, 0.4127)$. A high-speed jet of $(\rho, u, v, p) = (5, 800, 0, 0.4127)$ enters through $\{0\} \times (-0.05, 0.05)$. Simulations are conducted up to $t = 0.001$ on a grid of $320 \times 80$ cells, with cell-averaged density, pressure, and temperature presented in logarithmic scale in Figure \ref{fig: jet2000}.
\end{itemize}
Both simulations effectively capture bow shocks and interface shear flows, demonstrating the OECDG scheme's robustness in handling challenging jet simulations without generating nonphysical oscillations.

\end{exmp}

\begin{figure}[!tbh]
    \centering 
    \begin{subfigure}{0.32\linewidth}
        \includegraphics[width=1\linewidth]{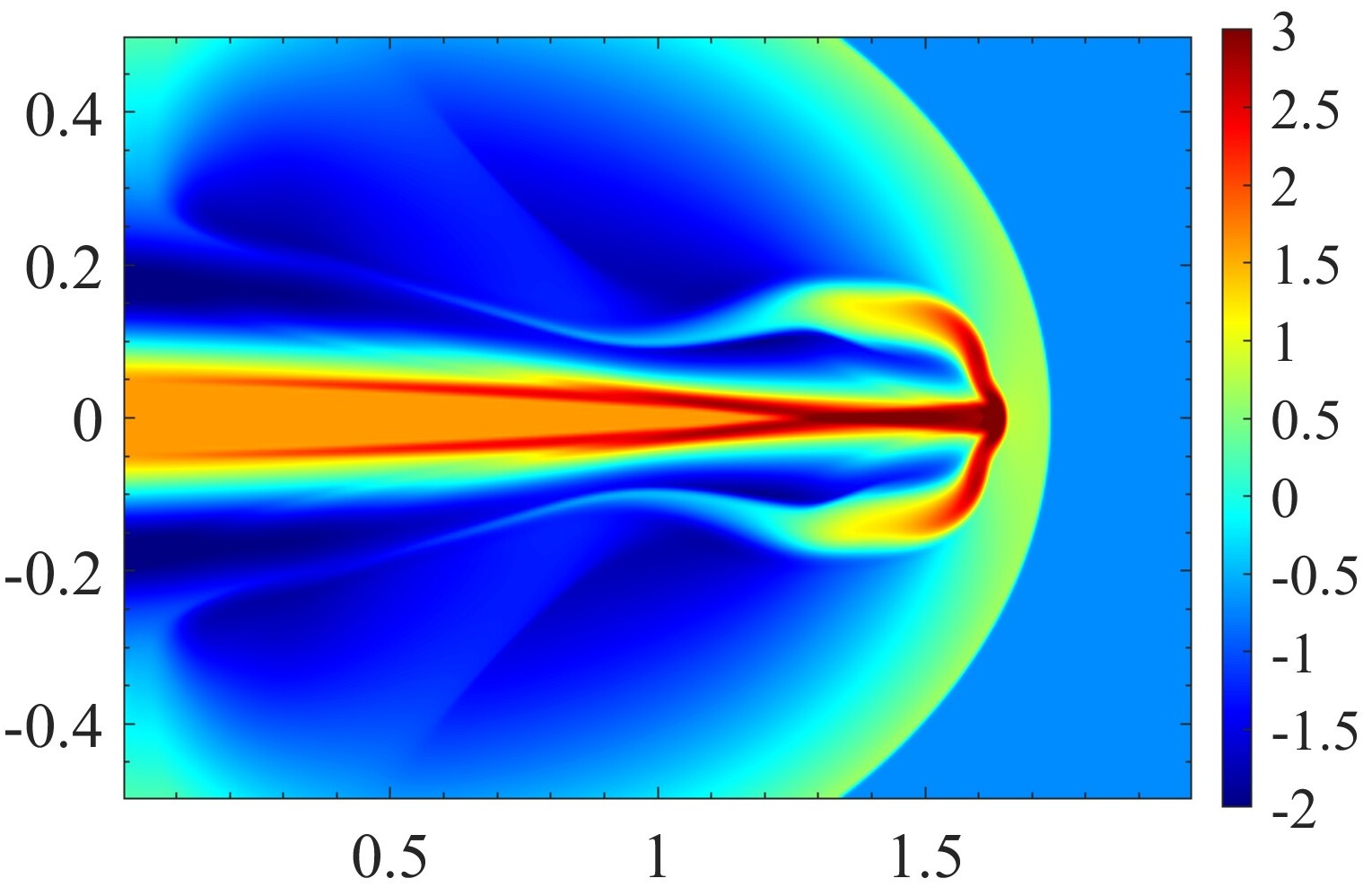}
    \end{subfigure}
    \begin{subfigure}{0.32\linewidth}
        \includegraphics[width=1\linewidth]{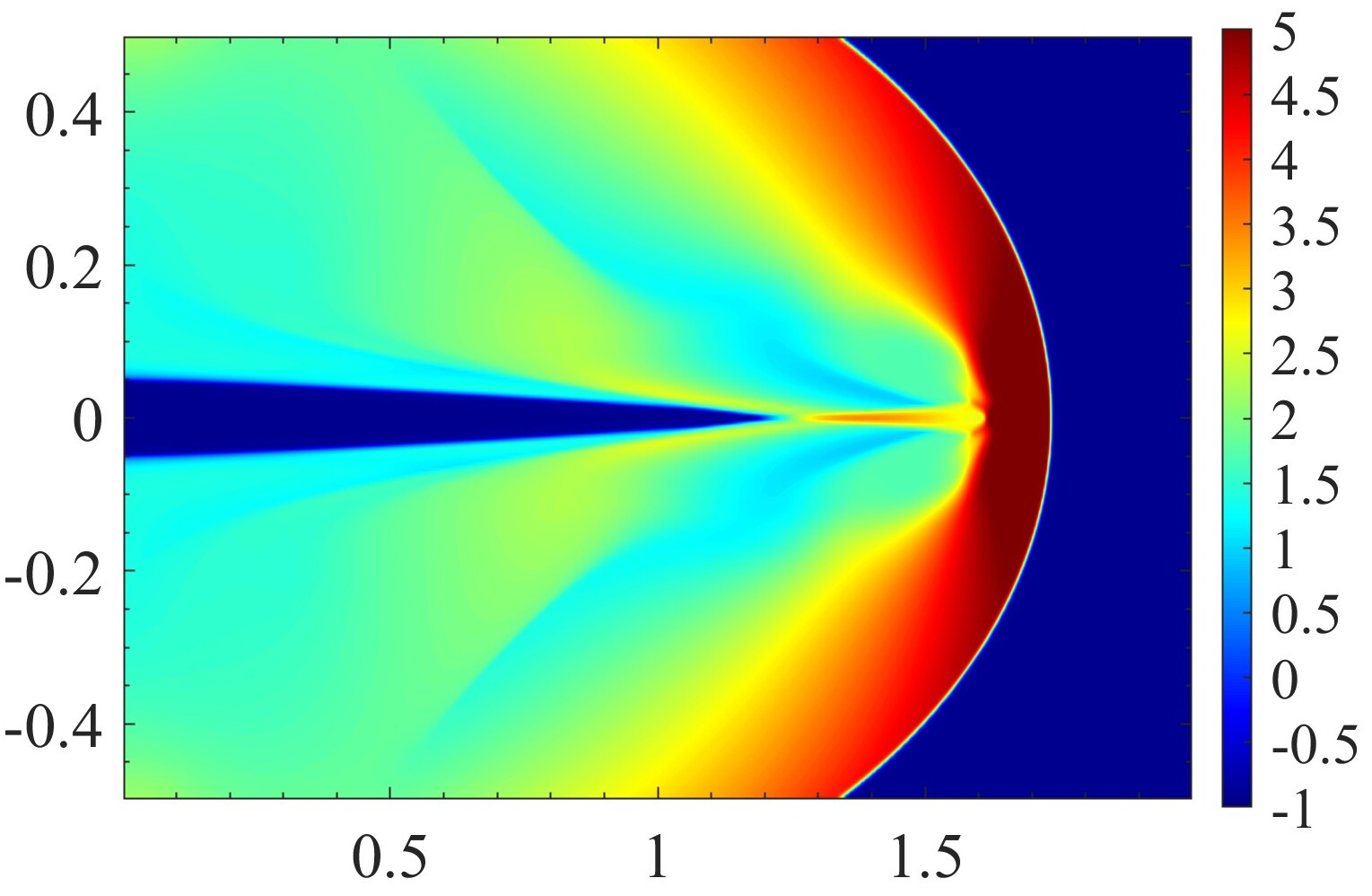}

    \end{subfigure}
    \begin{subfigure}{0.32\linewidth}
        \includegraphics[width=1\linewidth]{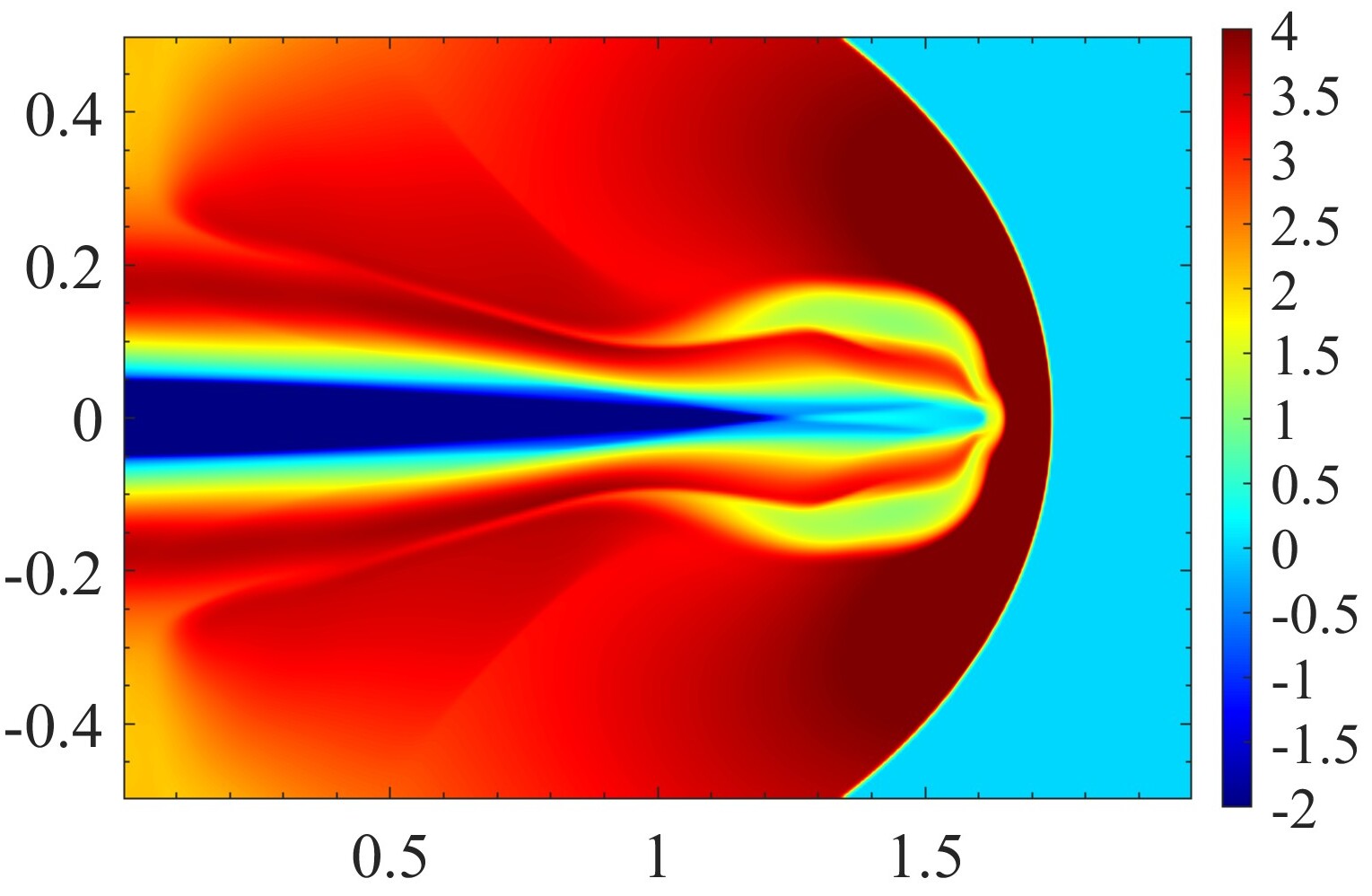}
    \end{subfigure}
    \\
    \begin{subfigure}{0.32\linewidth}
        \includegraphics[width=1\linewidth]{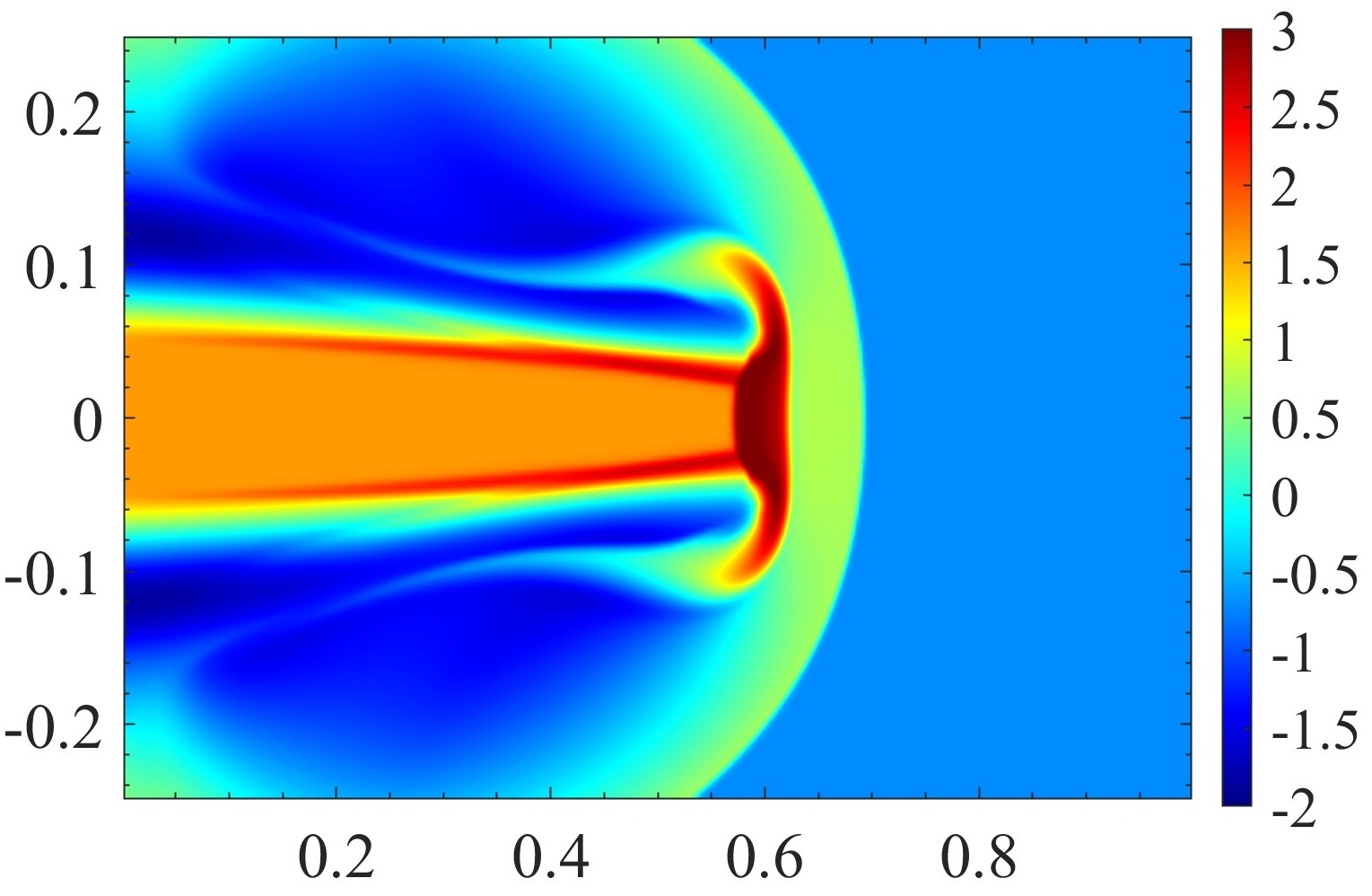}
    \end{subfigure}
    \begin{subfigure}{0.32\linewidth}
        \includegraphics[width=1\linewidth]{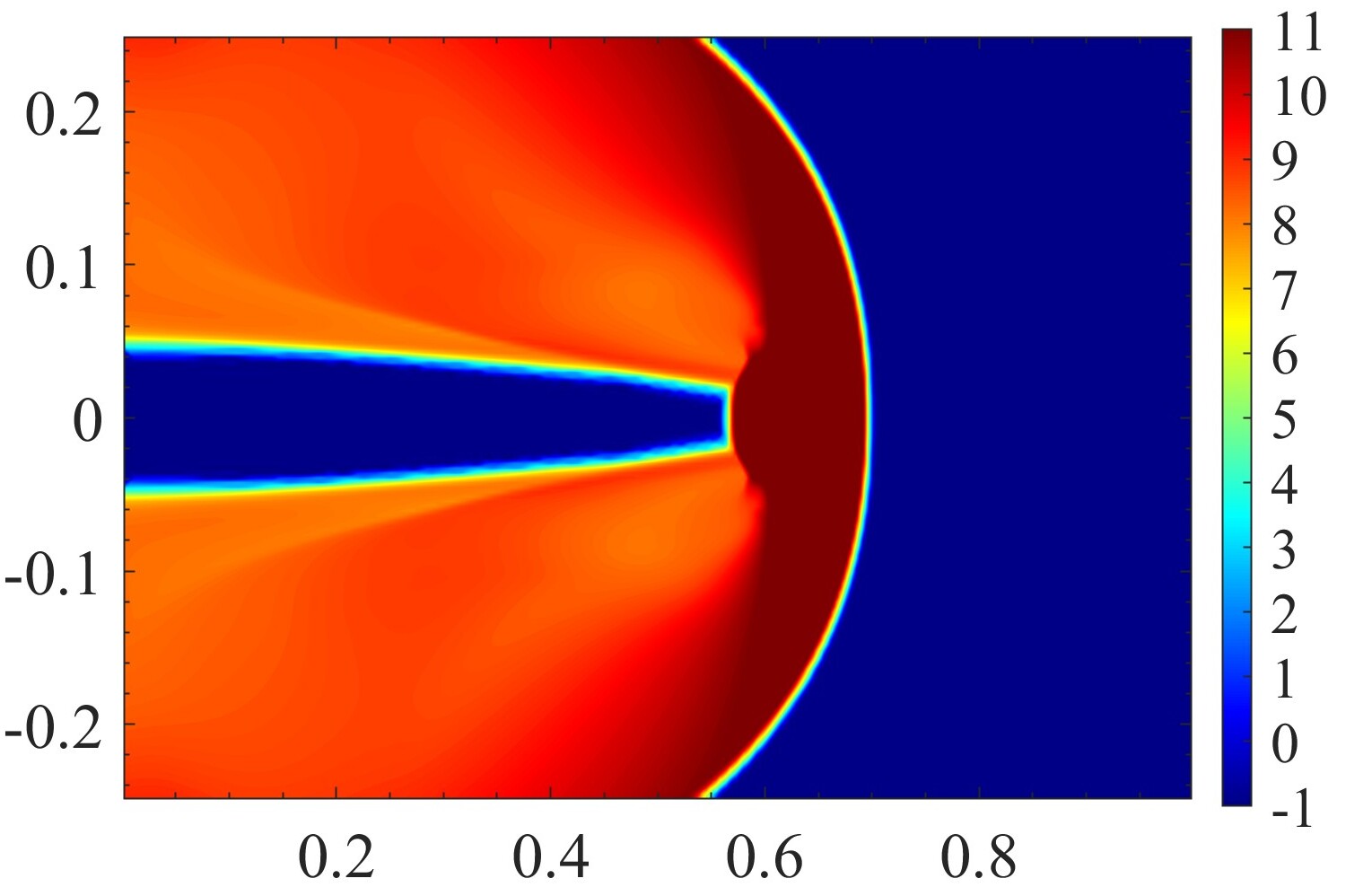}
    \end{subfigure}
    \begin{subfigure}{0.32\linewidth}
        \includegraphics[width=1\linewidth]{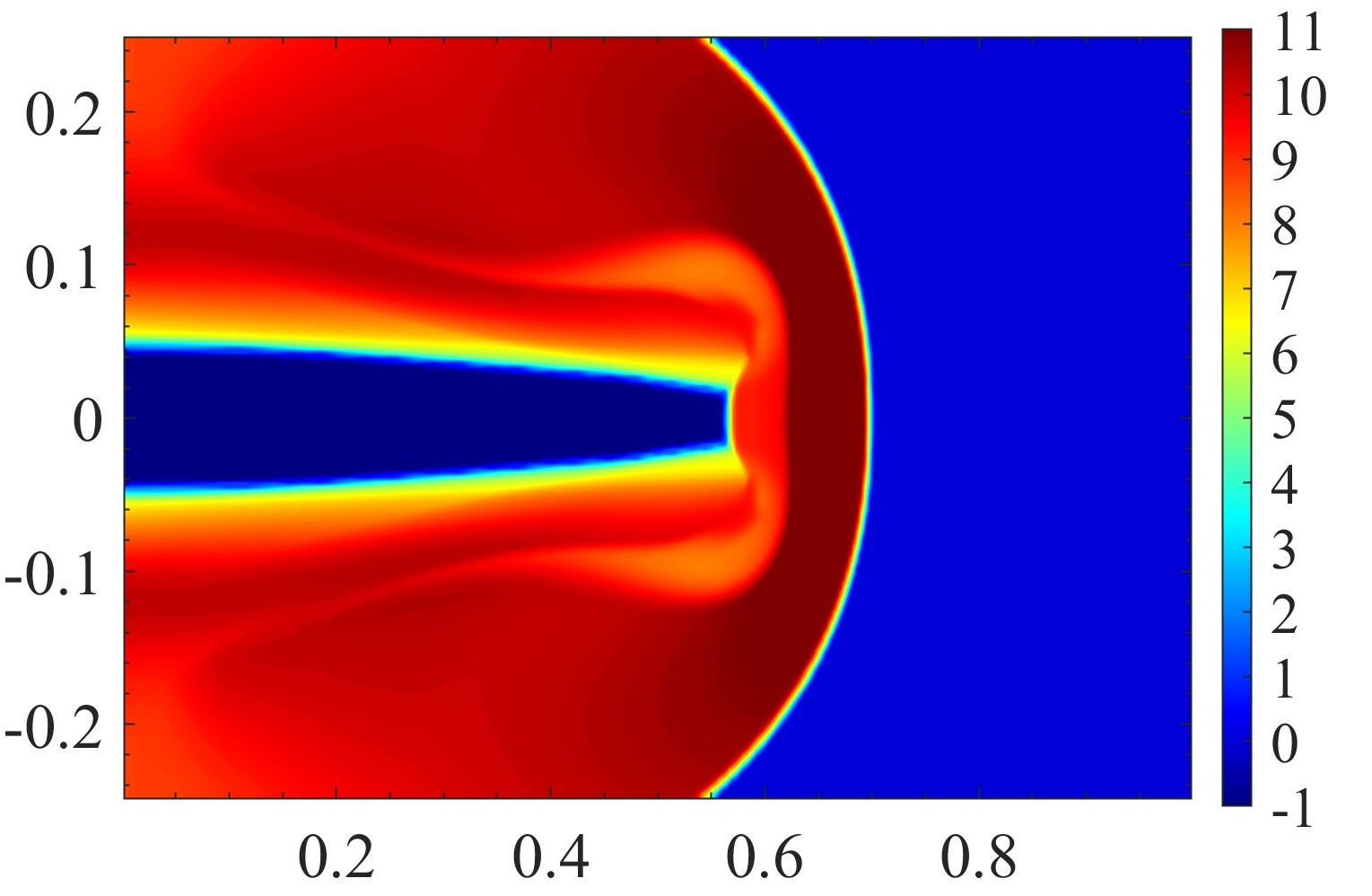}
    \end{subfigure}
\caption{Mach 80 jet at $t = 0.07$ (top) and Mach 2000 jet at $t = 0.001$ (bottom): density, pressure, and temperature logarithms (left to right).}
    \label{fig: jet2000}
\end{figure}

\section{Conclusions}\label{sec:conclusions}

In this work, we have developed and analyzed the OECDG method, a novel central discontinuous Galerkin (CDG) scheme that effectively integrates a new oscillation-eliminating (OE) procedure with a novel dual damping mechanism within the CDG framework. 
Inspired by the numerical dissipation inherent to CDG schemes, the new OE damping mechanism is designed based on the difference between overlapping solutions, enhancing both resolution and stencil compactness compared to original OE procedures that rely on inter-cell solution jumps. 
Our theoretical contributions include the rigorous derivation of optimal error estimates for the fully discrete OECDG method, addressing a critical gap in the error analysis of fully discrete CDG schemes, including those without oscillation control. For the first time, we have established the approximate skew-symmetry and weak boundedness of the CDG spatial discretization, which provides a robust foundation for stability analysis. Using matrix transfer techniques, we have proven the linear stability of CDG methods coupled with Runge–Kutta time discretization. Based on these results, we have established optimal error estimates for the fully discrete OECDG method, a notable challenge given the method’s nonlinear nature, even for linear advection equations. 
The OECDG method's efficacy has been further validated through an extensive suite of numerical experiments across a diverse set of hyperbolic conservation laws, including convection equations, Burgers' equation, traffic flow models, non-convex conservation laws, and Euler equations. The results consistently demonstrate that the OECDG method not only achieves optimal convergence rates but also effectively captures complex wave structures without introducing spurious oscillations. 
Overall, this study affirms the accuracy and robustness of OECDG schemes while underscoring their potential as versatile tools for tackling challenging hyperbolic problems.

    \begin{changemargin}{-0.5cm}{-0.3cm}  %
	    \renewcommand\baselinestretch{0.86}
	    \bibliography{refs}

\begin{thebibliography}{10}

\bibitem{CaoPengWu}
{\sc H.~Cao, M.~Peng, and K.~Wu}, {\em Robust discontinuous {G}alerkin methods
  maintaining physical constraints for general relativistic hydrodynamics}, J.
  Comput. Phys.,  (2025), p.~113770.

\bibitem{chavent1989local}
{\sc G.~Chavent and B.~Cockburn}, {\em The local projection
  ${P}^0-{P}^1$-discontinuous-{G}alerkin finite element method for scalar
  conservation laws}, ESAIM: Math. Model. Numer. Anal., 23 (1989),
  pp.~565--592.

\bibitem{cockburn1990runge}
{\sc B.~Cockburn, S.~Hou, and C.-W. Shu}, {\em The {R}unge--{K}utta local
  projection discontinuous {G}alerkin finite element method for conservation
  laws. {IV}. {T}he multidimensional case}, Math. Comput., 54 (1990),
  pp.~545--581.

\bibitem{cockburn1989tvb}
{\sc B.~Cockburn, S.-Y. Lin, and C.-W. Shu}, {\em {TVB} {R}unge--{K}utta local
  projection discontinuous {G}alerkin finite element method for conservation
  laws {III}: {O}ne-dimensional systems}, J. Comput. Phys., 84 (1989),
  pp.~90--113.

\bibitem{cockburn1998runge}
{\sc B.~Cockburn and C.-W. Shu}, {\em The {R}unge--{K}utta discontinuous
  {G}alerkin method for conservation laws {V}: {M}ultidimensional systems}, J.
  Comput. Phys., 141 (1998), pp.~199--224.

\bibitem{cockburn2001runge}
{\sc B.~Cockburn and C.-W. Shu}, {\em {R}unge--{K}utta discontinuous {G}alerkin
  methods for convection-dominated problems}, J. Sci. Comput., 16 (2001),
  pp.~173--261.

\bibitem{DingCuiWu2024}
{\sc S.~Ding, S.~Cui, and K.~Wu}, {\em Robust {DG} schemes on unstructured
  triangular meshes: {O}scillation elimination and bound preservation via
  optimal convex decomposition}, J. Comput. Phys.,  (2025), p.~113769.

\bibitem{du2023oscillation}
{\sc J.~Du, Y.~Liu, and Y.~Yang}, {\em An oscillation-free bound-preserving
  discontinuous {G}alerkin method for multi-component chemically reacting
  flows}, J. Sci. Comput., 95 (2023), p.~90.

\bibitem{FanWu2024}
{\sc C.~Fan and K.~Wu}, {\em High-order oscillation-eliminating {H}ermite
  {WENO} method for hyperbolic conservation laws}, J. Comput. Phys.,  (2024),
  p.~113435.

\bibitem{ha2005numerical}
{\sc Y.~Ha, C.~L. Gardner, A.~Gelb, and C.-W. Shu}, {\em Numerical simulation
  of high {M}ach number astrophysical jets with radiative cooling}, J. Sci.
  Comput., 24 (2005), pp.~29--44.

\bibitem{hiltebrand2014entropy}
{\sc A.~Hiltebrand and S.~Mishra}, {\em Entropy stable shock capturing
  space-time discontinuous {G}alerkin schemes for systems of conservation
  laws}, Numer. Math., 126 (2014), pp.~103--151.

\bibitem{huang2020adaptive}
{\sc J.~Huang and Y.~Cheng}, {\em An adaptive multiresolution discontinuous
  {G}alerkin method with artificial viscosity for scalar hyperbolic
  conservation laws in multidimensions}, SIAM J. Sci. Comput., 42 (2020),
  pp.~A2943--A2973.

\bibitem{kurganov2007adaptive}
{\sc A.~Kurganov, G.~Petrova, and B.~Popov}, {\em Adaptive semidiscrete
  central-upwind schemes for nonconvex hyperbolic conservation laws}, SIAM J.
  Sci. Comput., 29 (2007), pp.~2381--2401.

\bibitem{kurganov2000new}
{\sc A.~Kurganov and E.~Tadmor}, {\em New high-resolution central schemes for
  nonlinear conservation laws and convection-diffusion equations}, J. Comput.
  Phys., 160 (2000), pp.~241--282.

\bibitem{li2012arbitrary}
{\sc F.~Li and L.~Xu}, {\em Arbitrary order exactly divergence-free central
  discontinuous {G}alerkin methods for ideal {MHD} equations}, J. Comput.
  Phys., 231 (2012), pp.~2655--2675.

\bibitem{li2011central}
{\sc F.~Li, L.~Xu, and S.~Yakovlev}, {\em Central discontinuous {G}alerkin
  methods for ideal {MHD} equations with the exactly divergence-free magnetic
  field}, J. Comput. Phys., 230 (2011), pp.~4828--4847.

\bibitem{li2010central}
{\sc F.~Li and S.~Yakovlev}, {\em A central discontinuous {G}alerkin method for
  {H}amilton--{J}acobi equations}, J. Sci. Comput., 45 (2010), pp.~404--428.

\bibitem{li2016maximum}
{\sc M.~Li, F.~Li, Z.~Li, and L.~Xu}, {\em Maximum-principle-satisfying and
  positivity-preserving high order central discontinuous {G}alerkin methods for
  hyperbolic conservation laws}, SIAM J. Sci. Comput., 38 (2016),
  pp.~A3720--A3740.

\bibitem{LiuWu2024}
{\sc M.~Liu and K.~Wu}, {\em Structure-preserving oscillation-eliminating
  discontinuous {G}alerkin schemes for ideal {MHD} equations: Locally
  divergence-free and positivity-preserving}, J. Comput. Phys.,  (2025),
  p.~113795.

\bibitem{liu2004central}
{\sc Y.~Liu}, {\em Central schemes and central discontinuous {G}alerkin methods
  on overlapping cells}, in Conference on Analysis, Modeling and Computation of
  PDE and Multiphase Flow, Stony Brook, NY, 2004.

\bibitem{liu2022essentially}
{\sc Y.~Liu, J.~Lu, and C.-W. Shu}, {\em An essentially oscillation-free
  discontinuous {G}alerkin method for hyperbolic systems}, SIAM J. Sci.
  Comput., 44 (2022), pp.~A230--A259.

\bibitem{liu2022oscillation}
{\sc Y.~Liu, J.~Lu, Q.~Tao, and Y.~Xia}, {\em An oscillation-free discontinuous
  {G}alerkin method for shallow water equations}, J. Sci. Comput., 92 (2022),
  p.~109.

\bibitem{liu2007central}
{\sc Y.~Liu, C.-W. Shu, E.~Tadmor, and M.~Zhang}, {\em Central discontinuous
  {{G}alerkin} methods on overlapping cells with a nonoscillatory hierarchical
  reconstruction}, SIAM J. Numer. Anal., 45 (2007), pp.~2442--2467.

\bibitem{liu2008stability}
{\sc Y.~Liu, C.-W. Shu, E.~Tadmor, and M.~Zhang}, {\em ${L}^2$ {stability}
  {analysis} of the {central} {discontinuous} {{G}alerkin} {method} and a
  {comparison} between the {central} and {regular} {discontinuous} {{G}alerkin}
  {methods}}, ESAIM Math. Model. Numer. Anal., 42 (2008), pp.~593--607.

\bibitem{liu2018optimal}
{\sc Y.~Liu, C.-W. Shu, and M.~Zhang}, {\em Optimal error estimates of the
  semidiscrete central discontinuous {G}alerkin methods for linear hyperbolic
  equations}, SIAM J. Numer. Anal., 56 (2018), pp.~520--541.

\bibitem{lu2021oscillation}
{\sc J.~Lu, Y.~Liu, and C.-W. Shu}, {\em An oscillation-free discontinuous
  {G}alerkin method for scalar hyperbolic conservation laws}, SIAM J. Numer.
  Anal., 59 (2021), pp.~1299--1324.

\bibitem{lu2008explicit}
{\sc Y.~Lu, S.~Wong, M.~Zhang, C.-W. Shu, and W.~Chen}, {\em Explicit
  construction of entropy solutions for the {L}ighthill--{W}hitham--{R}ichards
  traffic flow model with a piecewise quadratic flow-density relationship},
  Transportation Research Part B: Methodological, 42 (2008), pp.~355--372.

\bibitem{nessyahu1990non}
{\sc H.~Nessyahu and E.~Tadmor}, {\em Non-oscillatory central differencing for
  hyperbolic conservation laws}, J. Comput. Phys., 87 (1990), pp.~408--463.

\bibitem{peng2023oedg}
{\sc M.~Peng, Z.~Sun, and K.~Wu}, {\em {OEDG: Oscillation-eliminating
  discontinuous Galerkin method for hyperbolic conservation laws}}, Math.
  Comput., in press (2024).
\newblock
  \href{https://dx.doi.org/10.1090/mcom/3998}{https://dx.doi.org/10.1090/mcom/3998},
  arXiv preprint, arXiv:2310.04807, 2023.

\bibitem{qiu2005runge}
{\sc J.~Qiu and C.-W. Shu}, {\em {R}unge--{K}utta discontinuous {G}alerkin
  method using {WENO} limiters}, SIAM J. Sci. Comput., 26 (2005), pp.~907--929.

\bibitem{sun2019strong}
{\sc Z.~Sun and C.-w. Shu}, {\em Strong stability of explicit {R}unge--{K}utta
  time discretizations}, SIAM J. Numer. Anal., 57 (2019), pp.~1158--1182.

\bibitem{tao2023oscillation}
{\sc Q.~Tao, Y.~Liu, Y.~Jiang, and J.~Lu}, {\em An oscillation free local
  discontinuous {G}alerkin method for nonlinear degenerate parabolic
  equations}, Numer. Methods Partial Differ. Equ., 39 (2023), pp.~3145--3169.

\bibitem{xu20192}
{\sc Y.~Xu, Q.~Zhang, C.-w. Shu, and H.~Wang}, {\em The ${L}^2$-norm stability
  analysis of {R}unge--{K}utta discontinuous {G}alerkin methods for linear
  hyperbolic equations}, SIAM J. Numer. Anal., 57 (2019), pp.~1574--1601.

\bibitem{yakovlev2013locally}
{\sc S.~Yakovlev, L.~Xu, and F.~Li}, {\em Locally divergence-free central
  discontinuous {G}alerkin methods for ideal {MHD} equations}, J. Comput. Sci.,
  4 (2013), pp.~80--91.

\bibitem{YanAbgrallWu}
{\sc R.~Yan, R.~Abgrall, and K.~Wu}, {\em Uniformly high-order bound-preserving
  {OEDG} schemes for two-phase flows}, Math. Models Methods Appl. Sci., 34
  (2024), pp.~2537--2610.

\bibitem{yu2020study}
{\sc J.~Yu and J.~S. Hesthaven}, {\em A study of several artificial viscosity
  models within the discontinuous {G}alerkin framework}, Commun. Comput. Phys.,
  27 (2020), pp.~1309--1343.

\bibitem{zhang2010positivity}
{\sc X.~Zhang and C.-W. Shu}, {\em On positivity-preserving high order
  discontinuous {G}alerkin schemes for compressible {E}uler equations on
  rectangular meshes}, J. Comput. Phys., 229 (2010), pp.~8918--8934.

\bibitem{zhao2017runge}
{\sc J.~Zhao and H.~Tang}, {\em {R}unge--{K}utta central discontinuous
  {G}alerkin methods for the special relativistic hydrodynamics}, Commun.
  Comput. Phys., 22 (2017), pp.~643--682.

\bibitem{zhao2017runge2}
{\sc J.~Zhao and H.~Tang}, {\em {R}unge--{K}utta discontinuous {G}alerkin
  methods for the special relativistic magnetohydrodynamics}, J. Comput. Phys.,
  343 (2017), pp.~33--72.

\bibitem{zhong2013simple}
{\sc X.~Zhong and C.-W. Shu}, {\em A simple weighted essentially nonoscillatory
  limiter for {R}unge--{K}utta discontinuous {G}alerkin methods}, J. Comput.
  Phys., 232 (2013), pp.~397--415.

\bibitem{zhu2017numerical}
{\sc J.~Zhu and C.-W. Shu}, {\em Numerical study on the convergence to steady
  state solutions of a new class of high order {WENO} schemes}, J. Comput.
  Phys., 349 (2017), pp.~80--96.

\bibitem{zingan2013implementation}
{\sc V.~Zingan, J.-L. Guermond, J.~Morel, and B.~Popov}, {\em Implementation of
  the entropy viscosity method with the discontinuous {G}alerkin method},
  Comput. Methods Appl. Mech. Engrg., 253 (2013), pp.~479--490.

\end{thebibliography}
	    \bibliographystyle{siamplain}
	\end{changemargin}
\end{document}